\documentclass[reqno]{imsart}
\usepackage{color}
\usepackage{comment}
\usepackage{pdfpages}
\usepackage{graphicx}
\usepackage{dcolumn}
\usepackage{dsfont}

\RequirePackage[numbers]{natbib}
\RequirePackage[colorlinks=true, pdfstartview=FitV, linkcolor=blue,
  citecolor=blue, urlcolor=blue]{hyperref}
\usepackage{paralist}

\usepackage{tikz} 
\usepackage{dcolumn}
\usetikzlibrary{arrows,positioning}
\tikzset{
    >=stealth',
    pil/.style={
           ->,
           thick,
           shorten <=2pt,
           shorten >=2pt,}
}

\usepackage{graphicx}
\graphicspath{{./figures/}}
\usepackage{amsfonts}
\usepackage{amsmath}
\usepackage{amsthm}
\usepackage{amssymb}
\usepackage{amsbsy}
\usepackage{epsfig}
\usepackage{fullpage}
\usepackage{natbib, mathrsfs}
\usepackage{verbatim}
\usepackage[latin1]{inputenc}
\usepackage{mhequ}
\usepackage{algorithm}
\usepackage{algorithmic}

\numberwithin{equation}{section}

\def \config{\Omega_{\Lambda}}
\def \pconfig{\Omega_{\mathcal{B}(\Lambda)}}

\def \conc{\bullet}
\def \S{\mathcal{S}}

\newcommand \num[2]{\,\#_{#1}{\left(#2\right)}}

\newcommand \init[1]{#1^{(\mathrm{init})}}
\newcommand \fin[1]{#1^{(\mathrm{fin})}}
\newcommand \occ[1]{\mathfrak{N}_{#1}}
\newcommand \mass[1]{\tcr{[ERROR - NOT NEEDED]}}

\newcommand \flowna[1]{F^{(\mathrm{na})}_{#1}}
\newcommand \flowfull[1]{F_{#1}}
\newcommand \flowtrunc[2]{F^{(#2)}_{#1}}
\newcommand \flowtot[1]{F^{(\mathrm{pt})}_{#1}}

\def \cnaive{\beta}
\def \cultra{\beta^{-1}}

\newcommand \res[2]{#2|_{#1}}

\def \rect{\mathbf{R}}
\def \edges{\mathbf{E}}

\def \Id{\mathbf{I}}
\def \type{\mathbf{TYPE}}

\newcommand \naiveanc[1]{G^{(\mathrm{na})}(#1)}
\newcommand \longanc[2]{G^{(\mathrm{long})}(#1,#2)}

\newcommand \allext[2]{\mathfrak{T}_{#1}(#2)}
\newcommand \allallext[1]{\mathfrak{T}_{#1}}

\newcommand \recrit[1]{\mathcal{R}_{n}}

\def \be{\begin{equs}}
\def \ee{\end{equs}}
\def \P{\mathbb{P}}
\def \E{\mathbb{E}}

\newcommand \tcr[1]{\textcolor{red}{(#1)}}

\def \G{\mathcal{G}}

\def \t{\theta}
\def \ut{\underline{\theta}}

\def \L{\mathrm{L}}

\def \d{\mathrm{d}}
\newcommand \D[1]{\Delta_{t}(#1)}

\def \m{n}

 \let\b=\beta   \let\d=\delta  
 \let\g=\gamma \let\h=\eta      
         \let\p=\pi  
  \let\s=\sigma \let\t=\tau   
  \let\z=\zeta
\let\D=\Delta   \let\G=\Gamma  \let\L=\Lambda 
\let\O=\Omega

\renewcommand{\leq}{\;\leqslant\;}                   
\renewcommand{\geq}{\;\geqslant\;}                   
\DeclareMathOperator*{\union}{\bigcup}       
\newcommand{\inftwo}[2]{\inf_{\substack{#1 \\ #2}}} 
\renewcommand{\b}{\beta}

\newcommand{\1}{\mathds{1}}

\newcommand{\var}{\operatorname{Var}}

\newcommand{\tmix}{T_{\rm mix}}
\newcommand{\trel}{T_{\rm rel}}

\newcommand{\gap}{{\rm gap}}

\newcommand{\vmax}{v_{\mathrm{max}}}
\newcommand{\umax}{u_{\mathrm{max}}}

\newcommand{\N}{\mathbb N}

\newcommand{\cA}{\ensuremath{\mathcal A}}
\newcommand{\cB}{\ensuremath{\mathcal B}}
\newcommand{\cC}{\ensuremath{\mathcal C}}
\newcommand{\cD}{\ensuremath{\mathcal D}}

\newcommand{\cG}{\ensuremath{\mathcal G}}

\newcommand{\cL}{\ensuremath{\mathcal L}}

\newcommand{\cQ}{\ensuremath{\mathcal Q}}

\newcommand{\cS}{\ensuremath{\mathcal S}}

\newcommand{\cV}{\ensuremath{\mathcal V}}

\newcommand{\bbR}{{\ensuremath{\mathbb R}} }

\newcommand{\bbZ}{{\ensuremath{\mathbb Z}} }
\newcommand{\Z}{{\ensuremath{\mathbb Z}} }

\newtheorem{theorem}{Theorem}[section]
\newtheorem{lemma}[theorem]{Lemma}

\newtheorem{prop}[theorem]{Proposition}

\theoremstyle{plain}
\newtheorem{thm}[theorem]{Theorem}
\newtheorem*{thm-non}{Theorem}
\newtheorem{cor}[theorem]{Corollary}

\theoremstyle{definition}
\newtheorem{defn}[theorem]{Definition}
\newtheorem{remark}[theorem]{Remark}

\begin{document}

\begin{frontmatter}
\title{Mixing of the Square Plaquette Model on a Critical Length Scale}
\runtitle{Mixing of the Square Plaquette Model}


\begin{aug}
\author{\fnms{Paul} \snm{Chleboun}\ead[label=e1]{paul.i.chleboun@warwick.ac.uk}}
\and
\author{\fnms{Aaron} \snm{Smith}\thanksref{t2}\ead[label=e2]{asmi28@uottawa.ca}}
\thankstext{t2}{Supported by a grant from NSERC}
\runauthor{Chleboun and Smith}
\affiliation{
 University of Warwick and
 University of Ottawa\thanksmark{m2}}
\address{
Department of Statistics,\\
 University of Warwick,  \\
Coventry,  CV4 7AL \\  United Kingdom \\ \printead{e1}}
\address{
Department of Mathematics and Statistics\\
University of Ottawa \\
 585 King Edward Avenue,  Ottawa\\
ON K1N 7N5 \\ Canada  \\\printead{e2}}
\end{aug}

\maketitle




\begin{abstract}
Plaquette models are short range ferromagnetic spin models that play a key role in the dynamic facilitation approach to the liquid glass transition. In this paper we study the dynamics of the square plaquette model at the smallest of the three critical length scales discovered in \cite{Chleboun2017}.
Our main results are estimates of the spectral gap and mixing time for two natural boundary conditions.
As a consequence, we observe that these time scales depend heavily on the boundary condition in this scaling regime.
\end{abstract}

\end{frontmatter}


\setcounter{tocdepth}{1}
\tableofcontents



\section{Introduction} \label{sec:intro}

In this paper we consider the dynamics of the square plaquette model (SPM) at low temperature.
Spin plaquette models were originally associated with glassy behaviour in \cite{Garrahan2002b,Newman1999}, where it was argued that they have more physically-realistic dynamics and thermodynamic properties than kinetically-constrained models, while remaining mathematically tractable.
These models have also recently been generalised to quantum systems, called fracton models, which show extremely similar relaxation behaviour \cite{Chamon2005,Nandkishore2018}.
Plaquette models are defined over an integer lattice $\Lambda \subseteq \Z^{d}$, and configurations of the plaquette model correspond to $\pm 1$-valued labellings of the lattice $\Lambda$. Every configuration of the SPM has an associated energy given by certain short range ferromagnetic interactions that are defined in terms of the  \textit{plaquettes}, a collection of subsets of $\Z^{2}$ denoted by $\mathcal{P}$.

More formally, plaquette models are families of probability distributions with state space $\{-1,+1\}^{\Lambda}$. Every configuration $\sigma \in \{-1,+1\}^{\Lambda}$ has an associated energy value, given by the Hamiltonian \[H(\sigma) = -\frac{1}{2}\sum_{P\in \mathcal{P}}\prod_{x \in P}\sigma_x\,.\] For any fixed inverse-temperature $\beta > 0$, we then associate to this Hamiltonian the probability distribution $\pi_{\Lambda}^\beta$ on $\{-1,1\}^{\Lambda}$ given by $\pi_{\Lambda}^\beta(\sigma) \propto \textrm{exp}(-\beta H_\Lambda(\sigma))$.  In the case of the SPM, we always take $\Lambda \subset \Z^{2}$ and the plaquettes are exactly the collection of unit squares contained in $\Lambda$. There are several natural Markov chains associated with these probability distributions, many of which have very similar behaviour. In this paper we study the most popular of these Markov chains, the \textit{continuous-time single-spin Glauber dynamics} (also known as the \textit{Gibbs sampler}). Roughly speaking, this Markov chain evolves from a configuration $\sigma$ by choosing a random site $x \in \Lambda$ at unit rate and then updating the value $\sigma(x)$ conditional on $\{\sigma(y)\}_{y \in \Lambda \backslash \{x\}}$ according to the measure $\pi_{\Lambda}^{\beta}$.

Despite the relatively simple form of the Hamiltonian, the thermodynamics of these measures is non trivial \cite{Chleboun2017}. Although there is no phase transition in the SPM (i.e. there is a unique infinite volume Gibbs measure for each $\beta$), static correlation lengths grow extremely quickly as the temperature tends to zero ($\beta \to \infty$). Furthermore, the ground states (configurations $\sigma$ that minimise $H(\sigma)$) depend on the boundary conditions, and are often highly degenerate. It turns out that this plays an important role for the dynamics of the process. In this paper we consider the dynamics of the SPM in the low temperature regime, i.e. as $\beta \to \infty$, on the smallest of the critical length scales in \cite{Chleboun2017}.

Plaquette spin models, under single spin-flip Glauber dynamics, have recently attracted a great deal of attention in the physics literature in the context of glassy materials (see e.g. \cite{Jack2016} and references therein). Understanding the liquid-glass transition and the dynamics of amorphous materials remains a significant challenge in condensed matter physics (for a review see \cite{Berthier2011a}). One particularly successful approach to studying such systems, known as \emph{dynamic facilitation}, supposes that local relaxation events facilitate further relaxation events in neighbouring regions, but in the absence of such events the system is locally unable to change state (transitions are \textit{blocked}). This idea led to the introduction of a class of interacting particle systems called kinetically constrained models (KCMs), which feature trivial stationary measures but complicated dynamics.  These models display many of the key features of glassy systems, such as aging \cite{Faggionato2011} and dynamical heterogeneity \cite{Chleboun2012}, and have been extremely well studied in both the mathematical and physics literature (see \cite{Garrahan2010} for a review). There are two crucial difficulties in justifying KCMs as models for the liquid glass transition: it is not clear how kinetic constraints emerge from the microscopic dynamics of many-body systems, and, since KCMs have trivial thermodynamics, KCMs cannot account for the growth of static correlations. It turns out that plaquette spin models address both these issues, and in particular their dynamics are effectively constrained \cite{Garrahan2002b,Newman1999}.

The dynamics of plaquette spin models has been the focus of several works in the physics literature. Initial simulations clearly indicate the occurrence of glassy dynamics (extremely slow relaxation) at low temperature \cite{Jack2005a,Newman1999}.   The products $\prod_{x \in P} \sigma_{x}$ of the spins over the individual plaquettes $P \in \mathcal{P}$ play an important role in studying these dynamics. These products  are called the \emph{plaquette variables}, and a plaquette variable equal to $-1$ is said to be a \emph{defect}. In the SPM, we note that the plaquettes can be naturally identified with their lower-left corner, and so the plaquette variables also form a spin system on an integer lattice.

At low temperature, it is natural to describe the dynamics of the SPM via the locations of the defects, which are effectively constrained. Indeed, flipping a single spin will flip the value of all the plaquette variables whose plaquettes contain the corresponding site - in the SPM, this means flipping the four adjoining plaquette variables. If the plaquette variables associated to the four plaquettes containing a given spin are currently all $+1$ (i.e. there are no defects) then the spin flips at an extremely slow rate of $e^{-4\beta}$; if there is currently one defect associated with the spin and the other plaquette variables are positive then the spin flips with rate $e^{-2\beta}$; finally, if there are two or more defects associated with a spin, then it flips with rate $1$ (see Fig. \ref{FigPairMove}).
For this reason, defects are infrequently created, and isolated defects move extremely slowly, while paired defects can move quickly.

Simulations, and heuristic analysis based on this observation, suggest that for the SPM the relaxation time scales like $e^{c \beta}$ (Arrhenius scaling), and that the dynamics are closely related to those of the Fredrickson-Andersen KCM for which mixing properties have been well studied (see \cite{Pillai2017,Pillai2017a,Blondel2012a} and references therein).
On the other hand, in a related model called the triangular plaquette model the relaxation time is expected to scale like $e^{c \beta^2}$ (super-Arrhenius scaling) \cite{Garrahan2002b}.
These dynamics are closely related to a particularly KCM known as the East model which has been widely studied (see \cite{Chleboun2015,Faggionato2012,Ganguly2015} and references therein).
This difference between Arrhenius and super-Arrhenius scaling is fundamental due to the nature of the energy barriers that should be overcome to bring isolated defects together and annihilate them.
Despite their importance, as far as we are aware this work represents the first rigorous results related to the dynamics of plaquette spin models.

Our main results are on the dynamics of the SPM in boxes $\Lambda = \{1,2,\ldots,L\}^{2}$ with side length $L$ given by what physicists call the \textit{critical length scale} -  the correlation length for the product of spin variables in the infinite volume Gibbs measure. In \cite{Chleboun2017}, this critical length scale was shown to satisfy $L \approx e^{\frac{\beta}{2}}$ as $\beta \rightarrow 0$. Our results show that the relaxation time (inverse spectral gap) and total variation mixing times indeed have Arrhenius scaling on this critical length scale. We also show that, on this length scale, the relaxation time has a dramatic dependence on the boundary conditions, a phenomena that has been previously observed for certain kinetically constrained models \cite{Chleboun2014}.

Our first result is that, for the SPM in boxes of side length $L \approx e^{\frac{\beta}{2}}$ with all plus boundary conditions, the relaxation time scales as $e^{3.5\beta}$ up to polynomial factors in $\beta$ as $\beta \to \infty$. Furthermore, the total variation mixing time in this case is between $e^{3.5\beta}$ and $e^{4\beta}$ (See Theorem \ref{ThmMainResPlus}).
Our second main result is on the same length scale for periodic boundary conditions. We find that with periodic boundary conditions on the critical length scale the relaxation time and total variation mixing time both scale as $e^{4\beta}$ up to polynomial corrections in $\beta$ (see Theorem \ref{ThmMainResPer}).

It turns out that these different time scales are caused by the structure of the ground states. In particular, with all plus boundary conditions there is a unique ground state (all sites have spin $+1$), and the low temperature dynamics are dominated by the time to reach the ground state.
On the other hand with periodic boundary conditions there are $2^{2L-1}$ ground states, where $L$ is the side length of the box, which correspond to flipping all the spins in any set of rows and columns with respect to the all plus state.
In this case the dynamics are dominated by an induced random walk on the ground states.
It is possible to construct boundary conditions such that there are a small number of  very well separated ground states (see Remark \ref{Rem:bad}).
In this case we conjecture that the relaxation time is at least $e^{4.5\beta}$.
In a companion paper \cite{PlaquetteBigSquareLattice}, we give Arrhenius bounds on the mixing time and spectral gap for all boundary conditions for length $L \gg e^{\frac{\beta}{2}}$.

The main tools used in the proof of the upper bounds will be detailed canonical path bounds using multi-commodity flows \cite{Sinclair1992}, combined with the spectral profile method introduced in \cite{Goel2006}.
Rather surprisingly, it turns out that a lot of effort is required to construct flows which do not have very bad congestion on edges connecting certain low-probability configurations with many defects. The main difficulty, informally sketched in Section \ref{SecHeuristic} and near the beginning of Section \ref{SubsecPathLowDef}, is common to plaquette models and other models for which mixing is greatly facilitated by the exchange of small, short-lived, and rapidly-moving configurations of interacting particles (see Figure \ref{FigPairMove}).
In particular we believe that much of the effort in constructing the multi-commodity flows will be useful also for the study of other plaquette models.
To prove the upper bounds in the case of periodic boundary conditions we compare the trace of the process on the ground states with the simple random walk on the hypercube, and use several times the result that mixing times are related to the hitting times of large sets \cite{Oliveira2012,Peres2015}.

\subsection{Guide to Paper}

We begin by setting basic notation for the paper in Section \ref{SecNotation}. This section also includes a heuristic description of the dynamics of the SPM near stationarity. This heuristic guides our proof strategy and also suggests the final results, and we suggest that readers fully digest this heuristic before reading the more precise proofs.  Section \ref{SecMainResults} includes a precise statement of our main results.  Section \ref{SecPrelimResults} gives some important basic results that will be used frequently throughout the paper, including a description of all possible plaquette configurations and a concentration result for the number of defects under the stationary measure. Section \ref{SecConCanPath} is the bulk of the paper. It describes the canonical path method, derives some specific forms of these bounds that will be used in this paper, constructs two families of canonical paths that will be used to analyze the SPM, and gives detailed bounds on the properties of these paths. The start of this section includes a more detailed guide to its contents. Finally, Sections \ref{SecAllPlusRes} and \ref{SecPerBoundRes} contain the proofs of our main results: bounds on the mixing and relaxation time for the all-plus and periodic boundary conditions respectively. Both sections rely heavily on the canonical path arguments for their upper bounds, and use ad-hoc constructions of special test functions for their lower bounds.

\section{Notation and background} \label{SecNotation}

\subsection{Basic Conventions}

We denote the two canonical basis vectors of $\bbZ^2$ by $e_1 = (1,0)$ and $e_2= (0,1)$. For $x \in \bbZ^2$ we denote its projection on $e_1$ and $e_2$ by $x_1$ and $x_2$ respectively.
 We define the shorthand $[a\!:\!b] = \{a,a+1,\ldots,b\}$ when $b-a \in \mathbb{N}$.

Given $\L \subseteq \bbZ^2$ we will denote by $\config$ the state space of the plaquette model, given by $\{ -1, 1\}^{\L}$, endowed with the product topology.
We let $\O=\O_{\bbZ^2}$.
  Given  $A \subset \L \subseteq \bbZ^2$ and a configuration $\s \in \O_\L$ we define  $\res{A}{\s}$ as the restriction of $\s$ to $A$.
We define the  \emph{plaquette} at site $x \in \bbZ^2$ by the set of four sites $B_x = \{x,x+e_1,x+e_2,x+e_1+e_2\}$.
For short we also write $\overline{B}_x= B_{x-e_1-e_2} = \{x,x-e_1,x-e_2,x-e_1-e_2\}$.

For functions $f, g \, : \, \mathbb{R}^{+} \mapsto \mathbb{R}^{+}$ we write $f = O(g)$ if there exists $0 < C,X < \infty$ so that $f(x) \leq C \, g(x)$ for all $x > X$. We also write $f = o(g)$ if $\lim_{x \rightarrow \infty} \frac{f(x)}{g(x)} = 0$, and we write $f = \Omega(g)$ if $g = O(f)$. Finally, we write $f = \Theta(g)$ if both $f = O(g)$ and $g = O(f)$. To save space, we also write $f \lesssim g$ for $f = O(g)$, we write $f \gtrsim g$ for $f = \Omega(g)$, and we write $f \asymp g$ for $f = \Theta(g)$. Similarly, to save space, all inequalities should be understood to hold only for all $\beta > \beta_{0}$ sufficiently large. For example, we may write $e^{\beta} \geq \beta + 25$ without additional comment. Since all of our results are asymptotic as $\beta$ goes to infinity, this convention will not cause any difficulties.
 For any function $f:A\to B$, we denote  the image of $f$ by $f(A)$.

Finally, we define two orders on $\mathbb{Z}^{2}$. 
For $x \neq y \in \mathbb{Z}^{2}$, we say that $x$ is less than $y$ in \textit{lexicographic order}\footnote{There are several different ``lexicographic" orders in the literature. The order in this paper corresponds to the order in which words are read in English if the Cartesian plane is drawn in the usual way.} if and only if one of the two following  conditions hold:
\begin{enumerate}
\label{eq:lexi}
\item $x_{2} > y_{2}$, or
\item $x_{2} = y_{2}$ and $x_{1} < y_{1}$.
\end{enumerate}
Similarly, we say that $x$ is less than $y$ in \textit{anti-lexicographic order} if and only if one of the two following  conditions hold:
\begin{enumerate}
\item $x_{2} > y_{2}$, or
\item $x_{2} = y_{2}$ and $x_{1} > y_{1}$.
\end{enumerate}
By a small abuse of notation, we say that a set $S_{1}$ is less than a set $S_{2}$ in lexicographic order if \textit{every} element of $S_{1}$ is less than \textit{every} element of $S_{2}$.

\subsection{Equilibrium Gibbs measures}
We will define the finite volume Gibbs measures on $\L  \subset \bbZ^2$ with  fixed and periodic boundary conditions.
Let $\cB(\L) = \{x\in \bbZ^2 \,:\, B_x  \cap \L \neq \emptyset\}$ be the set of plaquettes which intersect $\Lambda$, indexed by their bottom left vertex, and
let $\cB_{-}(\L) = \{x \in \bbZ^2 \, : \, \overline{B}_{x} \subset \L\}$ (if $\L$ is a rectangle, then $\cB_{-}(\L)$ is just $\L$ without the left most column and bottom row).
For a boundary condition $\t \in \O$ we will denote by $\O_\L^{\t} = \{\s\in \O \,:\, \res{\L^c}{\s} \equiv \res{\L^c}{\t}\}$. Finally we denote the external boundary of $\L$ by  $\partial(\L) = \cup_{x\in\cB(\L)}B_x\setminus  \L$.

For fixed boundary conditions $\t$, the plaquette variables associated with a spin configuration are defined by the map $p^{\t}:\O_\L^{\t}\to \O_{\cB(\L)}$ which is given by the formula
\begin{align}
  \label{EqDefectMap}
  p^{\t}_x(\s) = \prod_{y \in B_x}\s_y= \s_{(x_1,x_2)}\s_{(x_1+1,x_2)} \s_{(x_1,x_2+1)}\s_{(x_1+1,x_2+1)}\,, \quad \textrm{for } x \in \cB(\L)\,.
\end{align}
Similarly,  for periodic boundary conditions on a box $\L=[0\!:\!L_1-1]\times [0\!:\!L_2-1]$, define $p^{\rm per}:\O_\L \to \O_{\L}$ by
\begin{align}
  p^{\rm per}_x(\s) = \prod_{y \in B_x}\s_y= \s_{(x_1,x_2)}\s_{(x_1+1,x_2)} \s_{(x_1,x_2+1)}\s_{(x_1+1,x_2+1)}\,, \quad \textrm{for } x \in \L\,,
\end{align}
where the sums $x_1{+}1$ and $x_2{+}1$ above are taken modulo $L_1$ and $L_2$ respectively.
We say there is a \emph{defect} in $\s$ at $x \in \cB({\L})$ if $p^{\t}_x(\s) = -1$ (similarly for periodic boundary conditions). By a small abuse of notation, we consider ``per" to be a boundary condition.

For $\s \in \O_\L$ or $\O_\L^\t$ let $|\s| = |\{x \in \L \,:\, \s_x = -1\}|$ denote the number of minus spins and $|p^{\t}(\s)| =  |\{x \in \cB(\L)\,:\, p^{\t}_x(\s) = -1\}|$ the number of defects (similarly for $|p^{\rm per}(\s)|$).
We define a partial order on plaquette variables with respect to defects by
\begin{align}
  p^\t(\s) \leq p^\t(\h) \iff \{x \in \cB(\L)\,:\, p^{\t}_x(\s) = -1\} \subset \{x \in \cB(\L)\,:\, p^{\t}_x(\h) = -1\}\,,
\end{align}
similarly for periodic boundary conditions.

We define the Hamiltonian $H_\L^{\t}: \O_{\L}^{\t}\to \bbR$ with boundary condition $\t$ by
\begin{align}
\label{eq:hamtau}
  H^{\t}_\L(\s) = -\frac{1}{2}\sum_{x\in \cB(\L)}p^{\t}_x(\s)\,,
\end{align}
and similarly for periodic boundary condition. The finite volume Gibbs measure on $\L$ with boundary condition $\t$ is then denoted by $\p_\L^{\t}$ and given by
\begin{align}
 \label{eq:pitau}
  \pi_{\L}^{\t}(\s) = \frac{e^{-\b H^{\t}_\L(\s)}}{Z^{\t}_\L(\b)}\,,
\end{align}
where $Z^{\t}_\L(\b) =\sum_{\s\in\O_\L^{\t}} e^{-\b H^{\t}_\L(\s)}$ is the partition function. The analogous formula gives the finite volume Gibbs measure $\pi_{\L}^{\rm per}$ for periodic boundary conditions. For brevity, if $\t \equiv \pm 1$ we will replace $\t$ with $\pm$; for example we write $H_\L^+$, $\pi^+_\L$ for plus boundary conditions.
Also, where there is no confusion, we denote by $+\in \O$ the configuration of all $+1$ spins.
When the boundary conditions and lattice are clear from the context, we may drop the boundary condition superscript and the lattice subscript.

It follows from (\ref{eq:hamtau}), (\ref{eq:pitau}) and  Lemmas \ref{lem:parity} and \ref{lem:parityPer} in Section \ref{SecPrelimResults}, that the measure induced by $\pi^{\t}_\L$ on the plaquette variables satisfy
\begin{align} \label{eq:defectprob}
\pi^{\t}_\L(p^{\t}(\s) = p) \propto e^{-\beta |p|}\1_{p^\tau(\O_\L^\t)}(p),\     \textrm{ for } p \in \O_{\cB(\L)}\,.
\end{align}
Also, since $|p^{+}(+)| = |p^{\rm{per}}(+)| = 0$, we will frequently use
\begin{align}
  \label{eq:simple}
  \pi_\L^{\t}(\s) =  \pi_\L^{\t}(+) e^{-\b |p^{\t}(\s)|}\,,\quad \textrm{ for } \s \in \O_{\L}^{\t},\quad \textrm{ and } \tau \in \{+,\rm{per}\}\,.
\end{align}

\subsection{Finite volume Glauber dynamics}

For a set $S \subset \L$ and $\s \in \config$, denote by $\s^{S}$ the configuration obtained by flipping all the spins of $\s$ that lie in $S$. With slight abuse of notation, we define $\s^{x} = \s^{\{x\}}$ for $x \in \L$.

Given a finite region $\L$ and boundary condition $\tau$, we consider the continuous time Markov process determined by the generator
\begin{align}
\label{eq:gen}
  \cL_{\L}^{\t} f(\sigma)=\sum_{x\in\L}c^{\tau}_{\L}(x,\sigma)(f(\s^x)-f(\s)) = \sum_{x\in\L}c^{\tau}_{\L}(x,\sigma)\nabla_xf(\s)\,,
\end{align}
where we define $\nabla_xf(\s) = (f(\s^x)-f(\s))$, and where the Metropolis spin-flip rates $c^{\tau}_\L(x,\sigma)$ are given by the formula
\begin{equation}
\label{eq:rates}
c^{\tau}_\L(x,\sigma)=
\begin{cases}
e^{-\beta(H_{\L}^{\t}(\sigma^x)-H_{\L}^{\t}(\sigma))} & \text{if } H_{\L}^{\t}(\sigma^x)>H^{\t}_{\L}(\sigma)\,,\\
1 & \text{otherwise.}
\end{cases}
\end{equation}
With a slight abuse of notation, we denote  the elements of the associated transition rate matrix by by $\cL_{\L}^{\t}(\s,\h)$, for $\s,\h \in \O_\L^\t$.
The process is  reversible with respect to the finite volume equilibrium measure $\pi_{\L}^{\t}$.
\begin{remark}
All our results hold equally well for the standard heat-bath dynamics, since $\left(1+e^{\beta(H_{\L}^{\t}(\sigma^x)-H_{\L}^{\t}(\sigma))} \right)^{-1} \asymp \min\{e^{-\beta(H_{\L}^{\t}(\sigma^x)-H_{\L}^{\t}(\sigma))},1\}$ for large $\beta$.
\end{remark}

Since $H^\t_\L(\s)$ only depends on the plaquette variables, the spin dynamics also induce a dynamics on these ``defect" variables which is Markov. The generator of the defect dynamics is given by
\begin{align}
  \label{eq:genplaq}
  \mathcal{Q}_\L^\t f(p) = \sum_{x\in \L} k^\t_\L(x,p)\left(f(p^{\overline{B}_{x}})- f(p)\right),
\end{align}
where we recall
\begin{align*}
  p^{\overline{B}_{x}}_z =
  \begin{cases}
    -p_z & \textrm{if } z \in \overline{B}_{x} = \{x-e_1-e_2,x-e_1,x-e_2,x \}\,,\\
    p_z & \textrm{otherwise.}
  \end{cases}
\end{align*}
From \eqref{eq:rates}, the transition rates for this process are given by
\begin{align}
  \label{eq:ratesp}
  k^\t_\L(x,p) = \min\left\{ {\rm exp}\left[\beta\left(|p|-|p^{\overline{B}_{x}}| \right)\right], 1\right\}\,.
\end{align}

\subsection{Measures of Mixing Rate}
We set notation for some common notions related to mixing rates. The Dirichlet form associated with $\cL_{\L}^{\t}$ is denoted by $\cD_\L^\t(f) = -\pi_{\L}^\t(f\cL_{\L}^{\t} f)$, and it satisfies the formula
\begin{align}
  \label{def:dir}
  \cD_\L^\t(f) = \frac{1}{2}\sum_{\h\in\O}\sum_{x\in \L}\pi_{\L}^\t(\h)c_\L^\tau(x,\h)\left(\nabla_x f(\h)\right)^2\,.
\end{align}
Define $\var_\L^\t(f)$ to be the variance of $f$ with respect to $\pi_\L^\t$.

\begin{defn}[Relaxation time]
The smallest positive eigenvalue of $-\cL^\t_\L$
is called the spectral gap and it is denoted by
$\gap(\cL_{\L}^\t)$. It satisfies the Rayleigh-Ritz variational principle
\begin{align}
\label{eq:gap}
\gap ( \cL^\t_\L):= \inftwo{f \,:\,\O_\L \mapsto \bbR}{f \text{ non constant} }  \frac{ \cD_\L^\t(f) }{\var_\L^\t(f) }\,.
\end{align}
The relaxation time $T^\t_{\rm rel} (\L)$ is defined as the inverse of the spectral gap:
\begin{equation}\label{rilasso}
T_{\rm rel}^\t  (\L)= \frac{1}{ \gap (\cL^\t_\L)}\,.
\end{equation}
If $\L = [1:L]^2$ we simply write $T_{\rm rel}^\t  (L)$.
\end{defn}

\begin{defn}[Total Variation Distance]
We denote the \textit{total variation} distance between two measures $\mu, \nu$ on a common $\sigma$-algebra $\mathcal{F}$ by:
\be 
\| \mu - \nu \|_{\rm TV} = \sup_{A \in \mathcal{F}} | \mu(A) - \nu(A)|.
\ee 
\end{defn}

\begin{defn}[Mixing time]
We denote by
\be
T_{\rm mix}^{\t}(\L) = \inf \{s > 0 \, : \, \max_{\sigma \in \config} \| e^{s \,\cL^{\t}_{\L}}(\sigma,\cdot) - \pi_{\L}^\t(\cdot) \|_{\rm TV} < 1/4 \}
\ee
the mixing time of $\cL^\t_\L$. If $\L = [1,L]^{2}$ we simply write $T_{\rm mix}^\t  (L)$.

\end{defn}

We define the critical scale as the correlation length for the product of spin variables in the infinite volume Gibbs measure, see \cite{Chleboun2017} for further details and other important length scales.

\begin{defn}[The critical scale]
We define the critical length scale by $L_c = \lfloor e^{\frac{\b}{2}} \rfloor$.
\end{defn}

\subsection{Heuristics for Mixing with + Boundary } \label{SecHeuristic}

Call a configuration \textit{metastable} if no plaquette has more than one defect, and \textit{unstable} otherwise. As the unstable states are short-lived, we concentrate on the transitions between ``nearby'' metastable states.

Due to parity constraints (see Lemma \ref{lem:parity}), the lowest-energy metastable configurations have exactly four defects, placed at the vertices of a rectangle.
Starting the SPM process in such a ``rectangular'' configuration, with height and width of the associated rectangle at least four, it is overwhelmingly likely that the next metastable configuration will be another ``rectangular'' configuration, with either height or width changed by exactly one. The typical intermediate dynamics between such metastable states are shown in Figure \ref{FigPairMove}.

\begin{figure}[htb]
\includegraphics[width=0.95\linewidth]{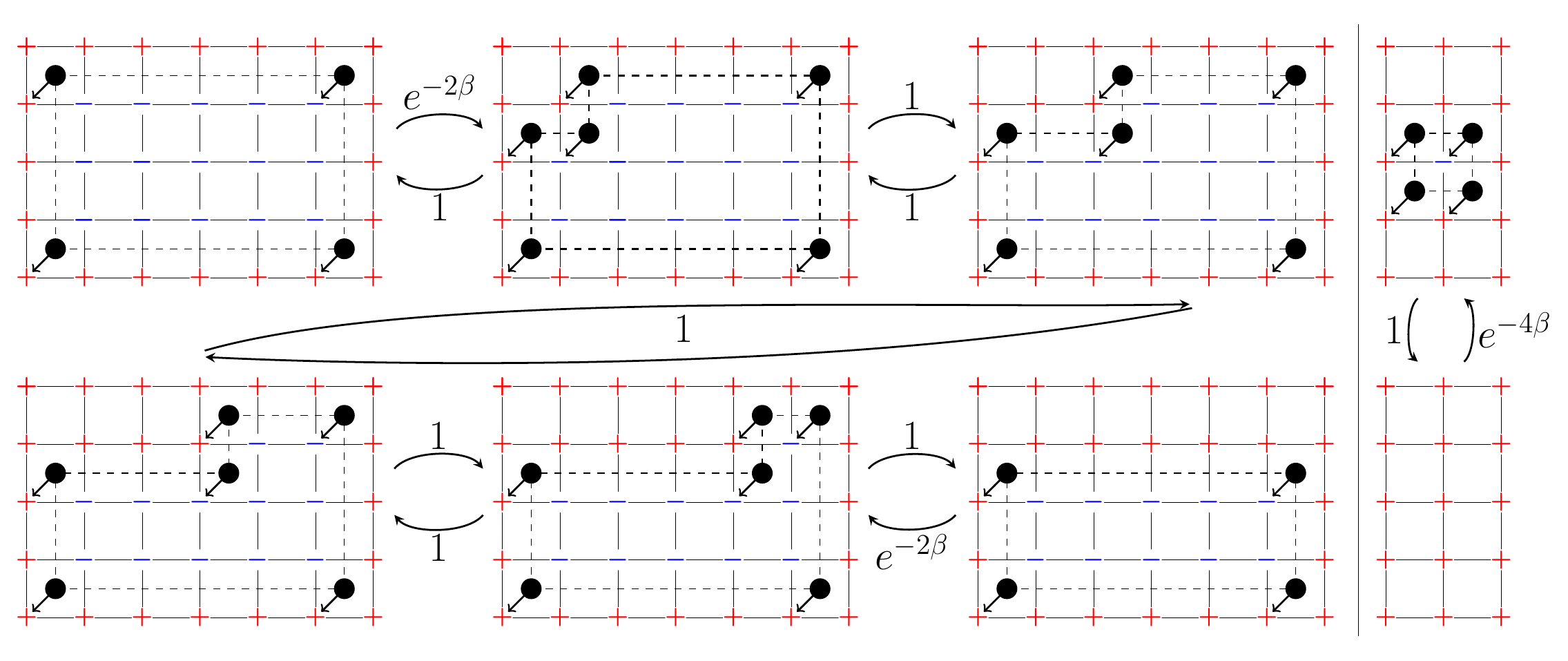}
\caption{\label{FigPairMove} The spin configuration is represented by $+$ (red) and $-$ (blue) on the lattice. The black circles represent defects ($-1$ plaquette variables), which are associated with the vertex at the bottom left of the corresponding plaquette. Dashed lines separate regions of $+$ and $-$ spins. The labeled arrows indicate the associated transition rates. An isolated defect creates two defects at rate $e^{-2\beta}$, subsequently a pair of defects is emitted and moves along an edge of the rectangle according to a simple random walk, and is then annihilated upon colliding with another defect (possibly the defect that emitted the pair). Right: the only type of transition not included on the left is to  add or remove four neighbouring defects. }
\end{figure}

Note that the ``rectangular'' configurations are entirely determined by the upper-left and lower-right vertices. Following the heuristic of Figure \ref{FigPairMove}, these corner defects perform nearly-independent simple random walks as shown in Figure \ref{FigCornerWalk}.
A routine SRW calculation says that this rectangle will collapse to the unique ground state in a time that scales like $\Theta \left( e^{3.5 \beta} \right)$.  Since typical configurations have few defects, it is natural to guess that the relaxation time of these simple configurations will be very close to the relaxation time of the full process. This turns out to be correct, and will guide our proof.

\begin{figure}[h]
\includegraphics[width=0.85\linewidth]{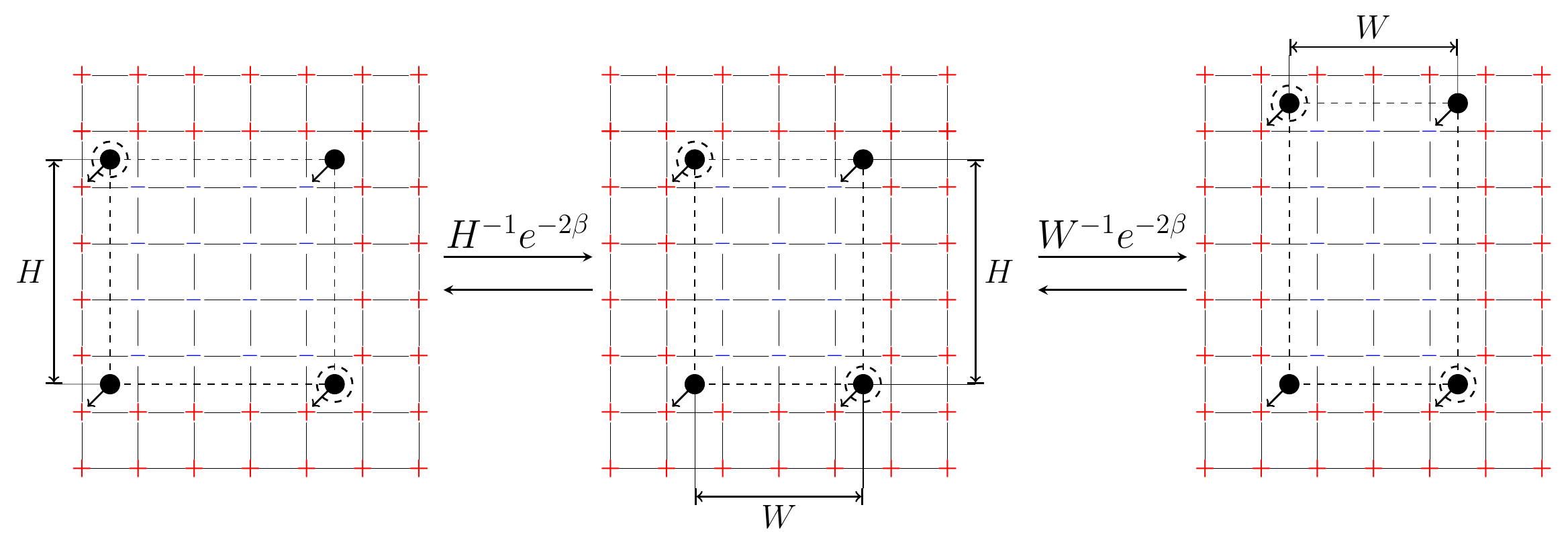}
\caption{\label{FigCornerWalk} The defects at the corners of a box perform random walks at different rates: they move left-right at rate $H^{-1} e^{-2 \beta}$ and up-down at rate $W^{-1} e^{-2 \beta}$, where $H$ is the height of the box and $W$ is its width. These random walks are close to independent for $\beta$ large.}
\end{figure}

The canonical path method (see Section \ref{SecCanonicalPathAbs}) can be used to extend this heuristic argument to configurations with many defects. If we restrict our attention to configurations whose defects form non-overlapping rectangles, the following canonical path construction would again give an $O(e^{3.5 \beta})$ bound on the relaxation time:

\begin{enumerate}
\item Pick a rectangle uniformly at random.
\item Follow a minimal-length path that collapses this rectangle without ever adding more than 2 defects to the initial configuration, as given by the heuristic in Figures \ref{FigPairMove} and \ref{FigCornerWalk}.
\item If any defects remain, go back to step \textbf{(1)}.
\end{enumerate}

In other words, if the rectangles don't overlap, we can essentially treat them as evolving independently. Unfortunately, configurations with overlapping rectangles are possible. Allowing for overlaps, this simple path has very high congestion and so gives very poor bounds (see Figure \ref{FigBadConfigNest}).
This problem occurs essentially becase the pairs of short-lived particles in Figure \ref{FigCornerWalk} can ``interact'' in this situation, and a similar problem is expected to occur in other plaquette models.
In Section \ref{SubsecPathLowDef} we will show that it is possible to salvage our heuristic calculations by choosing rectangles according to a much more complicated rule.
This rule is governed by a ``soft'' selection process that focuses on the most ``prominent" rectangles. We use this to show that, on average, these interactions between short-lived defect pairs don't contribute significantly to the congestion.
We suspect that this basic approach can be used to analyze similar interacting-particle systems (in particular other plaquette models).

\section{Main Results} \label{SecMainResults}

We give bounds on the mixing and relaxation time of the square plaquette model on the critical scale:

\begin{thm} [Mixing and Relaxation Times for Plus Boundary Conditions] \label{ThmMainResPlus}

The square plaquette process with all plus boundary conditions, on the critical scale,  satisfies
\be \label{IneqPlusRel}
e^{3.5 \beta} \lesssim T_{\rm rel}^+( L_c ) \lesssim \beta^{6} e^{3.5 \beta}.
\ee
Furthermore,
\be \label{IneqPlusMix}
e^{3.5 \beta} \lesssim T_{\rm mix}^+( L_c ) \lesssim \beta^{9} e^{4 \beta}.
\ee
\end{thm}

\begin{thm} [Mixing and Relaxation Times for Periodic Boundary Conditions] \label{ThmMainResPer}

The square plaquette process with periodic boundary conditions, on the critical scale, satisfies
\be \label{IneqPerRel}
e^{4 \beta} \lesssim T_{\rm rel}^{\rm per}( L_c ) \lesssim \beta^{4} e^{4 \beta}.
\ee
Furthermore,
\be  \label{IneqPerMix}
e^{4 \beta} \lesssim T_{\rm mix}^{\rm per}( L_c ) \lesssim \beta^{9} e^{4 \beta}.
\ee
\end{thm}

\begin{remark}
	With some extra technical effort in the canonical paths argument, the upper bound on the mixing time with plus boundary conditions can be improved (see \cite{SPMcutoff} where it is shown that $T_{\rm mix}^+( L_c ) \lesssim \beta^{7} e^{3.75 \beta}$).
	We conjecture that the lower bound given here is correct up to polynomial factors in $\beta$.
\end{remark}

\begin{remark}
	We recall that the fundamental reason for the different scaling of the relaxation and mixing time with all plus and periodic boundary conditions is the structure of the ground states. 
	With all plus boundary conditions there is a unique ground state (in which all sites have spin $+1$), and the low temperature dynamics are dominated by the time to reach the ground state.
	On the other hand, with periodic boundary conditions, the low temperature dynamics are dominated by an induced random walk on the $2^{2L-1}$ ground states.
	It turns out that this random walk behaves very similarly to the standard hypercube random walk.
\end{remark}

\begin{figure}[tb]
\includegraphics[width=0.9\linewidth]{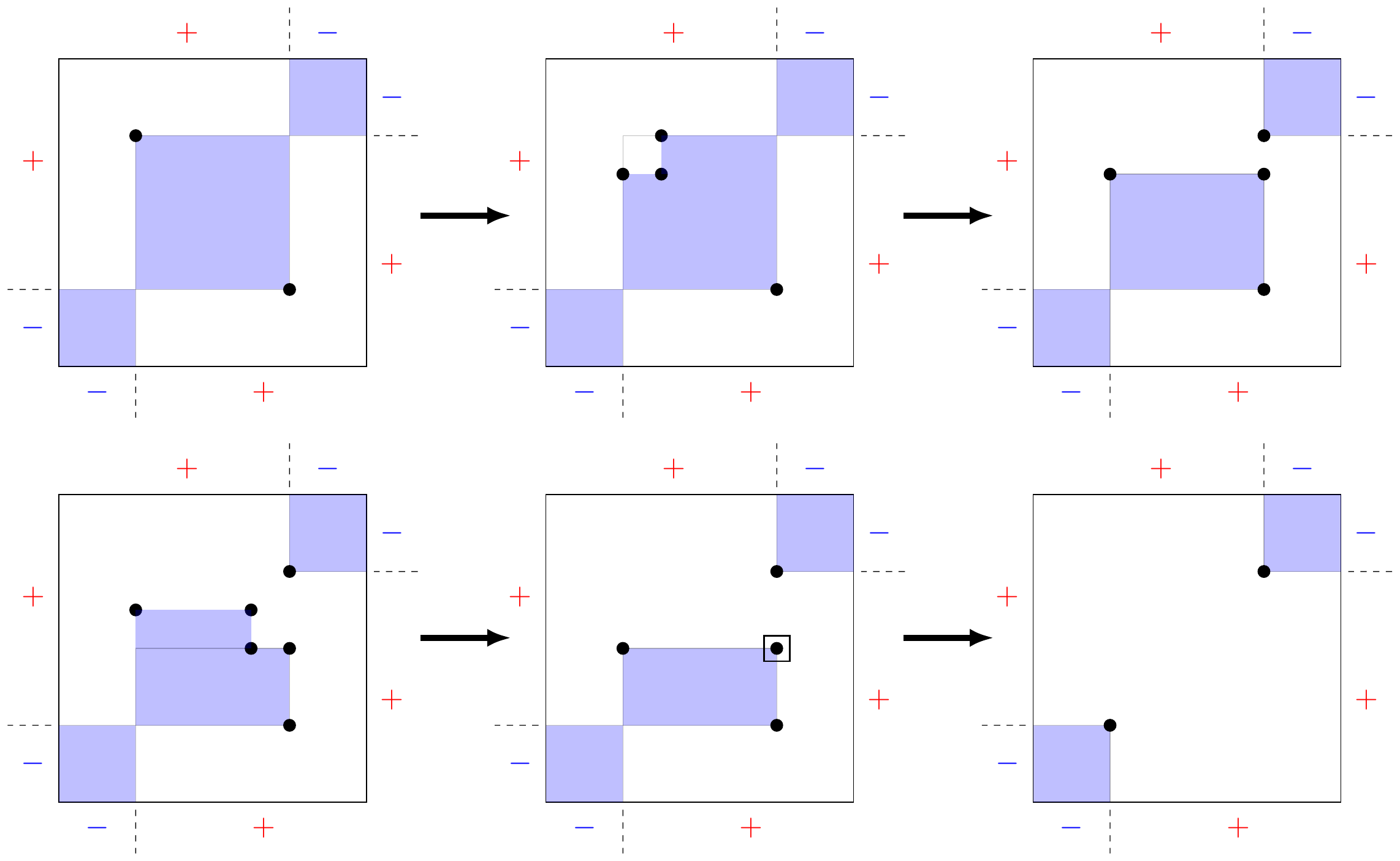}
\caption{\label{FigBad} Two well separated ground states (top left and bottom right), with a typical path between them, as described in Remark \ref{Rem:bad}. The shaded blue region indicates the $-$ spins, and the empty regions are $+$ spins.}
\end{figure}
\begin{remark}
\label{Rem:bad}
For certain special choices of boundary conditions, there exist a small number of ground states that are well separated, in the sense that it takes an extremely long time to switch between them. In the case of the boundary conditions given in Figure \ref{FigBad} we conjecture
  \begin{align*}
 e^{4.5 \beta} \lesssim   T_{\rm rel}^{\t_{\rm bad}}(L_c).
  \end{align*}
  The heuristic behind this bound is as follows.
  Typical paths between the two ground states are of the form described in the figure.
  Initially two new defects are created, at one of the existing defects, which occurs with rate $e^{-2\beta}$.
Subsequently, with probability $1/L$, a pair of defects is not immediately re-absorbed, and travels a distance of order $L$.
At some point during this excursion (top right frame of Fig. \ref{FigCornerWalk}) the process may generates two new defects, with probability $\approx e^{-2\beta}L^2$.
Subsequently, with probability  $1/L$ again, these defects reach one of the isolated defects before being re-absorbed, leading to the stable configuration in the bottom middle Fig \ref{FigCornerWalk}.
  Finally the defect marked with a square box makes a two dimensional random walk following the same mechanism as given in Figure \ref{FigCornerWalk}.
 The projection of the position of this defect onto the diagonal (bottom left to top right) is a martingale, and so with probability $O(1/L)$ this defect will travel a distance $\Omega(L)$ before being re-absorbed in the previous ground state.
Putting all these factors together gives the heuristic bound.
\end{remark}

\section{Preliminary Calculations and Notation} \label{SecPrelimResults}



In this section, we make some simple observations about which configurations are possible and likely at stationarity. We first observe that only defect configurations satisfying certain parity constraints are possible, and these depend on the boundary conditions (Lemmas \ref{lem:parity}, and \ref{lem:parityPer}). These parity conditions  immediately imply some rough bounds on the number of configurations with a given number of defects (Lemma \ref{LemmaExtensionWoo}), which in turn implies that the ground state has high probability under the stationary measure (Lemma \ref{lem:dom}). \\

In the case $\L = [1 : \ell_{1}] \times[1 : \ell_{2}]$, we have $\cB(\L) = [0 : \ell_1] \times [0 : \ell_2]$.
In this case we denote all the sites belonging to the $i^{\mathrm{th}}$ column of $\cB(\L)$ by $C_i = \{i\}\times [0 : \ell_2]$, and the  $j^{\mathrm{th}}$ row by $R_j = [0 : \ell_1]\times \{j\}$.
For $x,y \in \cB(\L)$ we write $y \ll x$ if $y_1 < x_1$ and $y_2 < x_2$.

\begin{lemma}[$\t$-b.c. parity constraints]
\label{lem:parity}
If $\L=[1\! :\! \ell_1]\times[1\! :\! \ell_2]$ then for boundary condition $\t \in \O$
\begin{align*}
  p^\t(\O_\L^\t) = \{p \in  \Omega_{\cB(\L)}\,:\, p \textrm{ satisfies \eqref{eq:par1} and \eqref{eq:par2} below}\} \,,
\end{align*}
where the parity constraints are given by
\begin{align}
  \prod_{x\in C_i}p_x&= \t_{(i,0)}\t_{(i+1,0)}\t_{(i,\ell_2+1)}\t_{(i+1,\ell_2+1)}\,, \quad \forall\, i \in [0: \ell_1] \quad \textrm{and,}\label{eq:par1}\\
 \prod_{x\in R_j}p_x&= \t_{(0,j)}\t_{(0,j+1)}\t_{(\ell_1+1,j)}\t_{(\ell_1+1,j+1)}\,,\quad \forall\, j\in [0:\ell_2] \,. \label{eq:par2}
\end{align}
Furthermore, $p^\t:\O_\L^\t \to p^\t(\O_\L^\t)$ is bijective, and the inverse may be written as
\begin{align}
  \label{eq:inverse}
 (p^\t)^{-1}(p)_x= \s^\t(p)_x :=
  \begin{cases}
    \t_{(0,0)}\t_{(x_1,0)}\t_{(0,x_2)}\prod_{y\ll x}p_{y}& \textrm{if } x \in \L,\\
    \t_x & \textrm{otherwise,}
  \end{cases}
\end{align}
for $p\in p^\t(\O_\L^\t)$.
\end{lemma}

\begin{lemma}[Periodic b.c. parity constraints]
\label{lem:parityPer}
If $\L=[0: \ell_1]\times[0:\ell_2]$ then for periodic boundary conditions
\begin{align*}
  p^{\mathrm{per}}(\O_\L) = \{p \in  \Omega_{\L}\,:\, p \textrm{ satisfies \eqref{eq:parper} }\} \,,
\end{align*}
where the parity constraints are given by
\begin{align}
  \prod_{x\in C_i}p_x&= \prod_{x\in R_j}p_x= 1\,, \quad \forall\, i \in [0:\ell_1-1] \quad \textrm{and} \quad \forall\, j\in [0:\ell_2-1] \,. \label{eq:parper}
\end{align}
and for any fixed $\t \in \{-1,1\}^{C_0 \cup R_0}$, we have  $p^{\mathrm{per}}:\{\s\in\O_\L\,:\, \s_{C_0\cup R_0}\equiv \tau\} \to p^{\mathrm{per}}(\O_\L)$ is bijective.
\end{lemma}

\begin{proof}
  The proof relies on the following observation; for $\s \in \{-1,1\}^{\L'}$, for some $\L'$ finite, and $A,B \subset \L'$ we have
  \begin{align}
\label{eq:prod}
    \prod_{x\in A}\s_x\prod_{y\in B}\s _y=\prod_{x\in A\triangle  B}\s_x\,,
  \end{align}
where $A\triangle B$ denotes the symmetric difference of $A$ and $B$.
It follows that for $\s \in \O_\L^\t$ and $i \in [0:\ell_1]$
\begin{align*}
  \prod_{x\in C_i}p_x^\t(\s)= \prod_{x\in C_i}\prod_{y\in B_x}\s_y = \t_{(i,0)}\t_{(i+1,0)}\t_{(i,\ell_2+1)}\t_{(i+1,\ell_2+1)}\,,
\end{align*}
and similarly for rows, i.e. $p^\t(\O^\t_\L) \subset \{p \in  \Omega_{\cB(\L)}\,:\, p \textrm{ satisfies \eqref{eq:par1} and \eqref{eq:par2}}\}$.

To prove the opposite inclusion, fix $\z \in  \{p \in  \Omega_{\cB(\L)}\,:\, p \textrm{ satisfies \eqref{eq:par1} and \eqref{eq:par2}}\}$, then $\s^\t(\z) \in \O_\L^\t$ by the definition given in \eqref{eq:inverse}.
Furthermore, using (\ref{EqDefectMap}) and \eqref{eq:prod},
\begin{align*}
   p^\t(\s^\t(\z))_x = \prod_{y\in B_x}\s^\t(\z)_y = \prod_{y\ll x}\z_y\prod_{z\ll x+e_1}\z_z\prod_{y'\ll x+e_2}\z_{y'}\prod_{z'\ll x+e_1+e_2}\z_{z'} = \z_x
\end{align*}
This completes the proof of Lemma \ref{lem:parity}, since $|\O_\L^\t| = | \{p \in  \Omega_{\cB(\L)}\,:\, p \textrm{ satisfies \eqref{eq:par1} and \eqref{eq:par2}}\}| = 2^{|\L|}$.
The proof of Lemma \ref{lem:parityPer} follows the same arguments.
\end{proof}

These parity considerations imply the following bounds on the number of configurations with a given number of defects:
\begin{lemma} \label{LemmaExtensionWoo}
Fix $k \in \mathbb{N}$. If $\L=[1:L]^2$, then with all plus boundary conditions
  \begin{align}
    |\{\s \in \O_\L^+\,:\, |p^+(\s)| = 2k\}| \leq \min\{(e k)^{2k} L^{2k},L^{3k}\} \,.
  \end{align}
If $\L=[0:L]^2$, then with periodic boundary conditions
   \begin{align}
    |\{\s \in \O_\L\,:\, |p^\mathrm{per}(\s)| = 2k\}| \leq  2^{2L+1}\min\{(e k)^{2k} L^{2k},L^{3k}\} \,.
  \end{align}
\end{lemma}
\begin{proof}
We begin with the $+$ boundary conditions. By Lemma \ref{lem:parity} we have, for each $k \in \N$, $$|\{\s \in \O_\L^+\,:\, |p^+(\s)| = 2k\}| = |\{p \in p^+(\O_\L^+) \,:\, |p| = 2k\}|$$
and  $|\{p \in p^+(\O_\L^+) \,:\, |p| = 2k\}|$ is exactly the number of configurations in $\{0,1\}^{\L}$, where we identify defects with $1$'s, which have even row and column sums and the total number of defects is $2k$.
Since any row which contains at least one defect must also contain at least two, the total number of rows which have at least one defect is clearly bounded above by $k$ (similarly for columns).
It follows that the number of configurations which satisfy the row and column parity constraint is at most the number of ways of firstly choosing $k$ rows and $k$ columns (with replacement), and then arranging the $2k$ defects anywhere on the sub-lattice defined by these rows and columns.
The number of ways of arranging the former is $L^{2k}$, and the number of ways of arranging the latter is clearly bounded above by $\binom{k^2}{2k}$. Using $\binom{n}{k} < (e\, n)^{k}/k^k$ and taking the product gives the first upper bound.
The second upper bound follows by using only the row parity, which implies that the defects may be grouped into $k$ disjoint pairs, such that defects belonging to the same pair occupy the same row.
There are then at most $L^3$ ways to arrange each pair on $\L$ and the upper bound $L^{3k}$ follows immediately.

We next consider periodic boundary conditions. By the above argument, for every choice of $\t \in \{-1,1\}^{C_0 \cup R_0}$ we have
\be
|\{\s \in \O_\L\,:\, |p^\mathrm{per}(\s)| = 2k, \, \res{C_{0} \cup R_{0}}{\s} = \t \}| \leq \min\{(e k)^{2k} L^{2k},L^{3k}\}.
\ee
Applying the bijection in Lemma \ref{lem:parityPer} and noting that $|\{-1,1\}^{C_0 \cup R_0}| = 2^{2L+1}$ completes the proof.
\end{proof}

For a given boundary condition $\tau$, define the collection of \textit{ground states} by:
\begin{align}
  \cG = \cG^{\t} = \{ \sigma \in \O_\L^{\t} \, : \, H_\L^{\t}(\sigma) = \min_{\eta \in \O_\L^\t} H_\L^{\t}(\eta) \}.
\end{align}

\begin{lemma}[Domination of Ground States]
\label{lem:dom}
  Let $\t \in \{+,\rm{per}\}$ and $L=L_c$, then
  \begin{align}
    \pi(\cG) \geq 1 - O(e^{-2\b})\,.
  \end{align}
\end{lemma}

\begin{proof}
By Lemmas \ref{lem:parity} and \ref{lem:parityPer}, we note that all possible configurations have an even number of defects, and  also that no configurations can have exactly 2 defects. When $\t = +$, there is a unique ground state $+ \equiv (+1,\ldots,+1)$.
In this situation, Lemma \ref{LemmaExtensionWoo}  gives
\begin{align*}
 1-\pi(+)&= \sum_{k \geq 2}\pi(+) | \{ \sigma \in \O_\L^+\, : \, |p^+(\s)| = 2k \} | e^{-2k \, \beta}  \lesssim \pi(+) e^{-2\beta}\,. 
\end{align*}
Thus $\pi(+) = 1 - O(e^{-2 \beta})$, completing the proof of the lemma for $\t = +$.
For the case of periodic boundary conditions, observe that there are exactly $2^{2L+1}$ ground states, so $\pi(\cG) = 2^{2L+1}\pi(+)$ and the same argument holds.
\end{proof}

\section{Construction and Analysis of Canonical Paths} \label{SecConCanPath}

Our arguments for the upper bounds in Theorems \ref{ThmMainResPlus} and \ref{ThmMainResPer} will be based on the method of ``canonical paths.'' The idea in this method is to construct a family of (possibly random) paths between any pairs of configurations $\sigma, \eta$, such that the paths do not use any one edge too much.  In this section, we construct the canonical paths that will be used for those proofs, and also give some initial analysis of their properties. As a guide to the remainder of this section:

\begin{itemize}
\item In Section \ref{SecCanonicalPathAbs}, we state generic bounds on relaxation and mixing times of a Markov chain in terms of canonical paths. These are small variants of well-known results, but may be useful in the study of other Markov chains.
\item In Section \ref{SubsecPrelPathNot}, we give preliminary notation related to canonical paths.
\item In Section \ref{SubsecPathLowDef}, we construct and analyze the main ``building block" of our canonical path construction. This building block determines the entire path for any initial configuration $\sigma$ with $O(\beta L)$ defects.
\item In Section \ref{SubsecPathHighDef}, we construct a ``building block" that may be used when $\sigma$ has many more than $L$ defects. We also combine our two building blocks to construct the main canonical path used throughout this paper.
\item In Section \ref{SecPathAnalysisLemmas}, we give simple bounds on various quantities related to combining the building blocks into a longer canonical path.

\item In Section \ref{SecFinalBoundCanonical}, we combine the results in this section to obtain final bounds on the properties of the canonical paths we have defined.
\end{itemize}

\subsection{Canonical Path Bounds} \label{SecCanonicalPathAbs}
In this section we describe the use of canonical path arguments to bound the spectral gap and spectral profile of a Markov process on finite state space.
We first set some standard notation for this subsection only. For a general Markov chain with transition rate matrix $K$ on state space $\Omega$, denote by $\edges = \{(\eta,\sigma) \in \Omega^{2} \, : \, K(\eta,\sigma) > 0\}$ the collection of transitions allowed by $K$. We recall the notion of \textit{paths} in such a Markov chain:

\begin{defn}[Path]

A sequence $\gamma = (\h^{(0)},\h^{(1)},\ldots,\h^{(m)}) \in \Omega$ is called a \textit{path from $\h^{(0)}$ to $\h^{(m)}$}  if  $(\h^{(i-1)},\h^{(i)}) \in \edges$ for all $1 \leq i \leq m$. We say that this path has length $|\gamma| \equiv m$. For $1 \leq i \leq m$, we call the pair $(\h^{(i-1)},\h^{(i)})$ the \textit{$i$'th edge of $\gamma$}. For $\h,\s \in \Omega$, we denote by $\Gamma_{\h,\s}$ the collection of all paths from $\h$ to $\s$. Similarly we let $\Gamma_{\h} = \cup_{\s} \Gamma_{\h,\s}$ be the collection of all paths starting at $\h$ and $\Gamma = \cup_{\h,\s \in \config} \Gamma_{\h,\s}$ be the collection of all paths.

For a path $\gamma = (\h^{(0)},\ldots,\h^{(m)})$, we denote by $\init{\gamma} = \h^{(0)}$ and $\fin{\gamma} = \h^{(m)}$ the \textit{initial} and \textit{final} elements of $\gamma$. If $\gamma_{1}$, $\gamma_{2}$ are two paths with $\fin{\gamma_{1}} = \init{\gamma_{2}}$, we denote by $\gamma_{1} \conc \gamma_{2}$ the concatenation of $\gamma_{1}$ and $\gamma_{2}$ with the repeated element $\gamma_{1}^{(fin)}, \gamma_{2}^{(init)}$ removed. With some abuse of notation, we define $\emptyset \conc \gamma = \gamma \conc \emptyset = \gamma$ for any path $\gamma$.
\end{defn}

When applying the results in this Section, we will always use the process given by \eqref{eq:gen}.

For the remainder of this section, assume that $K$ has a unique reversible measure $\mu$ and let $\mu_{\ast} = \min_{\h\in \Omega} \mu(\h)$. For $S \subset \Omega$, define $c_{0}(S)$ to be the non-constant functions on $\Omega$ with support contained in $S$. For $S \subset \Omega$, also define
\be \label{EqDefSpecBoundLambdaDef}
\lambda(S) = \inf_{f \in c_{0}(S)} \frac{\mathcal{D}_{K}(f)}{\var_{\mu}(f)},
\ee
where $\mathcal{D}_{K}$ is the Dirichlet form associated with $K$,
 \be \label{eq:DirForm}
\mathcal{D}_{K} = -\mu(f\,Kf)= \frac{1}{2}\sum_{\h,\s \in \Omega}\mu(\h)K(\h,\s)(f(\h)-f(\s))^2\,.
\ee

\begin{lemma} [Multicommodity Flows for Bounding the Spectral Profile] \label{LemmaMultiSpec}

Let $S \subsetneq \Omega$. For $\h \in S$, let $F_{\h}$ be a probability measure on paths $\gamma = (\gamma^{0},\ldots, \gamma^{m})$ in $\Omega$ that have starting point $\gamma^{0} = \h$ and final point $\gamma^{m} \in S^{c}$. Then

\be
\lambda(S) \geq \mathcal{A}^{-1},
\ee
where
\be \label{IneqDefCanPathObject}
\mathcal{A} \equiv 2\, \max_{e \in \edges} \sum_{\h \in \Omega} \sum_{\gamma \ni e} F_{\h}(\gamma) \, | \gamma | \, \frac{\mu(\h)}{\mu(e_{-}) K(e_{-},e_{+})}
\ee
is the cost of the flow $\{F_{\h}\}_{\h \in S}$.
\end{lemma}

\begin{remark}
We will use this with $S = \config \backslash \{+\}$ when bounding the spectral gap of the SPM with + boundary conditions. The more general bound is useful for bounding the spectral profile, which  we will need to obtain a more refined mixing time bound.
\end{remark}

\begin{proof}

Fix $f \in c_{0}(S)$ with $\E_{\mu}[f] = 0$. By the Cauchy-Schwarz inequality and the fact that $f(y) = 0$ for all $y \in S^{c}$,
\be
\var_{\mu}(f) &=  \sum_{\h \in \Omega} \mu(\h) f(\h)^{2} = \sum_{\h \in \Omega}\mu(\h) \sum_{\gamma} F_{\h}(\gamma) ( \sum_{i=1}^{|\gamma| } (f(\gamma^{i-1}) - f(\gamma^{i})))^{2} \\
&\leq \sum_{\h \in \Omega} \sum_{\gamma} \mu(\h)F_{\h}(\gamma) \, | \gamma | \, \sum_{i=1}^{|\gamma|} (f(\gamma^{i-1}) - f(\gamma^{i}))^{2} \\
&= \sum_{\h \in \Omega} \sum_{\gamma}\mu(\h) F_{\h}(\gamma) \, | \gamma | \, \sum_{i=1}^{|\gamma|} \frac{\mu(\gamma^{i-1}) K(\gamma^{i-1}, \gamma^{i})}{\mu(\gamma^{i-1}) K(\gamma^{i-1}, \gamma^{i})} (f(\gamma^{i-1}) - f(\gamma^{i}))^{2}.
\ee
Comparing this to the formula for the Dirichlet form, this gives
\be
\frac{\var_{\mu}(f)}{\mathcal{D}_{K}(f)} \leq 2\, \max_{e \in \edges} \sum_{\h \in \Omega} \sum_{\gamma \ni e} F_{\h}(\gamma) \, | \gamma | \, \frac{\mu(\h)}{\mu(e_{-}) K(e_{-},e_{+})},
\ee
completing the proof of the lemma.

\end{proof}

This bound is all that we will require to bound the spectral gap of the SPM model with + boundary conditions. In order to bound the mixing time of both chains, as well as the spectral gap of the SPM with periodic boundary conditions, we will use the spectral profile method introduced in \cite{Goel2006}. We quickly recall the main results of that paper:

\begin{defn} [Spectral Profile Bounds.]

For $r > 0$, define
\be
\Lambda(r) = \inf_{\mu_{\ast} \leq \mu(S) \leq r} \lambda(S).
\ee
Theorem 1.1 of \cite{Goel2006} states that the mixing time $\tmix$ of $K$ satisfies
\be \label{IneqMainSpectralProfileBound}
\tmix \leq \int_{4 \, \mu_{\ast}}^{16} \frac{2}{x \, \Lambda(x)} dx.
\ee
\end{defn}

 In our application, the spectral profile $\Lambda(r)$ is defined as an infimum over a very large collection of subsets $S \subset \O_\L^\t$; we will find it easier to work with a much-reduced collection of subsets. For fixed $r>0$, let
\be \label{DefHighDensityRound}
k(r) = \min \{ k \, : \, \pi(+)  \, e^{-\beta k} \leq r \}.
\ee
For fixed $k \in \mathbb{N}$, let
\be \label{DefHighDensitySet}
S_{k} = \{ \sigma \in \O_\L^\t \, : \, |p^\t(\sigma)| \geq k \}.
\ee
We observe that, for fixed $r>0$ and $S \subset \config$ with $\pi(S) \leq r$, we have
\be \label{IneqSetContainmentSpectralProfile}
S \subset S_{k(r)}.
\ee
Thus, we have
\be \label{IneqSpecProfileSubsets}
\Lambda(r) &= \inf_{\mu_* < \mu(S) < r} \inf_{f \in c_{0}(S)} \frac{\mathcal{D}_{\cL}(f)}{\var_{\pi}(f)} \\
&\geq \inf_{f \in c_{0}(S_{k(r)})} \frac{\mathcal{D}_{\cL}(f)}{\var_{\pi}(f)}=\lambda(S_{k(r)}).
\ee
Thus, to bound the mixing time, it is enough to bound $\inf_{f \in c_{0}(S)} \frac{\mathcal{D}_{\cL}(f)}{\var_{\pi}(f)}$ below, for sets of the form $S = S_{k}$ and some $k \in \mathbb{N}$.

\subsection{Preliminary Notation} \label{SubsecPrelPathNot}

We will use the following notation for defects frequently:

\begin{defn} [Neighbouring Defects] \label{DefNeighbourDefect}

Let $\Lambda = [\ell_{1}:\ell_{2}] \times [\ell_{3}:\ell_{4}]$. Fix $p \in \pconfig$ and define
\be
D(p) = \{x \in \pconfig \,: \, p_{x} = -1 \}.
\ee

For $p \in \pconfig$ and $x = (x_{1},x_{2}) \in D(p)$, define the \textit{left row neighbour}, \textit{right row neighbour}, \textit{down column neighbour} and \textit{up column neighbour} of the defect at $x$ by
\be  \label{EqNearestNeighbours}
\ell(x) &= (\max\{x_1' < x_1\,:\, (x_1',x_2) \in D(p)\},x_{2}) \\
r(x) &= (\min\{x_1' > x_1\,:\, (x_{1}',x_{2}) \in D(p)\}, x_{2}) \\
d(x) &= (x_{1},\max\{x_2'<x_2\,:\, (x_{1},x_{2}') \in D(p)\}) \\
u(x) &= (x_{1},\min\{x_{2}'> x_2\,:\,(x_1,x_2') \in D(p)\})
\ee
when they exist; otherwise by $\emptyset$ as appropriate.
Finally, by a small abuse of notation we define
$
D(\sigma) = D(p^{+}(\sigma))
$
for $\sigma \in \O_\L^+$.
\end{defn}

\begin{figure}[tb]
\includegraphics[width=0.6\linewidth]{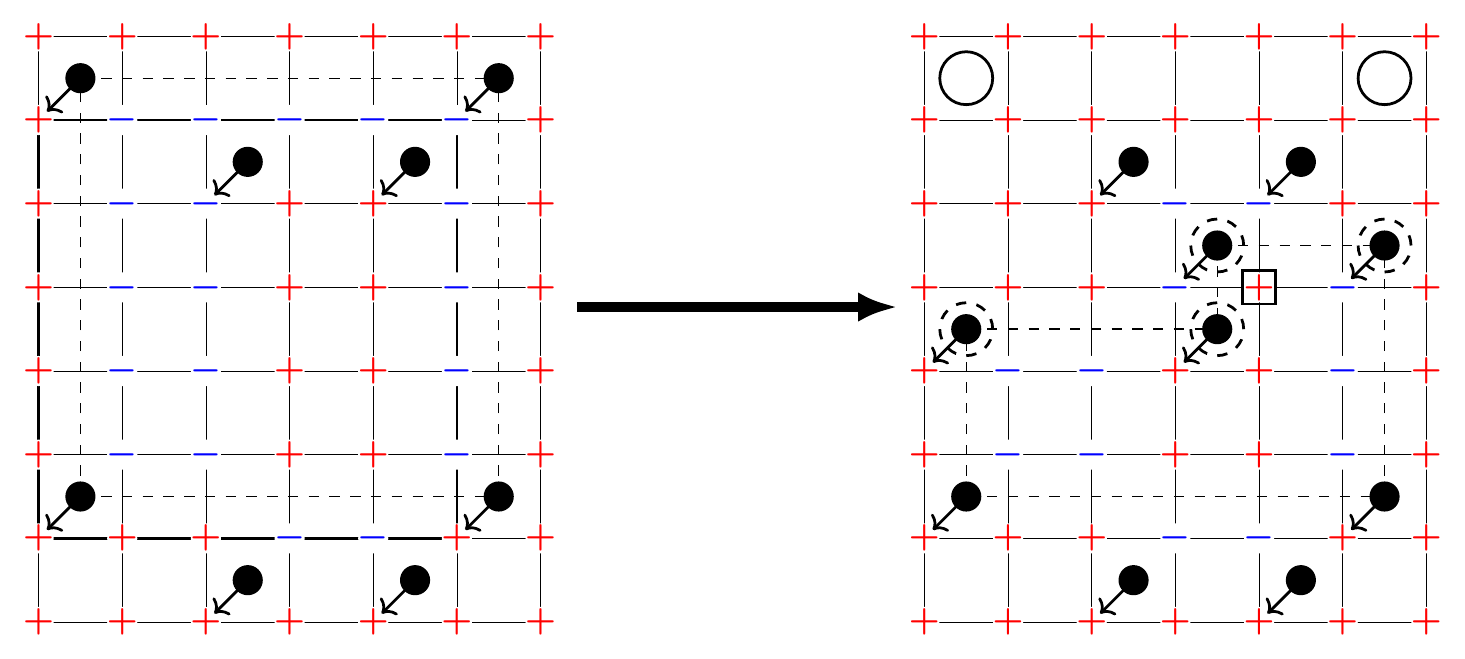}
\caption{\label{FigRectRemoveDiff} Left: an initial configuration $\sigma$, with defects $D(\sigma)$ indicated by black circles. The thick-black 
	 box shows the rectangle $R$, the sites inside the dashed box correspond to $\cB_-(R)$ which are flipped. Right: An example of a configuration $\eta \in \gamma_{\sigma,R}$, the next spin to be flipped in the path is marked by a solid box. Open circles denote \textit{absent} defects (with respect to $\sigma$), dashed circles denote defect variables which are also flipped with respect to $\sigma$. In this example the top row of the rectangle $R$, bottom row, and the two rows containing defect variables which differ from $D(\sigma)$, are all distinct; it is also possible for the rows to intersect (see Equation \eqref{eq:typedef} for an explicit list of these  cases - the example above shows a \bf{mid}-type edge).}
\end{figure}

We now define a simple ``base path" that flips all the spins in a rectangle $R \subset \Z^{2}$ in a sensible order. Typically in application, the rectangle $R$ in the following definition will have at least three defects at its vertices in the initial configuration $\sigma$, and this ``base path'' removes two or four of them without ever adding more than two in intermediate stages:

\begin{defn} [Rectangle-Removal Path] \label{DefRectRemovePath}
Let $\Lambda = [\ell_{1}\!:\!\ell_{2}]\times[\ell_{3}\!:\!\ell_{4}]$. Given $\sigma \in \config$ and a rectangle $R = [x_{1}\!:\!x_{2}] \times [y_{1}\!:\!y_{2}] \subset \cB(\Lambda)$ we define a path $\gamma_{\sigma,R}$, starting at $\sigma$, that flips all the spins at all the sites in $\cB_{-}(R)$. We construct the path according to the following cases:
\begin{enumerate}
\item If $(x_{1},y_{1}), (x_{1},y_{2}), (x_{2},y_{2}) \in D(\sigma)$, \textit{or} fewer than 3 of the four corners of the rectangle are in $D(\sigma)$, we define $\gamma_{\sigma,R}$ to be the path that flips all spins in $\cB_{-}(R)$ in lexicographic order.
\item $(x_{2},y_{2}), (x_{1},y_{2}), (x_{2},y_{1}) \in D(\sigma)$  but $(x_{1},y_{1}) \notin D(\sigma)$, we define $\gamma_{\sigma,R}$ to be the path that that flips all spins in $\cB_{-}(R)$ in anti-lexicographic order.
\item Otherwise, denote by $M \, : \, \config \mapsto \config$ the ``mirror reflection" map
\be \label{EqMirrorMap}
\rm{M}(\sigma)_{i,j} = \sigma_{i,\ell_{4} + \ell_{3} - j}\,
\ee
which flips the lattice in the $x$-axis.
Note that $M(R)$ is now in one of the two previous cases, and we define $\gamma_{\sigma,R} = M^{-1}(\gamma_{M(\sigma), M(R)})$.
\end{enumerate}
\end{defn}

Although a \textit{spin configuration} $\eta \in \gamma_{\sigma,R}$ can be very different from $\sigma$, their associated \textit{defect configurations} $D(\eta), D(\sigma)$ can only differ in a few locations, as illustrated in Figure \ref{FigRectRemoveDiff}.



Denote by $\rect_{\L}$ the collection of rectangles in the lattice $\L$. For $R = [x_{1}: x_{2}] \times [y_{1} : y_{2}] \in \rect_{\L}$, we informally call $y_{2}-y_1$ the ``height'' and $x_{2} - x_{1}$ the ``width'' of $R$. Denote by $\edges = \{(\s,\eta)\in\config^2 \, : \,  \cL_\L^+(\s,\eta) > 0\}$ the collection of edges associated with the generator $\cL_\L^+$. Finally, For any set $S$, denote by $\mathcal{P}(S)$ the power set of $S$.

\subsection{Path Construction: Initial Segment} \label{SubsecPathLowDef}

In this section, we give a description of the main building block of our canonical path. The basic heuristic for these paths was presented in Section \ref{SecHeuristic}: pick a random rectangle with at least three defects in its corners, then ``squish" the rectangle using the path in Definition \ref{DefRectRemovePath}.  For reasons discussed shortly, picking a rectangle uniformly at random leads to very poor bounds. In this section we show that the heuristic can be saved with a small tweak: rather than choosing uniformly at random, we make a very careful choice of ``good" rectangles at each stage. 

Throughout, we will denote by $F$ a function which maps a configuration to a collection of ``good" rectangles, and $\flowtot{\sigma}$ a measure on paths associated with this function $F$. In order to bound the congestion in Equation \eqref{IneqDefCanPathObject} (i.e. the total weight of paths using a given edge), we define a recovery function $G$:
\be
G(e) = \{\s \in \config\,:\, \exists\, \gamma \ni e, \ \ s.t. \ \  \flowtot{\sigma}(\gamma) > 0 \}\,,
\ee
(see Equation \eqref{eq:G} at the start of Definition  \ref{def:Gs} for an equivalent definition).
The details of the functions $F$ and $G$ are a bit complicated, and so for the reader's convenience we do not make explicit reference to them outside of this section. 
The only results required from this section are 
the notation for the path measure $\flowtot{\sigma}$ (see Definition \ref{DefPartialPath}), and our final estimates in Proposition \ref{LemmaTypicalReconstruction}.

We now give a rough description of our path construction, followed by the full details. 
The main problem with the heuristic of Section \ref{SecHeuristic} is  illustrated in Figure \ref{FigBadConfigNest}. In this figure, we see an edge $e= (e_{-}, e_{-}^{x})$ in a path $\gamma_{\sigma,R}$ for which the row containing the spin $x$ that is being flipped by $e$ has very many defects. In this case, the edge could appear in an enormous number of paths of the form in Definition \ref{DefRectRemovePath}. As a result, the congestion of our path is large and the canonical path method gives a very poor bound on the relaxation time.

\begin{figure}[h]
\includegraphics[width=0.95\linewidth]{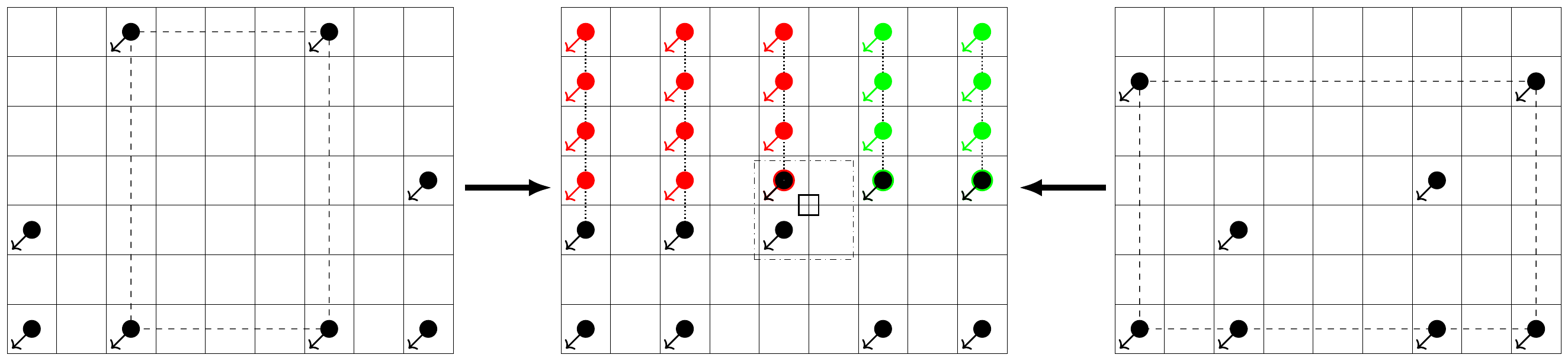}
\caption{\label{FigBadConfigNest}The middle display shows an edge $(e_{-}, e_{+})$; the defects in $e_{-}$ are marked by black circles, and the spin that is flipped in $e_+$ is marked with a square (the corresponding defect variables that are flipped are contained in the dot-dashed box).
If this edge occurs along the path in Definition \ref{DefRectRemovePath}, the top-left and top-right defects in the initial configuration could have appeared anywhere marked by red and green defects respectively, such that they are both in the same row.
Two representative possible initial configurations which both go through this same edge are shown left and right. Only part of the lattice $\L$ is shown; the configurations do not satisfy the parity constraint, Equation \ref{eq:par2}, on this part of the lattice.}
\end{figure}

There is a conceptually-simple fix to the problem illustrated in Figure \ref{FigBadConfigNest}: if no row has ``too many" defects, select rectangles \textit{uniformly at random among all rectangles} with at least three defects at vertices; otherwise, choose \textit{uniformly at random among rectangles with top on the high-density row.} 
This approach almost works, and leads to the following rectangle-choosing function for configurations $\sigma$ with fewer than $L$ defects:

\begin{enumerate}
	\item If no row contains more than $\beta^{\Theta(1)}$  defects, choose any rectangle with defects at three or more corners.
	\item Otherwise, let $K$ be the maximal number of defects in any row and, for $1 \leq k \leq K$, let $\#(k)$ be the collection of rows with $k$ defects. 
	Our first fix to our initial heuristic tells us to choose any rectangle that has defects at three or more corners \textit{and} whose top edge is in $\#(K)$.
	However, if  we observe an edge with maximal number of defects in any row is $k'$, and many rows contain $k'{-}2$ defects, then rectangle chosen could have had it's top edge in any of the rows with occupation $\#(k')$ or $\#(k'-2)$.
	Since the row occupations in an edge configuration are slightly perturbed with respect to the initial configuration, in order to be able to accurately predict which row we started from we must choose a high density row in the initial configuration such that there are not too many more rows with similar occupations.
	We therefore choose $k$ to be the largest number  that satisfies $|\#(k)| \gtrsim \beta^{-1} |\#(j)|$ for each $j$ in some bounded interval near $k$. We then choose any rectangle that has defects at three or more corners \textit{and} whose top edge is in $\#(k)$.
	It turns out this allows us to make a sufficiently good estimate of which row the top of the rectangle was in, having observed the edge $e$.
\end{enumerate}

If we were only interested in bounding the relaxation time for $+$ boundary conditions, it would suffice to choose a rectangle uniformly at random from the subset constructed above. However, in order to obtain a bound on the mixing time, and on the relaxation time for the periodic boundary conditions, we need to get a much tighter bound for initial configurations that have very many defects. This tighter bound gives us better control on the constant appearing in \eqref{IneqDefCanPathObject} when the initial set $S$ appearing in that lemma is small.

It turns out that we need a tighter bound exactly when the initial configuration, $\s$, has many more than $L$ defects. To obtain a better bound in this case, we will ``split'' the lattice $\Lambda$ into roughly $\m(\sigma) \approx \frac{|D(\sigma)|}{L}$ pieces, each of which contains $O(L)$ defects. It will turn out that our main congestion bounds will apply to each ``piece'' of the split, and so this splitting analysis will give us a ``free'' improvement to our final bound on the order of $\m(\sigma)$. This improvement is critical for obtaining a good bound on the mixing time of the process, since it will imply that defects can disappear much more quickly when there are very many of them. We don't know of any simple way to obtain this  improvement by a factor of $\m(\sigma)$ by other small modifications of our main argument.

The picture presented by this construction is quite simple, but we require quite a large amount of notation to describe it precisely. The main difficulties are that (i) we must also be able to reconstruct the part $i \in [1:L]$ of the split by observing an edge in a sample path, and (ii) the strategy used to define $F(\sigma,i)$ depends on comparing certain rough statistics of $\sigma$ to certain threshold values, and it is not possible to use a single threshold value for all $\sigma \in \config$. 
We resolve this by partitioning the state space $\config$ into a small number of parts indexed by the choice of threshold.
This type of threshold argument, and associated partitioning of the configuration space, seems to be useful for similar path arguments.

We now give some notation which is necessary to construct the functions $F$ discussed above. We begin by introducing some notation for the rectangles that will be chosen by these functions. When three vertices $x,y,z$ in a lattice are corners of a rectangle, we denote this rectangle by $\mathbf{R}(x,y,z)$. We then introduce some special collections of rectangles:

\begin{defn}[Marked Rectangles] \label{DefTriples}
We follow the notation of Definition \ref{DefNeighbourDefect}, assuming again that $\Lambda$ is of the form $\Lambda = [\ell_{1}:\ell_{2}] \times [\ell_{3}:\ell_{4}]$. For $x,y \in \L$, define
$\mathbf{Top}(x,y)$ and $\mathbf{Bottom}(x,y)$ to be the maximum and minimum of $x$ and $y$ with respect to lexicographical order (respectively).

For a defect configuration $p \in \O_{\cB(\L)}$ and  $x =(x_{1},x_{2}),y=(y_{1},y_{2}) \in D(p)$, we define the associated \textit{extended} ``down-" and ``up-"  rectangles of $x$ and $y$ by setting
\begin{align*}
T_{d}(x,y) &= \mathbf{R}(x,y, \mathbf{Top}(d(x),d(y)))\quad \textrm{if $x_2 = y_2$ and at least one of $d(x), d(y)$ exist,} \\
T_{u}(x,y) &=\mathbf{R}(x,y, \mathbf{Bottom}(u(x),u(y)))\quad \textrm{if $x_2 = y_2$ and at least one of $u(x), u(y)$ exist,}
\end{align*}
otherwise, say $T_{d}(x,y) = \emptyset$ or $T_{u}(x,y) = \emptyset$ respectively.

We denote by
\be \label{EqAllallextdef}
\allallext{p} = \{ T_a(x,y) \, : \, x,y \in D(p), \textrm{ and } a \in \{u,d\}\}
\ee
the collection of \textit{extended rectangles} of $p$. For $i \in \mathbb{N}$, define
\be \label{EqAllextdef}
\allext{p}{i} = \{T_{a}(x,y) \, : \, x,y \in D(p), \, i= y_{2} = x_{2}, \, a \in \{u,d\} \}
\ee
to be the collection of extended rectangles with two defects in row $i$. When the lattice $\L$ is not clear from the context, we use the notation $\allallext{p,\L}$ and $\allext{p,\L}{i}$ for Equations \eqref{EqAllallextdef} and \eqref{EqAllextdef}.

Finally, by a small abuse of notation we define $\allallext{\s} = \allallext{p^+(\s)}$ and $\allext{\s}{i} = \allallext{p^+(\s)}$ for $\sigma \in \O_\L^+$.
\end{defn}

We will split the lattice $\L$  into roughly $\max(1,\frac{|D(\sigma)|}{L})$ pieces, each of which is a union of adjacent columns. 

\begin{defn} [Splits] \label{DefBalancedSplits}
	Let $\bar \L =  \cB([1:L]^2)$ and 
	\be
		\bar \Lambda(i,j) = \{i,i+1,\ldots,j\} \times [0:L], \quad \textrm{for } 0 \leq i < j \leq L\,.
	\ee
	For $\s \in \O_{ \L}$ set $s_{1}(\s) = 0$, then inductively define
	\be \label{EqBalSplitRec}
	s_{i+1}(\s) &=  \min \{ j > s_{i}(\s) \, : \, |D(\s) \cap \bar \Lambda(s_{i}(\s),j-1)| \geq 100 \, L  \}, \quad \textrm{if this is finite,}\\
	s_{i+1}(\s) &= L+1, \qquad \text{otherwise,}
	\ee
	for $1 \leq i \leq m(\s) \equiv \min \{j \, : \, s_{j+1}(\s) = L+1\}$.
	When $\s$ is clear from context we write $s_i$ for $s_i(\s)$.
	Finally, define $\S(\s) = \{s_{i}\}_{i=1}^{m+1}$ to be this splitting of $[0:L]$, and $\S_{i}(\s) =\Lambda(s_{i},s_{i+1}-1)$. 
\end{defn}

The following definition gives some notation for counting the number of defects in a row of a sublattice:
\begin{defn} [Occupancy] \label{DefOcc}
Fix a subset $\Lambda = [\ell_{1}\!:\!\ell_{2}]\times [1\!:\!L] \subset [1\!:\!L]^{2}$, and fix $p \in \O_{\cB(\L)}$. For $0 \leq i \leq L$, define the \textit{occupancy} of row $i$ by
\be \label{EqDefOccInit}
\occ{i}(p) = \left| D(p) \cap \left( [\ell_{1}-1:\ell_{2}]\times\{i\}\right) \right|,
\ee
the number of defects in row $i$. 
When the lattice is not clear from the context, we use the extended notation $\occ{i,\L}(p)$ for $\occ{i}(p)$. Again, by a slight abuse of notation, for $\s \in \config^+$ we let $\occ{i}(\s) = \occ{i}(p^+(\s))$.
\end{defn}

We give some basic facts about splits. Recall that we view $p$ as a defect configuration, even if it is not in the image of Equation \eqref{EqDefectMap} for any configuration $\s$. In particular $p$ may not satisfy the parity condition, for example the parity condition in parts of the a split are satisfied in each column but not necessarily in each row.
 We now show that any piece of a split must have many  extended rectangles, relative to the number of defects which are in the piece:

\begin{lemma}[Splits Contain Many Rectangles] \label{LemmaSplitsManyRect}
Let $\L = [1:L]^2$, for any $\s \in \O_\L^+$,
\be \label{IneqManyRect}
| \allallext{\s} | \geq \frac{1}{4}|D(\s)|\,,
\ee
and for $1\leq i \leq m-1$, letting $p = p^+(\s)|_{\cS_i(\s)}$, then
\be \label{IneqSplitManyRect}
| \allallext{\s,\cS_i(\s)} | \geq \frac{1}{4}(|D(p)|-L-1)\,.
\ee
Furthermore, for any $j$ with $\occ{j,\cS_i(\s)}(\s) \geq 1$,
\be \label{IneqSplitManyRect2}
| \allext{\s,\cS_i(\s)}{j} | \geq \frac{1}{2} (\occ{j,\cS_i(\s)}(\s) - 1)^{2}.
\ee
\end{lemma}

\begin{proof}
For $\s \in \O_\L^+$, for each defect in $p^+(\s)$ there is at least one defect in the same row and at least one in the same column by the parity constraint. 
It follows that each defect belongs to at least one rectangle in $\allallext{\s}$, and can be located at one of at most four vertices of any such rectangle. Inequality \eqref{IneqManyRect} follows directly.

Note that, for any fixed $j \in [0 : L]$, the set
\be
\{ x = (x_{1},x_{2}) \in D(p) \, : \, x_{2} = j, \, r(x) = \emptyset\textrm{ and } \ell(x) = \emptyset \}
\ee
of defects in row $j$, of part $\cS_i(\s)$,  with no left or right-partners cannot have more than one element, and hence the total number of defects without left or right-partners is bounded by $L+1$. 
Every other defect also has at least one up or down-partner by the parity constraint, and the same argument as above implies Inequality \eqref{IneqSplitManyRect}.
Essentially the same considerations also imply Inequality \eqref{IneqSplitManyRect2}: any pair of defects in the same row must have at least one associated  extended rectangle.
\end{proof}

We observe that each piece of a split has at least $100 L$ defects (except possibly the last piece), and splits associated with a high-defect configuration $\sigma$ will have $m(\s) \gtrsim \frac{|D(\sigma)|}{L}$ pieces:

\begin{lemma} [Counting Split Components] \label{LemmaCountingSplits}
Fix $\s \in \pconfig$, and let $\S(\s) = \{s_{i}\}_{i=1}^{m+1}$ be as in Definition \ref{DefBalancedSplits}. Then for $L \geq 100$,
\be
\left\lfloor \frac{|D(\s)|}{ 101\, L} \right\rfloor \leq m(\s) \leq \frac{|D(\s)|}{100\, L} + 1.
\ee

\end{lemma}

\begin{proof}
By inspection,
\be
100 \, L \leq |D(\s) \cap \S_i(\s) | \leq 101 \, L
\ee
for all $1 \leq i \leq m-1$, where the upper bound also holds for $i=m$.
The result follows immediately.
\end{proof}

To reduce notation in the rest of this section we define the occupancy vector of the splits. 
\begin{defn} [Occupancy vector] \label{def:OccVec}
Let $\L = [1:L]^2$, we define the row occupation vector function, $r:\config \times [1:L] \to [0:L]^{L+1}$, by
\[
r_k(\s,i) = \occ{k,\cS_i(\s)}(\s)\,, \quad \textrm{for } \ k \in [0:L]\,, 
\]
so that $r_k(\s,i)$ is the number of defects in row $k$ of part $i$ of the split.

\noindent For any vector $v \in [0:L]^{L+1}$ we write  $\num{v}{k} = |\{i \in [0:L]\,:\,v_i = k\}|$ and, with a slight abuse of notation, $\vmax = \max_{k \in [0:L]}v_k$, for short.
\end{defn}

We now partition the set of possible occupation vectors in order to define our choice of good rectangles.
 The parameter $\theta$ will allow us to approximate which row we started from, having observed an edge, without loosing more than a factor of $O(\beta)$.

\begin{defn}[Good Partition]
  \label{def:goodpart}
	Let $\Theta=[0:\b]$, define $\cV_0 = \{v \in [0:L]^L\,:\, \vmax \leq \b^2\}$ and 
	\begin{align*} 
	\bar\cV_{\theta} &= \{v \in [0:L]^L\,:\, \vmax > \b^2\,,\ \forall k \in [-32,32]\ \ \b\num{v}{\vmax -\theta} \geq \num{v}{V_{\rm max} -\theta -k} \}\,, \quad \textrm{for } \theta \in \Theta\setminus\{0\} \,.
	\end{align*}
	For $\theta \geq 1$, define recursively $\cV_\theta = \bar \cV_\theta \setminus \cup_{i=0}^{\theta -1}\cV_i$.
\end{defn}

\begin{lemma}
	For all $\beta$ sufficiently large, the sets $\cV_0,\cV_1,\ldots,\cV_\b$, partition $[0:L]^\L$.  
\end{lemma}

\begin{proof}
	The sets $\cV_0,\cV_1,\ldots,\cV_\b$ are mutually disjoint by construction. 
	Fix $v \in [0:L]^\L$, if $\vmax \leq \b^2$ then $v \in \cV_0$. 
	On the other hand, if $\vmax > \b^2$ we claim that $\s \in \cup_{i=1}^{\b}\cV_{i}$. We prove this by contradiction: If not, there exists $n \in [0:\vmax]$ such that $\num{v}{n} > \b^{\b/32}$, which is greater than $L$ for all sufficiently large $\beta$.
\end{proof}

The configurations with row occupancy vectors in $\cV_0$ are called \emph{sparse} and will be treated differently from the other configurations.

\begin{defn} [Reconstructable Rectangles] \label{def:ReconRecs}
Fix $\L = [1:L]^2$, for $\sigma \in \config$ let $\S(\s) = \{s_{i}\}_{i=1}^{m+1}$ be the split associated with $\s$. 
For $i \notin [1:m(\sigma)]$, set $F(\sigma,i) = \emptyset$. For $ i \in [1:m(\sigma)]$, let $v = r(\s,i) \in [0:L]^L$, and 
define the set of good rectangles by  
\begin{align*}
	F(\s,i) = \begin{cases}
		\allallext{\s,\cS_i(\s)} & \textrm{if } v \in \cV_0\,,\\
		\bigcup_{j\,:\, v_j = \vmax -\underline{\theta}(v)}\allext{\s,\cS_i(\s)}{j} & \textrm{otherwise,} 
	\end{cases}
\end{align*}
where we define $\underline{\theta}(v)$ to be the unique value of $\theta$ such that $v \in \cV_\theta$.

\end{defn}
In the definition above, if the max row occupation in split $i$ is small then just take all rectangles in split $i$ with defects in at least three vertices. Otherwise, select \emph{only} rectangles with two defects in a row with occupation $\vmax - \underline{\theta}(v)$ and at least one other vertex containing a defect.

We note that if $|D(\s)|$ is sufficiently larger, so that the associated split has more than one part, then the last part of the split (the $m(\s)^{\textrm{th}}$) may contain no good rectangles (i.e. $F(\s,m(\s))$ may be empty). For this reason we avoid choosing this part of the split in the next definition. 

We will say that $\s \in \O_\L$ is \emph{sparse} in $\cS_i(\s)$ if $v \in \cV_0$ and \emph{$\theta$-dense} if $v \in \cV_\theta$ for $\theta >0$.
We now define the short paths which are used to construct the initial part of the full random paths. In words, we pick uniformly a part $i$ of the split from $1$ to $n(\s):=1\vee(m(\s) -1)$, and then independently and uniformly choose a good rectangle to flip from split $i$.  

\begin{defn}[Partial Random Path] \label{DefPartialPath}

For $\sigma \in \Omega_{[1:L]^{2}}$, define the ``partial path" probability measure $\flowtot{\sigma}$ by
\begin{align*}
  \flowtot{\sigma}(\gamma) = \frac{1}{\m(\sigma) \, |F(\sigma,i)|}\,, 
\end{align*}
if $\gamma = \gamma_{\sigma,R}$ for some $R \in F(\sigma,i)$ with $i \in [1:n(\s)]$, and 0 otherwise.\footnote{Note that $F(\sigma,i) \cap F(\sigma,j) = \emptyset$ for $i \neq j$ since $F(\s,i)$ only contains rectangles for which all vertices are inside the $i^\textrm{th}$ part of the split, so this is a probability measure.}
 Recall, here $n(\s):=1\vee(m(\s) -1)$.
\end{defn}

It is necessary to identify the index of the split initially chosen given an edge.
\begin{defn}[Identifying the split index] 
	\label{def:edgesplitfn}
For $e = (e_{-},e_{+}) \in \edges$ let $x(e) = (x_{1},x_{2}) \in \L$ be the single site that is flipped by $e$, so that $e_{+} = e_{-}^{x(e)}$, then define $\Id \, : \, \edges  \mapsto [1:L]$ by
\begin{align}
  \label{eq:Id}
  \Id(e) = i \quad \textrm{s.t.}  \quad  (x_1-1,x_2) \in \S_i(e_-) \,,
\end{align}
which is unique since $\{\cS_j(e_-)\}_{j=1}^{m}$ partitions $\cB(\L)$.
\end{defn}
The function $\Id$ always correctly identifies the index of the initial split chosen, made precise in the following sense.
\begin{lemma} \label{lem:SplitRecov}
	For each edge $e \in \edges$ and path $\gamma \ni e$ such that $\flowtot{\s}(\g)>0$ for some $\s \in \O_\L$, we have $\gamma = \gamma_{\s,R}$ for some $R \in F(\s,\Id(e))$.
\end{lemma}

\begin{proof}
	Fix  $e \in \edges$, let $x = (x_{1},x_{2}) \in \L$ be the single site that is flipped by $e$, and fix $\s$ and a path $\gamma$ such that  $e \in \gamma$ and $\flowtot{\s}(\g) > 0$. By Definition \ref{DefPartialPath} there exists an $i \in [1,n(\s)]$ and an $R \in F(\s,i)$ such that $\g = \g_{\s,R}$.
	By Definition \ref{LemmaTypicalReconstruction} we have $R \subset \cS_i(\s)$.
	Without loss of generality suppose that the top two corners of R are in $D(\s)$ (otherwise we flip the configurations according to the mirror map $M$).
	Label the columns of $\L$ containing the left edge of $R$, the site $x$ and the right edge of $R$ by $\ell,x_1,r$ respectively.
	By construction of $\gamma_{\s,R}$ the defect configurations $p(\s)$ and $p(e_-)$ are equal outside of the columns $\ell,x_1-1,r$ (See Fig. \ref{FigRectRemoveDiff}).
	It follows that $\cS_j(\s) = \cS_j(e_-)$ for $j < i$.
	Furthermore, the number of defects in columns $\ell$ and $r$ in $p(e_-)$ are less than or equal to the number of defects in the same columns of $p(\s)$, see Definition \ref{DefRectRemovePath} and Fig. \ref{FigRectRemoveDiff}.
	In particular $|D(e_-) \cap \bar \Lambda(s_{i}(e_-),x_1-2)| \leq |D(\s) \cap \bar \Lambda(s_{i}(\s),x_1-2)|$, and hence $s_{i+1}(e_-) -1 \geq x_1 -1$ which implies $(x_1-1,x_2) \in \cS_{i}(e_-)$ as required.
\end{proof}

For bounding the congestion of out paths it will be convenient to introduce the initial configurations which are compatible with a given edge.
\begin{defn}[Compatible Initial Configurations]\label{def:Gs}
	For each edge $e \in \edges$ we define the set of all initial configurations $\s \in \config$ such that $\flowtot{\sigma}$ gives positive weight to a path that use the edge $e$
\begin{align}\label{eq:G}
	G(e) = \{\s \in \config\,:\, \exists\, R \in \bigcup_{i=1}^{n(\s)} F(\s,i)\ \ s.t. \ \  e \in \g_{\s,R}\}\,.
\end{align}

We partition $G(e)$ in terms of the edge type and parameter $\theta$.
 Consider $\sigma\in\O_\L$ and a rectangle $R$ with at least three of its four corners in $D(\sigma)$. Assume for now that both of the top two corners of $R$ are in $D(\sigma)$. In this case, set
\begin{equation}
\label{eq:typedef}
\type(e,R,\sigma)=
\begin{cases}
\mathbf{none}, & \quad  e \notin \gamma_{\sigma,R}\,,\\
\mathbf{init}, & \quad  e \in \gamma_{\sigma,R} \text{ and } x \text{ is in the top row of } \cB_{-}(R) \text{ and } \\
& \qquad \qquad \qquad \qquad  \cB_{-}(R) \text{ has more than 1 row} \,,\\
\mathbf{fin}, & \quad e \in \gamma_{\sigma,R} \text{ and } x \text{ is in the last row of } \cB_{-}(R) \, , \\
\mathbf{mid}, & \quad  \text{otherwise.}
\end{cases}
\end{equation}
When the top two corners of $R$ are \textit{not} both in $D(\sigma)$, define the type by flipping all elements according to the mirror reflection map $M$ given in Definition \ref{DefRectRemovePath}. Then 
\be
G_{\theta}^{(\mathrm{U})}(e) &=  \{ \eta \in G(e) \,: \, \exists  R \in F(\eta, \Id(e)) \text{ s.t. } \type(e,R,\eta) = U  \textrm{ and } \underline{\theta}\big(v(\h,\Id(e))\big) = \theta \}\,.
\ee
\end{defn}

We now summarize the required congestion and energy bounds that will be used in the final analysis for the complete paths. 
\begin{prop}\label{prop:PTmain}\label{LemmaTypicalReconstruction}
Let $\L = [1:L]^2$, then:

\begin{enumerate}
\item \textbf{Covering Property:} For all $e \in \edges$
\be 
\label{IneqLemmaTypReconPath}
G(e) = \bigcup\limits_{ \theta \in \Theta}\big(  G_{\theta}^{(\mathrm{init})}(e) \cup G_{\theta}^{(\mathrm{mid})}(e) \cup G_{\theta}^{(\mathrm{fin})}(e)\big) \,.
\ee

\item \textbf{Small Congestion:} For all $e = (e_{-},e_{+}) \in \edges$
\be \label{IneqMainLemmaApprRecBd}
\sum_{\sigma \in G_{\theta}^{(\mathrm{U})}(e)}\sum_{ \g \ni e } \flowtot{\s}(\g) &\lesssim  \frac{\b^4\, L}{|D(e_{-})|}, \quad U \in \{\mathrm{init},\mathrm{mid}\}, \\
\sum_{\sigma \in G_{\theta}^{(\mathrm{fin})}(e)}\sum_{ \g \ni e} \flowtot{\s}(\g)&\lesssim \frac{\b^2\, L^{2}}{|D(e_{-})|}.
\ee


\item \textbf{Many Options:} We have
\be \label{IneqMainCombManyOptions}
F(\sigma,i) &\neq \emptyset\,, \  \textrm{ for all }\  \sigma \in \config\,,\ i \in [1: \m(\sigma)]\ \textrm{ and } \\
\m(\sigma) &\gtrsim \frac{|D(\sigma)|}{L}\,, \  \textrm{ for all }\ \sigma \in \config\,.
\ee

\item \textbf{Energy Bound:}
For all $\sigma \in \config$,  $\theta \in \Theta$ and $e \in \edges$,
\be \label{IneqMainLemmaApprEnergyBd}
\pi(e_{-}) \cL(e_{-}, e_{+}) &\geq e^{-2 \beta} \pi(\sigma), \quad \, \qquad \sigma \in G_{\theta}^{(\mathrm{init})}(e) \cup G_{\theta}^{(\mathrm{mid})}(e) \\
\pi(e_{-}) \cL(e_{-}, e_{+}) &\geq  \pi(\sigma), \qquad \quad \, \qquad \sigma \in G_{\theta}^{(\mathrm{fin})}(e).
\ee

\end{enumerate}
\end{prop}

To give the congestion bounds above we will use the following preliminary results related to the row occupancy vectors. 
\begin{lemma}
	\label{lem:deabs}
  Fix $e \in \edges$, and let $I= \Id(e_-)$ as in Definition \ref{def:edgesplitfn}, define $u \in [0:L]^2$ by $u = r(e_-,I)$, then
   \[
   \sup_{\s \in G(e) }\|r(\s,I)-u\|_\infty \leq 8\,,
   \] 
   	where $r(\s,I)$ is given in Definition \ref{def:OccVec}.
\end{lemma}
\begin{proof}
	Fix $e \in E$ and $\s \in G(e)$, then there exists a rectangle $R$ such that $e \in \gamma_{\s,R}$.
	Any Rectangle-Removal Path (see Definition \ref{DefRectRemovePath}) can change the number of defects in any row of the lattice by at most two, i.e. $|\occ{k}(e_-) - \occ{k}(\s)| \leq 2$ for $k \in [0:L]$. Furthermore the total number of defects can differ by at most six, $\big| |D(e_-)|- |D(\s)|\big|\leq 6$. By the same argument as in the proof of Lemma \ref{lem:SplitRecov} we have $s_{I}(e_-) = s_{I}(\s)$, and since  $\big| |D(e_-)|- |D(\s)|\big|\leq 6$ we know that $S_{I}(e_-)\triangle S_I(\s)$ contains at most six columns which contain defects in $e_-$. Therefore, the difference in the number of defects in any row of the $I^{\mathrm{th}}$ split is bounded above by $2+6=8$.
\end{proof}

The following lemma will be used, in the $\theta$-dense case, to bound the number of rows a rectangle removal path could have started from. 

\begin{lemma}
  \label{LemmaAbsCombLemma}
  For each $u \in [0:L]^L$ and $\theta \in \Theta\setminus\{0\}$, let
  \[
  	\Delta_u^{\theta} =\{v \in \cV_\theta\,:\, \|v-u\|_\infty \leq 8 \}\,,
  \]
  and 
  \begin{align}
  	H_\theta(u) :=   \{i \in [1:L]\,:\, \exists\, v \in \D_u^\theta \ \ \textrm{s.t.}\ \  v_i = \vmax - \theta\}\,,
  \end{align}
  then
  \begin{align}
    \label{eq:Hbound}
  	\left|H_\theta(u)\right| \lesssim \beta \inf_{v \in \Delta_u^\theta} \num{v}{\vmax - \theta}\,.
  \end{align}
\end{lemma}

\begin{proof}
	Fix $u \in [0:L]^L$ and $\theta \in \Theta\setminus\{0\}$. We first show that 
	\[
	H_\theta(u) \subseteq \{i \in [1:L]\,:\, \exists\, \d \in [-16:16]\ \ \textrm{s.t.} \ \ u_i = \umax - \theta + \delta\}\,.
	\]
	Fix $j \in [1:L]$ and suppose there exists a $v \in \D_u^\theta$ such that $v_j = \vmax - \theta$, then $u_j= \vmax - \theta + \d'$ for some $\d' \in [-8:8]$ by definition of $\D_u^\theta$. Similarly $\umax = \vmax + \d''$ for some $\d'' \in [-8:8]$, and putting these together $u_j = \umax - \theta + \d$ for some $\d \in [-16:16]$, proving the claim. We now show that the set on the right hand side of the display above is not too large. Observe that
	\[
	\left| \{i \in [1:L]\,:\, \exists\, \d \in [-16:16]\ \ \textrm{s.t.} \ \ u_i = \umax - \theta + \delta\} \right| = \sum_{\d \in [-16:16]}\num{u}{\umax - \theta + \d}\,,
	\]
	and by the same argument as above, for any $v \in \D_u^\theta$ we have \[\num{u}{\umax - \theta + \d} \leq \sum_{\d' \in [-16:16]}\num{v}{\vmax-\theta + \d + \d'}\,.\]
	 Hence 
	\[
	\left| \{i \in [1:L]\,:\, \exists\, v \in \D_u^\theta \ \ \textrm{s.t.}\ \  v_i = \vmax - \theta\}\right| \leq 256 \sup_{k \in [-32:32]} \num{v}{\vmax-\theta + k}\lesssim \beta \num{v}{\vmax - \theta}\,,
	\]
	where the last inequality follows from the definition of $\cV_\theta$.
\end{proof}

We now give the main bounds on the sizes of $G_{\theta}^{(\mathrm{init})}$, $G_{\theta}^{(\mathrm{mid})}$ and $G_{\theta}^{(\mathrm{fin})}$ used for bounding the congestion in Proposition \ref{prop:PTmain}. We give bounds in two cases, corresponding to initially sparse configurations and initially dense configurations.

\begin{lemma} [Number of Sparse Preimages] \label{LemmaTypAncBound}
We consider the original lattice $ [1\!:\!L]^{2}$. Fix an edge $e = (e_{-},e_{+}) \in \edges$. Let $x = (x_{1},x_{2})$ be the spin that is flipped by edge $e$, so that $e_{+} = e_{-}^{x}$.  Let $I = \Id(e)$.
We have:
\be \label{IneqTypAncBound1}
|\{ (\sigma,R) \in G_{0}^{(\mathrm{fin})}(e) \times \rect_{[1:L]^{2}} \, : \, e \in \gamma_{\sigma,R} \textrm{ and } \flowtot{\s} (\g_{\s,R}) >0 \}| &\lesssim \b^2 L^{2}
\ee
 and
\be \label{IneqTypAncBound2}
|\{ (\sigma,R) \in G_{0}^{(\mathrm{init})}(e) \times \rect_{[1:L]^{2}} \, : \, e \in \gamma_{\sigma,R} \textrm{ and } \flowtot{\s} (\g_{\s,R}) >0  \}|  \lesssim \b^4 L \\
|\{(\sigma,R) \in G_{0}^{(\mathrm{mid})}(e) \times \rect_{[1:L]^{2}} \, : \, e \in \gamma_{\sigma,R} \textrm{ and } \flowtot{\s} (\g_{\s,R}) >0 \}|  \lesssim \b^4 L\,.
\ee
\end{lemma}

\begin{proof}

Before giving a proof, we note that our reconstruction possibilities are illustrated in Figure \ref{FigSparseReconstruction}. The following proof is essentially a detailed explanation of the figure, and the heuristic appearing in the figure's caption does not hide any important details.

\begin{figure}[htb]
\includegraphics[width = 0.8\textwidth]{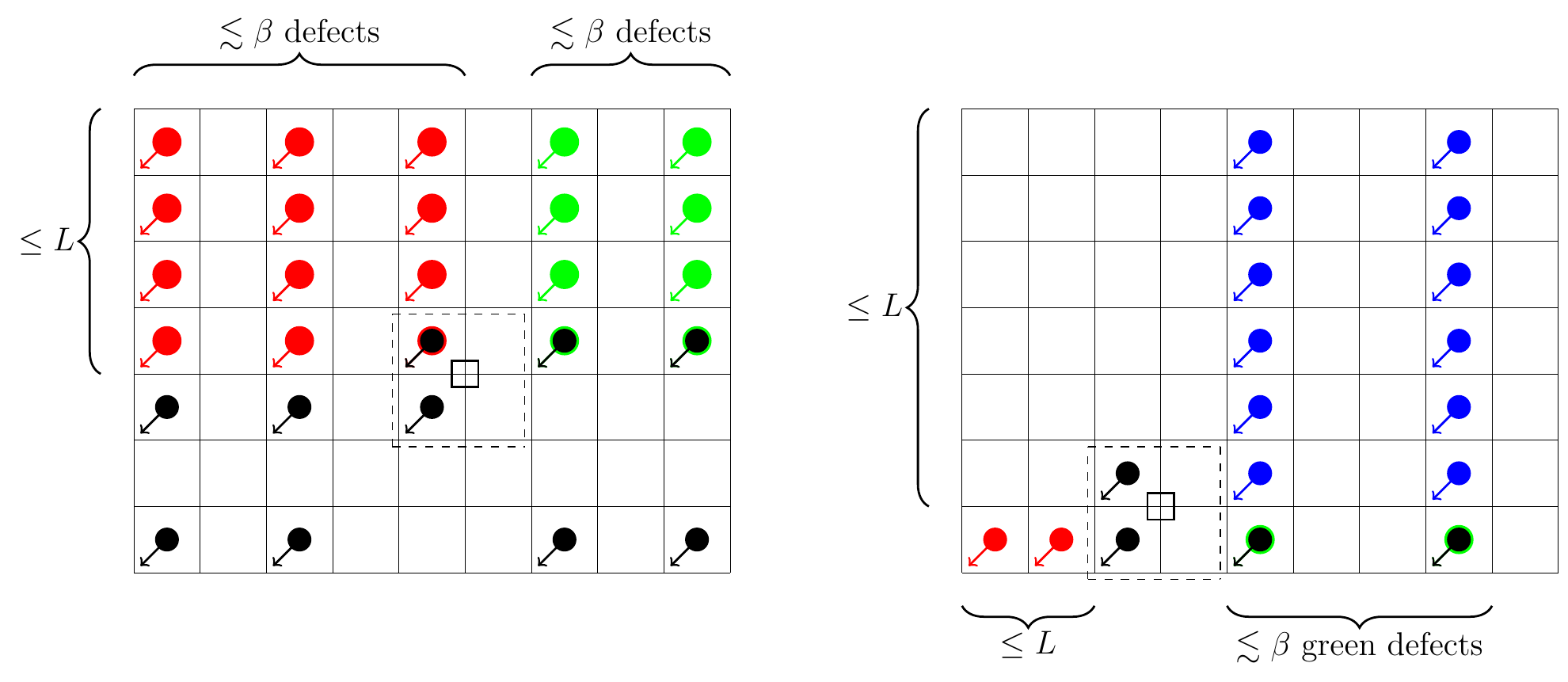}
\caption{\label{FigSparseReconstruction} The defects in $D(e_{-})$ are indicated by black circles, the site $x$ which is flipped in the edge $e$ is marked by a solid box, the dashed box shows the associated defects which are flipped. Left: As in Figure \ref{FigBadConfigNest}, the red and green circles indicate where the initial top defects of the rectangle $R$ in $D(\sigma)$ are allowed to be, for $\mathbf{mid}$ or $\mathbf{int}$ type. Right: Type  $\mathbf{fin}$; the top right corner of $R$ must be above one of the defects of $e_-$, in row $x_2-1$, to the right of $x$ (marked by green), the bottom left corner must be to the left of $x$, in row $x_2-1$ (marked by red circles), and the top right corner must be in one of at most $L$ sites above the bottom left corner.} 
\end{figure}

We begin by proving Inequality \eqref{IneqTypAncBound1}. Let 
\begin{align}
\label{EqPlaceholderSetContainmentSparse}
P = \{ (\sigma,R) \in G_{0}^{(\mathrm{fin})}(e) \times \rect_{[1:L]^{2}} \, : \, e \in \gamma_{\sigma,R} \textrm{ and } \flowtot{\s} (\g_{\s,R}) >0 \}\,.
\end{align}
It is clear from Definition \ref{DefRectRemovePath} that, for any fixed edge $e$ and rectangle $R$, there is at most one $\sigma$ satisfying $e\in \gamma_{\sigma,R}$; denote it $\sigma = \sigma(e,R)$.
Based on this observation, we will find an upper bound on $|P|$ by reconstructing all candidate rectangles $R$ such that $(\sigma(e,R),R) \in P$.

Set $\cS'_I(e) = \L(s_I(e_-),s'_{I+1}(e_-))$, where 
\begin{align}
  \label{eq:Sprime}
  s'_{I+1}(e_-) = \min \{ j > s_{I}(\s) \, : \, |D(e_-) \cap \bar \Lambda(s_{I}(e_-),j-1)| \geq 100 \, L + 6 \}.
\end{align}
Then for any $\s\in G(e)$ with $e \in \gamma_{\sigma,R}$, we have $\S_{I}(\s) \subset \S'_{I}(e)$ and $\cB_{-}(R) \subset \cS'_I(e)$. To reduce notation let $\L = \cS'_I(e)$ for the remainder of this proof.

For a rectangle $R$, denote by $(R^{(u, \ell)}, R^{(u,r)}, R^{(d,r)}, R^{(d,\ell)})$ the 4 corners of $R$ in the usual clockwise order, starting at the upper-left hand corner. Note that it is possible to reconstruct all of $R$ if you know the location of the opposite corners $R^{(u,\ell)}$ and $R^{(d,r)}$. 

 We next introduce notation that is meant to mimic the usual notation for conditional probability. Denote $R^{(a,b)}=(R_1^{(a,b)},R_2^{(a,b)})$ where $a\in\{u,d\}$ and $b\in\{\ell,r\}$.
For
\be
A,B \subset \{ (a,b,c)\}_{a \in \{u,d\}, b \in \{\ell,r\}, c \in \{1,2\}}
\ee
some subsets of the 8 parameters specifying the coordinates of vertices of $R$, denote by $N[ (R^{(a,b)}_{c})_{(a,b,c) \in A} \mid (R^{(a,b)}_{c})_{(a,b,c) \in B}]$ the total number of possible values of the collection $\{R^{(a,b)}_{c}\}_{(a,b,c)\in A}$ such that $(\sigma(e,R),R)\in P$, given that $\{R^{(a,b)}_{c}\}_{(a,b,c)\in B}$ are fixed.
Note that  $N[\cdot | \cdot]$ depends on the edge $e$ and the type, but we do not emphasize this in the notation. By a small abuse of notation, set $N[\cdot] = N[\cdot \mid \emptyset]$.

Note that the configuration $\sigma$ determined by $e$ and  $R$ must have defects in at least three corners of $R$; we assume for the remainder of the proof that it has defects in the upper-left, upper-right, and bottom-left corners. Note that this assumption under-counts the number of valid rectangles by at most a factor of 4, and so it will not affect the order of the final estimate of $|P|$.

We begin with the trivial bound
\be \label{IneqCountTypExtClean1}
N[R^{(u,\ell)}] \leq L^{2},
\ee
since there are at most $L^{2}$ elements of $\L$. Given $R^{(u,\ell)}$, we know that $R^{(d,\ell)}_{1} = R^{(u,\ell)}_{1}$ and $R^{(d,\ell)}_{2} = (x_{2}-1)$, since $\type(e,R,\sigma)=\mathbf{fin}$, so
\be \label{IneqCountTypExtClean2}
N[R^{(d,\ell)} | R^{(u,\ell)}] \leq 1.
\ee

Since $R^{(d,r)}_{2} = R^{(d,\ell)}_{2}$ and $R^{(d,r)} \in D(e_-)$, 
\be \label{IneqCountTypExtClean3}
N[R^{(d,r)} | R^{(d,\ell)}]  \leq \max_{(\sigma,R) \in P} \max_{j \in \mathbb{N}}\occ{j, \S_{I}(\sigma)}(\sigma) + 2 \lesssim \beta^2,
\ee
where  the last inequality follows from the fact that $\ut( v(\sigma,I))=0$ for $(\sigma,R) \in P$.
Combining Inequalities \eqref{IneqCountTypExtClean1}, \eqref{IneqCountTypExtClean2} and \eqref{IneqCountTypExtClean3}, we conclude in this case that
\be
|P| \leq N[(R^{(u,\ell)}, R^{(d,r)})] \lesssim \b^2 L^{2}.
\ee

We now prove Inequality \eqref{IneqTypAncBound2} following the same strategy.
We now define $P$ analogously to Equation \eqref{EqPlaceholderSetContainmentSparse} for type $\mathbf{init}$ or $\mathbf{mid}$.
We begin with the trivial bound $N[R^{(u,\ell)}_{2}] \leq L$, since the height of $\L$ is $L$. Next, note that the column containing $R^{(u,\ell)}$ must intersect the row $(x_2-1)$ at either a defect of $e_{-}$ or at $(x_1-1,x_2-1)$. 
Since there are at most $\max_{j \in \mathbb{N}}\occ{j, \S_{I}(\sigma)}(\sigma) + 2 \lesssim \beta^2$ defects in any row of $e_{-}$ for any $(\sigma,R) \in P$, this observation implies $N[R^{(u,\ell)}_{1} | R^{(u,\ell)}_{2}] \lesssim \b^2$, and so
\be \label{IneqCountTypExtCut1}
N[R^{(u,\ell)}] \lesssim \beta^2 L.
\ee

Inequality \eqref{IneqCountTypExtClean2} also holds in this case, though the argument is slightly different, as follows. Inspecting Equation \eqref{EqAllallextdef} from Definition \ref{DefTriples}, we see that if $\hat{p} \in \Omega_{[\ell_{1}:\ell_{2}]\times[\ell_{3} \times \ell_{4}]}$ is a defect configuration and $\hat{R}$ is a rectangle with three corners in $\allallext{\hat{p}}$, then $\hat{p}$ has no defects in the line between $\hat{R}^{(u,\ell)}$ and $\hat{R}^{(d,\ell)}$. Therefore, $R^{(d,\ell)}$ must be the highest defect of $p$ that is in the same column as $R^{(u,\ell)}$ and below the row $(x_{2}-1)$, and so Inequality \eqref{IneqCountTypExtClean2} holds.
Inequality \eqref{IneqCountTypExtClean3} holds in this case as well, thus combining Inequalities \eqref{IneqCountTypExtCut1}, \eqref{IneqCountTypExtClean2} and \eqref{IneqCountTypExtClean3} we have
\be
|P| \leq N[R^{(u,\ell)}, R^{(d,r)}] \lesssim  \beta^{4} L.
\ee
This completes the proof of Inequality \eqref{IneqTypAncBound2}.
\end{proof}

We now prove the analogous result for the remaining configurations:

\begin{lemma} [Number of Dense Preimages] \label{LemmaATypAncBound}
We consider the original lattice $[1:L]^{2}$. Fix $\theta \in \Theta\setminus\{0\}$ and an edge $e = (e_{-},e_{+}) \in \edges$, let $x = (x_{1},x_{2})$ be the spin that is flipped by edge $e$, and let $u = r(e_-,\Id(e_-))$ with $\umax = \max_{k\in [0:k]}u_k$.
Then we have:
\be \label{IneqAtypAncBound2}
|\{ (\sigma,R) \in G_{\theta}^{(\mathrm{init})}(e) \times \rect_{[1:L]^{2}} \, : \, e \in \gamma_{\sigma,R} \textrm{ and } \flowtot{\s} (\g_{\s,R}) >0 \}|  &\lesssim   \umax^{2} \, |H_{\theta}(u)|\,, \\
|\{  (\sigma,R) \in G_{\theta}^{(\mathrm{mid})}(e) \times \rect_{[1:L]^{2}} \, : \, e \in \gamma_{\sigma,R} \textrm{ and } \flowtot{\s} (\g_{\s,R}) >0  \}|  &\lesssim  \umax^{2} \, |H_{\theta}(u)|\,,  \\
\ee
and
\be \label{IneqAtypAncBound1}
|\{  (\sigma,R) \in G_{\theta}^{(\mathrm{fin})}(e) \times \rect_{[1:L]^{2}} \, : \, e \in \gamma_{\sigma,R} \textrm{ and } \flowtot{\s} (\g_{\s,R}) >0   \}| \lesssim  L  \, \umax \, |H_{\theta}(u)|\,.
\ee
\end{lemma}

\begin{proof}

The proof will be very similar to the proof of Lemma \ref{LemmaTypAncBound}, and the overall picture is somewhat similar to Figure \ref{FigSparseReconstruction}. The main difference is that, in the previous lemma and figure, we bounded the number of locations for the top row of the rectangle $R$ using the trivial upper bound of $L$. Here, we use the much better bound $|H_{\theta}(u)|$. Since any nontrivial dense configuration is quite complicated, we do not include a new figure in this case.

 We begin by proving Inequality \eqref{IneqAtypAncBound2}.
We set notation, fixing $U \in \{ \mathrm{mid}, \mathrm{init} \}$ and then:
\be 
P &= \{ (\sigma,R) \in G_{\theta}^{(U)}(e) \times \rect_{[1:L]^{2}} \, : \, e \in \gamma_{\sigma,R} \textrm{ and } \flowtot{\s} (\g_{\s,R}) >0 \}\,.
\ee
Again, we bound the number of elements $(\sigma,R) \in P$ by counting the number of ways to reconstructing a suitable rectangle $R$ from our knowledge of $e$.

Set $\L = \S_{I}'(e)$ as defined in the proof  of Lemma \ref{LemmaTypAncBound}. Also, define $R, N[\cdot | \cdot]$ as in the proof of Lemma \ref{LemmaTypAncBound}.
By the same argument as in Lemma \ref{LemmaTypAncBound}, leading to Inequality \eqref{IneqCountTypExtCut1}, we know that $R^{(u,\ell)}$ has to be in a column that either
\begin{enumerate}
\item is $x_1-1$, or
\item contains some $y = (y_{1},y_{2}) \in D(e_{-})$ with $y_{2} = x_{2}-1$.
\end{enumerate}
Since no row of $\L$ contains more than $\umax+8$ defects in $e_{-}$, this implies
\be \label{IneqCountAtypExtInt1prev}
N[R^{(u,\ell)}_{1}] \lesssim  \umax.
\ee
On the other hand, by  the definition of $F(\s,I)$ together with Lemmas \ref{lem:deabs} and \ref{LemmaAbsCombLemma}, the row containing $R^{(u,\ell)}$ must be in $H_{\theta}(u)$. Thus,
\be \label{IneqCountAtypExtInt1}
N[R^{(u,\ell)}_{2} | R^{(u,\ell)}_{1}]  \lesssim |H_{\theta}(u)|.
\ee
Combining Inequalities \eqref{IneqCountAtypExtInt1prev} and \eqref{IneqCountAtypExtInt1}, we have
\be \label{IneqCountAtypExtCut1}
N[R^{(u,\ell)}] \lesssim  \umax \, |H_{\theta}(u)|.
\ee
Next, $R^{(u,r)}$ must be in the same row as $R^{(u,\ell)}$, and must be in either
\begin{enumerate}
\item column $x_{1}$, or
\item a column containing an element $y = (y_{1},y_{2}) \in D(e_{-})$ with $y_{2} = x_{2}$.\footnote{It is not an accident that the value of $y_{2}$ is off-by-one from our analysis of the location of $R^{(u,\ell)}$; see Figure \ref{FigSparseReconstruction}.}
\end{enumerate}
\noindent
Thus,
\be \label{IneqCountAtypExtCut2}
N[R^{(u,r)} | R^{(u,\ell)}] \lesssim \umax\,.
\ee
\noindent

Finally, Inequality \eqref{IneqCountTypExtClean2} holds by the same arguments as in the previous lemma (recall that this inequality had two proofs - one if $U = \mathrm{fin}$, another if $U \in \{ \mathrm{mid}, \mathrm{init} \}$). Thus, combining Inequalities \eqref{IneqCountAtypExtCut1}, \eqref{IneqCountAtypExtCut2} and \eqref{IneqCountTypExtClean2},

\be
|P| \leq N[R^{(u,\ell)}, R^{(d,r)}] \lesssim \umax^{2} \, |H_{\theta}(u)|.
\ee
\noindent
This completes the proof of  Inequality \eqref{IneqAtypAncBound2}.

Next, we prove  Inequality \eqref{IneqAtypAncBound1} using a similar argument, again we redefine $P$ accordingly.
Inequality \eqref{IneqCountAtypExtInt1prev} does not hold in this case, since we have already removed the defect that was originally in the bottom left corner of $R$, so we use instead the weaker trivial bound
\be
N[R^{(u,\ell)}_{1}] \leq L,
\ee
since there are at most $L$ columns in $\L$. Next, noting that Inequality \eqref{IneqCountAtypExtInt1} holds with the same argument as above, this implies
\be \label{IneqCountAtypExtClean1}
N[R^{(u,\ell)}] \lesssim L \, |H_{\theta}(u)|.
\ee
Inequalities \eqref{IneqCountAtypExtCut2} and \eqref{IneqCountTypExtClean2} hold by the same arguments given above. Thus, combining Inequalities \eqref{IneqCountAtypExtClean1}, \eqref{IneqCountAtypExtCut2} and \eqref{IneqCountTypExtClean2}, we conclude
\be
|P| \leq N[R^{(u,\ell)}, R^{(d,r)}] \lesssim  L \,\umax \, |H_{\theta}(u)|.
\ee
This completes the proof of the lemma.
\end{proof}

As our final preliminary congestion estimate, we bound the number of choices given by the function $F$:

\begin{lemma}  \label{LemmaEdgeWeight}
Let $\Lambda = [1:L]^{2}$ and fix $\sigma \in \config$. Let $p = p^{+}(\sigma)$, fix $i \in [1:n(\s)]$ and define $v = r(\s,i)$.  Then we have the following bounds on $| F(\sigma,i)|$:

\begin{enumerate}
\item If $\ut(v) =0 $, then
\be \label{IneqEdgeWeight1}
| F(\sigma,i)| \gtrsim  |D(p|_{\S_{i}(\s)})|.
\ee
\item If $\ut(v) \in \Theta\setminus\{0\}$, then
\be \label{IneqEdgeWeight2}
|F(\sigma,i)| \gtrsim (\vmax - \ut(v))^{2} \, \num{v}{\vmax-\ut(v)}\,.
\ee
\end{enumerate}
\end{lemma}

\begin{proof}
Inequality \eqref{IneqEdgeWeight1} follows immediately from Inequality \eqref{IneqManyRect} and \eqref{IneqSplitManyRect} in Lemma \ref{LemmaSplitsManyRect}. Inequality \eqref{IneqEdgeWeight2} follows immediately from \eqref{IneqSplitManyRect2} in Lemma \ref{LemmaSplitsManyRect}.
\end{proof}

Having completed the preliminary congestion bounds we now prove Proposition \ref{LemmaTypicalReconstruction}:

\begin{proof} [Proof of Proposition \ref{LemmaTypicalReconstruction}]
	Properties \textbf{(1)}, and \textbf{(3)} are almost immediate.
	Property \textbf{(1)}, given in  \eqref{IneqLemmaTypReconPath}, follows immediately from Definition \ref{def:Gs}.  
        Property \textbf{(3)}, given in Equation \eqref{IneqMainCombManyOptions}, follows from the definition of $F$ in Definition \ref{def:ReconRecs} and Lemma \ref{LemmaCountingSplits} recalling $n(\s) = 1 \vee (m(\s)-1)$. 
	Thus, it only remains to prove Properties \textbf{(2)} and \textbf{(4)}. 
	
	We prove the last line of Inequality \eqref{IneqMainLemmaApprRecBd}. 
        Fix $e \in \edges$ and  let $x = (x_{1},x_{2})$ be the spin that is flipped by edge $e$ (so that $e_{+} = e_{-}^{x}$) and fix $\theta \in \Theta\setminus\{0\}$ such that $G_{\theta}^{(\mathrm{fin})}$ is not empty, finally let $i = \Id(e)$ and $u=r(e_-,i)$.
	Recalling \eqref{IneqTypAncBound1} of Lemma \ref{LemmaTypAncBound} and Inequality \eqref{IneqAtypAncBound1} of Lemma \ref{LemmaATypAncBound},
	\be \label{IneqBoundSparsePreimageTyp1}
	|\{ (\sigma,R) \in G_{0}^{(\mathrm{fin})}(e) \times \rect_{[1:L]^{2}} \, : \, e \in \gamma_{\sigma,R} \textrm{ and } \flowtot{\s} (\g_{\s,R}) >0 \}| &\lesssim \b^2 L^{2}\,, \\
        |\{  (\sigma,R) \in G_{\theta}^{(\mathrm{fin})}(e) \times \rect_{[1:L]^{2}} \, : \, e \in \gamma_{\sigma,R} \textrm{ and } \flowtot{\s} (\g_{\s,R}) >0   \}| &\lesssim  L  \, \umax \, |H_{\theta}(u)|\,.
	\ee
	On the other hand, consider $\sigma \in  G_{0}^{(\mathrm{fin})}(e)$, by Lemma \ref{LemmaEdgeWeight} and the fact that each element of the split (except possibly the last when $m(\s) > 1$) contains at least $100\,L$ defects, we have
	\be \label{IneqBoundSparsePreimageTyp1Rev1}
	| F(\sigma,i)| \gtrsim \min(L, |D(e_{-})|), \qquad \sigma \in  G_{0}^{(\mathrm{fin})}(e)\,,
	\ee 
        where we used the observation $|D(\s)| \asymp |D(e_-)|$, since any rectangle-removal path creates or removes at most six defects.
	When $\sigma \in  G_{\theta}^{(\mathrm{fin})}(e)$ with $\theta \in [1:\beta]$, we have 
	\be \label{IneqBoundSparsePreimageTyp1Rev2}
	| F(\sigma,i)| &\stackrel{\text{Lemma } \ref{LemmaEdgeWeight}}{\gtrsim}   (\vmax-\theta)^{2} \, \num{v}{\vmax-\theta}
	\gtrsim \umax^2 \, \num{v}{\vmax-\theta} \\
	&\stackrel{\text{Ineq. } \eqref{eq:Hbound}}{\gtrsim} \umax^{2} \, \beta^{-1} |H_{\theta}(u)|\,,
	\ee
        where in the second inequality we used $\vmax \geq \b^2$, $\vmax > \umax - 16$ and $\theta \leq \beta$.
	
	We consider the initially sparse and $\theta$-dense cases separately:
        \begin{enumerate}
		\item \textbf{Sparse initial configurations}. Fix $\t \in  G_{0}^{(\mathrm{fin})}(e)$, in this case, Inequalities \eqref{IneqBoundSparsePreimageTyp1} and \eqref{IneqBoundSparsePreimageTyp1Rev1}  give
		\be
		\frac{|\{ (\sigma,R) \in G_{0}^{(\mathrm{fin})}(e) \times \rect_{[1:L]^{2}} \, : \, e \in \gamma_{\sigma,R} \textrm{ and } \flowtot{\s} (\g_{\s,R}) >0 \}|}{F(\t,i)}  &\leq \frac{\beta^2 L^{2}}{\min(L,|D(e_{-})|)}\,.
		\ee

		\item \textbf{$\theta$-dense initial configurations}. Fix $\t \in  G_{\theta}^{(\mathrm{fin})}(e)$, in this case, Inequalities \eqref{IneqBoundSparsePreimageTyp1} and \eqref{IneqBoundSparsePreimageTyp1Rev2} give
		\be
		\frac{|\{  (\sigma,R) \in G_{\theta}^{(\mathrm{fin})}(e) \times \rect_{[1:L]^{2}} \, : \, e \in \gamma_{\sigma,R} \textrm{ and } \flowtot{\s} (\g_{\s,R}) >0   \}| }{ | F(\tau,i)|} &\lesssim \frac{L  \,\umax \, |H_{\theta}(e)|}{\beta^{-1} \umax^{2} \,  |H_{\theta}(e)|} \\
		&\lesssim \beta^{-1} L \leq  \frac{\beta^2L^{2}}{\min(L,|D(e_{-})|)}.
		\ee
		
	\end{enumerate}
	Combining these two cases,
	\be
	\sum_{\sigma \in G_{\theta}^{(\mathrm{fin})}(e)} \sum_{R\,:\,e \in \g_{\s,R}}  \frac{1}{|F(\sigma,i)|} &\leq \sum_{\sigma \in G_0^{(\mathrm{fin})}(e)}\sum_{R\,:\,e \in \g_{\s,R}} \frac{1}{|F(\sigma,i)|} + \sum_{\sigma \in G_\theta^{(\mathrm{fin})}(e)}\sum_{R\,:\,e \in \g_{\s,R}}  \frac{1}{|F(\sigma,i)|} \\
	&\lesssim  \frac{\beta^2 L^{2}}{\min(L,|D(e_{-})|)}.
	\ee
        If we consider initial configurations $\s$ such that $|D(\s)| \lesssim L$ and apply the trivial bound $n(\s) \geq 1$ we get the desired result by the definition of $\flowtot{\s}$.
	Finally considering the initial configurations $\s$ such that $|D(\s)| \gtrsim L$, in this case we apply \textbf{(3)} proved above to get the desired bound.
        This completes the proof of the last line of Inequality \eqref{IneqMainLemmaApprRecBd}. 

        The proof of the other two inequalities in \eqref{IneqMainLemmaApprRecBd} is essentially identical, with the following simple changes to the two cases:
	\begin{enumerate}
		\item Replace references to Inequality \eqref{IneqTypAncBound1} of Lemma \ref{LemmaTypAncBound} with references to Inequality \eqref{IneqTypAncBound2} of Lemma \ref{LemmaTypAncBound} (the upper bound 
		is better by a factor of $\beta^2 L^{-1}$).
		\item Replace references to Inequality \eqref{IneqAtypAncBound1} of Lemma \ref{LemmaATypAncBound} with references to Inequality \eqref{IneqAtypAncBound2} of Lemma \ref{LemmaATypAncBound} (in this case the upper bound is better by a factor of $L^{-1}\umax$).
	\end{enumerate}

	Finally, we prove Property \textbf{(4)}, the energy bound.
	Inequality \eqref{IneqMainLemmaApprEnergyBd} can be observed directly from the definition of the rectangle-removal paths in Definition \ref{DefRectRemovePath}. We find it easiest to see this by simply following along in Figure \ref{FigRectRemoveDiff}; we now give a short description of what can occur. Begin by assuming there are no defects in the interior of the rectangle. In this case the first spin that is ``flipped" in any row besides the last one will create a new pair of vertically adjacent defects that travel across the row, and the last spin that is ``flipped" in the row will delete both elements of this pair. This proves the second line of Inequality \eqref{IneqMainLemmaApprEnergyBd} in this case. The first line of Inequality \eqref{IneqMainLemmaApprEnergyBd} follows in this case from noticing that the last row of the rectangle will \textit{already} have a pair of vertically aligned defects, and so no new defects are created. Finally, it is straightforward to see that any defects in the \textit{interior} of the rectangle can only decrease the number of excess defects created during a rectangle-removal path.
	This completes the proof of Proposition \ref{LemmaTypicalReconstruction}.
\end{proof}

\subsection{Path Construction: High Density} \label{SubsecPathHighDef}

The first path in this section will be appended to previously-defined paths after roughly $L$ defects have been removed, as it allows us to give a good bound on the total path length (at the cost of an enormous penalty to both congestion and energy). The path simply flips all $-1$ spins to $+1$ spins, in order:

\begin{defn}[Naive Paths] \label{DefNaivePath}
For fixed $\sigma \in \Omega_{[1:L]^{2}}$, we define a ($\{0,1\}$-valued) probability mass $\flowna{\sigma}$ on $\Gamma_{\sigma, +}$ by giving an explicit algorithm for sampling from the measure. Let $z^{(1)},  z^{(2)}, \ldots$ be all points of $[1:L]^{2}$ in lexicographic order.

\begin{enumerate}
\item Initialize by setting $\gamma = \{\sigma\}$ and $i=0$.
\item While $\fin{\gamma} \neq +$, do the following:
\begin{enumerate}
\item Set $i = i +1$.
\item If $\fin{\gamma}_{z^{(i)}} = -1$, do the following:
\begin{enumerate}
\item Set $\sigma = \fin{\gamma}$.
\item Set $\gamma = \gamma \conc ( \fin{\gamma}, \sigma^{z^{(i)}})$.
\end{enumerate}
\end{enumerate}
\item Return the path $\gamma$.
\end{enumerate}
\end{defn}

We then define the complete measure on paths that will be analyzed in the remainder of this paper:

\begin{defn} [Full Random Path] \label{DefFinalPath}

For fixed $\sigma \in \Omega_{[1:L]^{2}}$, we define a probability measure $\flowfull{\sigma}$ on $\Gamma_{\sigma,+}$ by giving an explicit algorithm for sampling from this measure:

\begin{enumerate}
\item Initialize by setting $\gamma = \{\sigma\}$ and $i=1$.
\item While $D(\fin{\gamma}) \neq \emptyset$, iteratively sample subpaths $\gamma_{1},\gamma_{2},\ldots$ according to the following loop.
\begin{enumerate}
\item If $i \leq \cnaive L$, sample $\gamma_{i} \sim \flowtot{\fin{\gamma}}$.
\item If $i > \cnaive L$, sample $\gamma_{i} \sim \flowna{\fin{\gamma}}$, according to Definition \ref{DefNaivePath}.
\item In both of these cases, set $\gamma = \gamma \conc \gamma_{i}$ and then $i = i+1$.
\end{enumerate}
\item Once $\gamma$ satisfies $D(\fin{\gamma}) = \emptyset$, return the path $\gamma$.
\end{enumerate}

When bounding the mixing time, it is useful to consider truncated paths as follows. Fix a truncation level $0 \leq k \leq L^{2}$. We define the measure $\flowtrunc{\sigma}{k}$ on $\Gamma$ by the following algorithm:
\begin{enumerate}
\item Sample the path $(\sigma^{(0)},\ldots,\sigma^{(m)}) \sim \flowfull{\sigma}$.
\item Let $i_{\min} = \min \{i \, : \, |D(\sigma^{(i)})| \leq k\}$.
\item Return the path $\gamma = (\sigma^{(0)},\ldots,\sigma^{(i_{\min})})$.
\end{enumerate}
\end{defn}

\subsection{Path Analysis: Bounds on Congestion and Path Length} \label{SecPathAnalysisLemmas}

We analyze the random path constructed in Definition \ref{DefFinalPath}. We begin by bounding the congestion associated with the subpath given in Definition \ref{DefNaivePath}. We will then bound the congestion associated to entire paths.

\begin{defn} [Compatible Naive Paths]

Fix an edge $e = (e_{-},e_{+}) \in \edges$. We now define the set $\naiveanc{e} \subset \Gamma$ of \textit{compatible naive paths} associated with edge $e$  as follows. Say $\gamma \in \naiveanc{e}$ if it satisfies both of the following:

\begin{enumerate}
\item $e \in \gamma$, and
\item there exists a configuration $\sigma \in \Omega_{[1:L]^{2}}$ so that $\flowna{\sigma}(\gamma) > 0$.
\end{enumerate}

\end{defn}

We now bound the sizes of these sets of compatible paths. Recall the naive paths are deterministic, so $\flowna{\sigma}(\gamma) > 0$ for exactly one path $\gamma$. 

\begin{lemma} [Number of Compatible Naive Paths] \label{LemmaNumAncFullNaivePath}
Fix an edge $e \in \edges$ and integer $k \in \mathbb{N}$. With notation as above,
\be \label{IneqNumAncInterval}
 | \{ \gamma \in \naiveanc{e} \, : \,  |D(\init{\gamma})| = |D(e_{-})| +k \} | \leq  2^{2L}\sum_{m=(k-2L)^+}^{k+2L} \binom{ L^{2}}{ m}\,,
\ee
where $(k-2L)^+= \max\{k-2L,0\}$. 
\end{lemma}

\begin{proof}

Figure \ref{FignaiveEnergyBLAG} and its caption give an informal proof; a formal proof follows.

\begin{figure}[h]
\includegraphics[width = 0.4\textwidth]{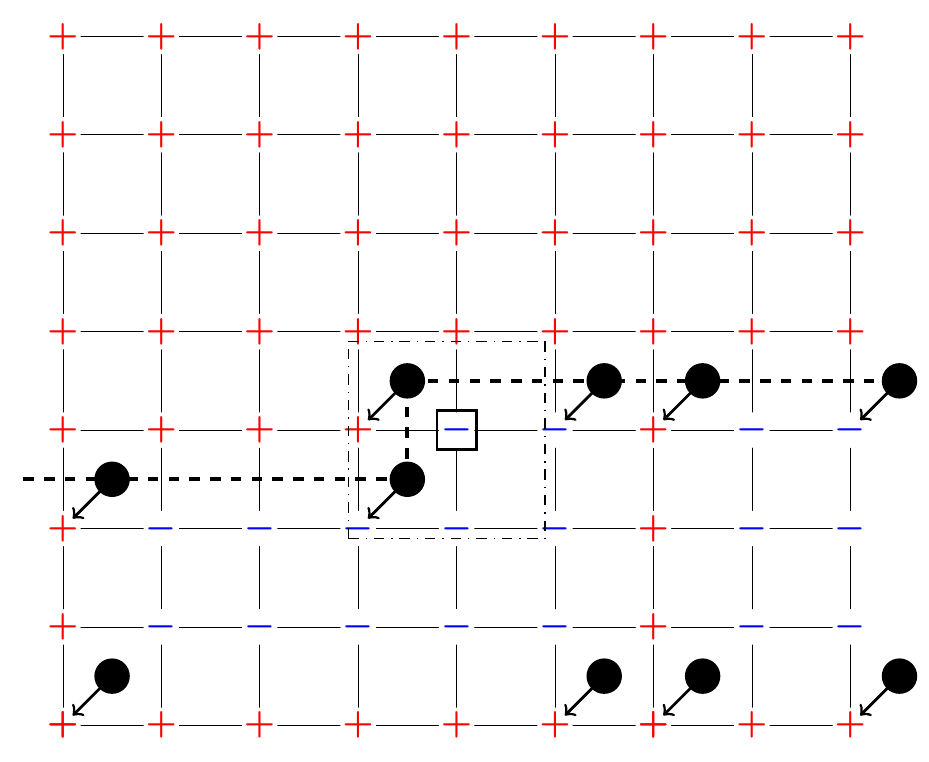}
\caption{\label{FignaiveEnergyBLAG} Reconstruction of $\sigma$ from an observed edge $e$. Defects in $e_-$ are shown by black circles. The small black square indicates the site $x$ where the spin is flipped in $e$. The four defects that are flipped by the edge $e$ are contained in the dot-dashed box. The region above the thick dashed line is all $+$ under $e_{-}$ (and in particular has no defects), and could have had any values in the initial configuration $\sigma$ (so that the total number of defects is $|D(e_-)| +k$). The  $e_{-}$ and $\sigma$ have the same spin configuration below the dashed line, and in particular have the same number of defects in this region excluding those on the dashed line, so there are no choices here. The dashed line intersects at most $(L+2)$ plaquettes, and in principle any of these could contain defects in $e_{-}$ that do not appear in $\sigma$ (in particular there may be up to $(L+2)$ extra defects here and thus $2^{L+2}$ choices). }
\end{figure}

More formally, let $z^{(1)}, z^{(2)}, \ldots$ be the elements of $\Lambda = [1:L]^{2}$ in lexicographic order. Recall that $e_{-} = e_{+}^{z^{(i)}}$ for some $z^{(i)} = (z_{1}^{(i)}, z_{2}^{(i)}) \in \Lambda$ and $1 \leq i \leq |\Lambda|$. Let $\gamma \in \naiveanc{e}$, and let $\sigma = \init{\gamma}$. We note that $\gamma$ is determined entirely by $\sigma$ (since the measure $\flowna{\sigma}$ puts all mass on a single path). Thus, it is enough to bound the possible choices of $\sigma$. We note the following about the possible choices for $\sigma$:

\begin{enumerate}
\item We must have $\sigma(z^{(j)}) = e_{-}(z^{(j)})$ for all $j \geq i$.
\item Knowledge of $e$ gives no obvious restrictions on $\sigma(z^{(j)})$ for $j < i$.
\end{enumerate}
This gives rise to the following restrictions on the locations of the defects of $\sigma$,
\begin{enumerate}
\item They must agree with the defects of $e_{-}$ for all sites $\{ z^{(k)} \, : \,  k > i + L +1\}$. These are the sites below the boundary ``strip" in Figure  \ref{FignaiveEnergyBLAG}.
\item The defects in the set $\{ z^{(k)} \, : \, k < i-1\}$ of sites above the boundary ``strip" can be in any position.
\end{enumerate}

There are also some complicated restrictions on the strip $\{ z^{(k)} \, : \, i-1 \leq k \leq i + L + 1\}$  itself, but we will ignore these restrictions (at the cost of over-counting the true number of possible configurations $\sigma$).
Thus, to create a candidate compatible path starting at ancestor configuration $\sigma$ with $ |D(\sigma)| = |D(e_{-})|+k$, we can:

\begin{enumerate}
\item Place defects in the set $\{ z^{(k)} \, : \,  k > i + L +1\}$ below the horizontal strip so that they agree with $e_{-}$ (there is one way to do this).
\item Place defects on the horizontal strip $\{ z^{(k)} \, : \, i-1 \leq k \leq i + L + 1\}$ in any way (there are $(L+2)$ sites on the strip, so at most $2^{L+2}\leq 2^{2L}$ ways to do this).
\item Place the remaining defects anywhere in the set $\{ z^{(k)} \, : \, k < i-1\}$ above the strip (there are between $(k-2L)^+$ and $k+2L$ defects to place after choosing the configuration on the strip, and so at most $\sum_{m=(k-2L)^+}^{k+2L} \binom{ L^{2}}{ m}$ ways to do this).
\end{enumerate}
Not all candidate ancestors are in fact possible ancestors, but this gives the desired upper bound on the number of possible ancestors (and thus the same upper bound on the number of possible compatible paths).

\end{proof}

Next, we bound the number of compatible paths associated with a full path segment $\gamma_{\sigma,R}$  obtained from the function $F$ in Proposition \ref{LemmaTypicalReconstruction}. We note that any such allowed path $\gamma_{\sigma,R}$ will certainly have $|D(\gamma_{\sigma,R}^{(\mathrm{fin})})| = |D(\gamma_{\sigma,R}^{(\mathrm{init})})| - k$ for $k \in \{2,4\}$, motivating the definition:

\begin{defn} [Long Compatible Non-Naive Paths] \label{DefAncNonNaive}

Fix $\eta \in \Omega_{[1:L]^{2}}$ and $k \in \{2,4\}$. We now define the set $\longanc{k}{\eta} \subset \Gamma$ of \textit{long compatible typical paths} associated with $\eta$ as follows. Say $\gamma \in \longanc{k}{\eta}$ if it satisfies all of the following:

\begin{enumerate}
\item $\fin{\gamma} = \eta$, and
\item $|D(\init{\gamma})| = |D(\eta)| + k$, and
\item There exists $\sigma \in \Omega_{[1:L]^{2}}$ such that $\flowtot{\sigma}(\gamma)>0$.
\end{enumerate}

\end{defn}

We now bound the number of compatible paths for a given configuration:

\begin{lemma} [Number of Long Compatible Non-Naive Paths] \label{LemmaAncestryLongPathBound}

For fixed $\eta \in \Omega_{[1:L]^{2}}$, we have the inequalities

\be \label{IneqAncestryLongPathMainBound}
|\longanc{2}{\eta}| &\leq |D(\eta)| \, (L^{2} - |D(\eta)|)  \leq \frac{L^{4}}{4} \\
|\longanc{4}{\eta}| &\leq (L^{2} - |D(\eta)|)^{2} \leq L^{4} \\
\ee
\end{lemma}

\begin{proof}
 Let $\gamma \in \longanc{2}{\eta}$. Note that, to reconstruct $\gamma$ from $\eta$, it is enough to reconstruct the rectangle $R \subset [1:L]^{2}$ that satisfies $\gamma = \gamma_{\sigma,R}$ for some $\sigma \in \Omega_{[1:L]^{2}}$ (note that knowledge of $R$ and $\eta$ allows you to reconstruct $\sigma$ uniquely). Let $r^{(\mathrm{init})}$, $r^{(\mathrm{fin})}$ be, respectively, the first and last elements of $R$ that are flipped while traversing $\gamma$. Since only two defects are removed over the course of $\gamma$, we must have that $r^{(\mathrm{fin})} \in D(\eta)$ and $r^{(\mathrm{init})} \notin D(\eta)$. Thus, there are at most $(L^{2} - |D(\eta)|)$ choices for  $r^{(\mathrm{init})}$ and at most $|D(\eta)|$ choices for $r^{(\mathrm{fin})}$. Since the choice of two opposite corners of $R$ determines $R$ completely, this completes the proof of the first line of Inequality \eqref{IneqAncestryLongPathMainBound}. The proof of the second line is essentially identical.
\end{proof}

Having bounded the congestion of the path types, we bound the path lengths.

\begin{lemma} [Path Length] \label{LemmaPathLengthBound}
Fix $\sigma \in \Omega_{[1:L]^{2}}$ and let $\gamma \in \Gamma_{\sigma,+}$ satisfy $\flowfull{\sigma}(\gamma) > 0$. Then the path length is bounded by
\be
|\gamma| \leq L^{2} \min \left( \cnaive L + 1, \frac{|D(\sigma)|}{2} \right).
\ee
\end{lemma}

\begin{proof}
Recall that the (non-truncated) paths constructed in Definition \ref{DefFinalPath} are given in terms of a decomposition of each path $\gamma$ into subpaths: 
\be
\gamma = \gamma_{1} \conc \ldots \conc \gamma_{m}
\ee
for some $m$. Note that for all $i < m$, we have $\gamma_{i}$ is of the form $\gamma_{\gamma_{i-1}^{(\mathrm{fin})}, R_{i}}$ for some rectangle  $R_{i}$ (with the convention $\gamma_{0}^{(\mathrm{fin})} = \sigma$). Recalling that $|D(\fin{\gamma_{i}})| \leq |D(\init{\gamma_{i}})| - 2$, this implies  $m \leq \frac{|D(\sigma)|}{2}$. Recalling from Definition \ref{DefFinalPath} that we start a naive path of the form in Definition \ref{DefNaivePath} after at most $\cnaive L$ rectangle-removal paths of the form Definition \ref{DefRectRemovePath}, this implies $m \leq \min( \cnaive L + 1, \frac{|D(\sigma)|}{2})$.

Recalling that each path component $\gamma_{i}$ consists of flipping all the spins in a single rectangle $R_{i} \subset [1:L]^{2}$ at most once, we have $|\gamma_{i}| \leq |R_{i}| \leq |[1:L]^{2}| = L^{2}$. Thus,
\be
|\gamma| \leq \sum_{i=1}^{m} |\gamma_{i}| \leq m L^{2} \leq  L^{2} \min( \cnaive L + 1, \frac{|D(\sigma)|}{2}).
\ee

\end{proof}

\subsection{Final Bounds on Canonical Paths} \label{SecFinalBoundCanonical}

In this section, we put together our main bounds on the canonical paths studied in this paper. We let $\L = [1:L]^{2}$. The calculation is somewhat lengthy, so we remind the reader of some conventions that are used throughout:

\begin{enumerate}
\item In this section only, we fix the constant $M =  \lfloor \cnaive L \rfloor $ + 1, one more than the number of path components taken before using the ``naive" path, i.e. an upper bound on the total number of sub-paths used to construct each path to the ground state.
\item All paths that have positive measure under Definition \ref{DefFinalPath} will consist of a sequence of paths
\be \label{EqMainLemmaShortPathRepInt}
\gamma_{1} \circ \ldots \circ \gamma_{k},
\ee
where each element $\gamma_{i}$ has positive measure under Definition \ref{DefPartialPath}, possibly followed by a path of the form given in Definition \ref{DefNaivePath}. We will always write our paths in terms of this decomposition. When a single path has several decompositions of this form, we sum over all of these representations.

\item The measures  $\{\flowtot{\sigma}\}$ are as in Definition \ref{DefPartialPath}, the measures  $\{\flowfull{\sigma}\}$, $\{\flowtrunc{\sigma}{k}\}$ are as in Definition \ref{DefFinalPath}, and the measures $\{ \flowna{\sigma} \}$ are as in Definition \ref{DefNaivePath}. We note that $\flowna{\sigma}$ assigns full mass to a single path, which we denote $\gamma_{\sigma,\mathrm{na}}$.
\item For $\s \in \O$ and $e = (e_{-},e_{+})\in \edges$, we define the energy associated with a path started at $\s$ on edge $e$ as $Q(\s,e) = \pi(\s)/\big(\pi(e_{-}) \cL(e_{-},e_{+})\big)$.
\end{enumerate}

We fix an edge $e \in \edges$ and calculate a bound on the sum that appears in Lemma \ref{LemmaMultiSpec}, with sets $S = S_{k}$ of the form \eqref{DefHighDensitySet} for some $k \in [0\!:\!L^{2}]$. Although in principle we could take advantage of the fact that paths from $S_{k}$ that are truncated when they enter $S_{k}^{c}$ are shorter than paths which go all the way from $S_{k}$ to $\{+\}$, we will not do so. However, we \textit{will} use the fact that the initial state of any path from $S_{k}$ to $S_{k}^{c}$ will have at least $k$ defects. 

Recalling $M \lesssim \beta L$, we have:

\be 
\quad \sum_{\sigma \in S_{k}} \sum_{\gamma \ni e} \flowtrunc{\sigma}{k}(\gamma) | \gamma|  Q(\s,e) 
&\stackrel{Lemma \,\, \ref{LemmaPathLengthBound}}{\lesssim} \sum_{\sigma \in S_{k}} \sum_{\gamma \ni e} \flowtrunc{\sigma}{k}(\gamma) L^{2}\,\min(|D(\sigma)|, \cnaive L)Q(\sigma,e) \\
&\stackrel{Eq. \eqref{EqMainLemmaShortPathRepInt}}{=}  \sum_{m = 1}^{M} \sum_{j=1}^{m} \sum_{\substack{\gamma = \gamma_{1} \conc \ldots \conc \gamma_{m}\\ \init{\gamma}\in S_k \,,\gamma_{j} \ni e}} \flowtrunc{\init{\gamma}}{k}(\gamma) L^{2}\,\min(|D(\init{\gamma})|,\cnaive L) Q(\init{\gamma},e) \\
&= L^2  \sum_{m = 1}^{M} \sum_{j=1}^{\min(m,M-1)} \sum_{\substack{\gamma = \gamma_{1} \conc \ldots \conc \gamma_{m}\\ \init{\gamma}\in S_k \,,\gamma_{j} \ni e}} \flowtrunc{\init{\gamma}}{k}(\gamma)\,\min(|D(\init{\gamma})|, \cnaive L) Q(\init{\gamma},e)  \\
&+ L^2\sum_{\substack{\gamma = \gamma_{1} \conc \ldots \conc \gamma_{M} \\ \init{\gamma} \in S_k\,, \gamma_{M} \ni e} }\flowtrunc{\init{\gamma}}{k}(\gamma) \min(|D(\init{\gamma})|, \cnaive L)Q(\init{\gamma},e) \\
&\equiv L^{2} S_{\mathrm{init}} +  L^{2} S_{\mathrm{naive}}.\label{IneqCanonicalPathBreakup}
\ee
We bound these last two terms separately, starting with $S_{\mathrm{init}}$. As the subscripts in the following sums will become somewhat complicated, we introduce some short-term notation to reduce the visual clutter. In the following definitions, we always think of $\eta$ as the configuration in $\gamma$ that appears ``just before" the part of the path containing the edge $e$, and $2u$ as the number of defects removed in going from $\init{\g}$ to $\eta$ in excess of the minimal number. 
\be
\Gamma^{(i)}(m) &= \{ \gamma_{1} \conc \ldots \conc \gamma_{m} \in \Gamma \, : \, \, \gamma_{i} \ni e, \, \init{\gamma_{1}} \in S_{k} \}\quad \textrm{for}\ i\leq m\,, \\
\Gamma(m)&=\Gamma^{(m)}(m) = \{ \gamma_{1} \conc \ldots \conc \gamma_{m} \in \Gamma \, : \, \, \gamma_{m} \ni e, \ \init{\gamma_{1}},\init{\gamma_m} \in S_{k} \}, \\
\Gamma(m,u) &= \{ \gamma_{1} \conc \ldots \conc \gamma_{m} \in \Gamma(m) \, :  |D(\init{\gamma})| = |D(\init{\gamma_{m}})|+ 2(m-1+u) \}, \\
\Delta(\eta) &= \{ \gamma \in \Gamma_{\eta} \, : \, \gamma \ni e \}, \\
h(m) &= \min(|D(e_{-})| + 4m, \, \b L).
\ee
In this notation, we have omitted the dependence on many variables (e.g. the edge $e$ and starting level $k$) that do not change in the calculation. We compute:

\be \label{IneqSInitStart}
S_{\mathrm{init}} &=   \sum_{m = 1}^{M} \sum_{i=1}^{\min(m,M-1)} \sum_{\gamma \in \Gamma^{(i)}(m)} \flowfull{\init{\gamma}}^{(k)}(\gamma) \, \min(|D(\init{\gamma})|, \b L)\,  Q(\init{\gamma},e) \\
&=\sum_{m=1}^{M-1} \sum_{\gamma \in \Gamma(m)} \left( \prod_{j=1}^{m} \flowtot{\init{\gamma_{j}}}(\gamma_{j}) \right) \, \min(|D(\init{\gamma})|, \b L)\, Q(\init{\gamma},e)  \,,
\ee
where in the second line we took the marginal distribution of the first $i$ segments of the path $\gamma$ that is being summed over (integrating out the remaining segments from index $i+1$ to $m$, which are not relevant to the other terms in the sum). Denote the inner sum by:
\be
\label{eq:Sme}
S(m,e) &\equiv   \sum_{\gamma \in \Gamma(m)} \left( \prod_{j=1}^{m} \flowtot{\init{\gamma_{j}}}(\gamma_{j}) \right) \, \min(|D(\init{\gamma})|, \b L)\, Q(\init{\gamma},e)  \,.
\ee

We now decompose $S(m,e)$ over the possible values of $u$ and  factor out the term  which is due to the final part of the path, and observe that $\min(|D(\init{\gamma})|,\b L) \leq h(m)$, to get
\begin{align}	
  S(m,e) &=   \sum_{u=0}^{m-1} \sum_{\gamma \in \Gamma(m,u)} \left( \prod_{j=1}^{m} \flowtot{\init{\gamma_{j}}}(\gamma_{j}) \right) \,\min(|D(\init{\gamma})|,\b  L)\, Q(\init{\gamma_{m}},e) e^{-2(m+u-1)\beta}\nonumber\\
&\leq \sum_{u=0}^{m-1} \sum_{\eta \in S_k } \sum_{\gamma \in \Delta(\eta)} \, \flowtot{\eta}(\gamma) h(m) Q(\eta,e) e^{-2(m-1+u)\beta} \cdot\sum_{\gamma \in \Phi(m,u, \eta)} \left( \prod_{i=1}^{m-1} \flowtot{\init{\gamma_{i}}}(\gamma_{i}) \right)\,,\label{ineq:Sme}
\end{align}
where
\be
\Phi(m,u,\eta) &= \{  \gamma_{1} \conc \ldots \conc \gamma_{m-1} \, : \, \flowtot{\init{\gamma_{i}}}(\gamma_{i}) >0 \, \,  \forall \, i \in [1:(m-1)], \, \init{\gamma_{1}} \in S_{k}, \, \fin{\gamma_{m-1}} = \eta, \\
& \qquad \qquad |D(\init{\gamma})| = |D(\eta)| + 2(m-1+u) \}\,.
\ee
To continue, we bound the inner part of the sum, $\sum_{\gamma \in \Phi(m,u, \eta)} \left( \prod_{i=1}^{m-1} \flowtot{\init{\gamma_{i}}}(\gamma_{i}) \right)$, in two cases:

\begin{enumerate}
	\item \textbf{$|D(\h)| \leq \cultra L^{2}$.} In this case we begin with the trivial bound $\prod_{i=1}^{m-1} \flowtot{\init{\gamma_{i}}}(\gamma_{i}) \leq 1$. Thus, it is enough to bound the size of the set $\Phi(m,u,\h)$. Let $\mathbf{Bin}(u) = \{ s \in \{0,1\}^{m-1} \, : \, \sum_{i=1}^{m-1} s_{i} = u\}$; for $s \in \mathbf{Bin}(u)$ define
	\be
	\Phi_s(m,u,\eta) = \{ \gamma_{1} \conc \ldots \conc \gamma_{m-1} \in \Phi(m,u,\eta) \, : \, \forall \, i \in [1:(M-1)], \ |D(\init{\gamma_{i}})| = |D(\fin{\gamma_{i}})| + 2(1 + s_{i})\}.
	\ee
	Since each component $\gamma_{i}$ of a path in $\Phi(m,u,\eta)$ satisfies $|D(\init{\gamma_{i}})| - |D(\fin{\gamma_{i}})| \in \{2,4\}$, the sets $\{\Phi(u,\eta,s)\}_{s \in \mathbf{Bin}(u)}$ in fact partitions $\Phi(u,\eta)$:
	\be \label{EqGammaUEtaPart}
	\Phi(m,u,\eta) = \sqcup_{s \in \mathbf{Bin}(u)} \Phi_s(m,u,\eta).
	\ee
	
	Since $m \leq \b L$ and $|D(\h)| \leq \cultra L^{2}$ we have $|D(\fin{\gamma_i})|\lesssim \beta^{-2}L^2$ for each $i \in [1:m-1]$. So applying the first bounds in  Lemma \ref{LemmaAncestryLongPathBound} we have
	\be \label{EqGammaUEtaSFin}
	| \Phi_s(m,u,\eta)| \lesssim L^{4(m-1)} \beta^{-m +1 + u}
	\ee
	for all $s \in \mathbf{Bin}(u)$. Combining these bounds,
	\be
	\sum_{\g\in \Phi(m,u,\eta) }  \left( \prod_{i=1}^{M-1} \flowtot{\init{\gamma_{i}}}(\gamma_{i}) \right) &\leq |\Phi(m,u,\eta)| \stackrel{\eqref{EqGammaUEtaPart}}{\leq} \sum_{s \in \mathbf{Bin}(u)} | \Phi_s(m,u,\eta) | \stackrel{\eqref{EqGammaUEtaSFin}}{\lesssim} L^{4(m-1)} \beta^{-m + 1+ u} \binom{m-1}{u}.
	\ee
	
	\item \textbf{$|D(\h)| > \cultra L^{2}$.} The bound in this case is essentially the same. Since $|D(\h)| > \cultra L^{2}$ is large, we can improve the trivial bound $\prod_{i=1}^{M-1} \flowtot{\init{\gamma_{i}}}(\gamma_{i}) \leq 1$, by applying \eqref{IneqMainCombManyOptions}, to the (still very weak) bound $\prod_{i=1}^{m-1} \flowtot{\init{\gamma_{i}}}(\gamma_{i}) \leq \beta^{-m+1}$ (in fact the upper bound is closer to $\left(L/\beta \right)^{-m+1}$ but we don't use this additional strength).  On the other hand, since $|D(\h)|$ is large, applying Lemma \ref{LemmaAncestryLongPathBound} gives only the bound
	\be \label{EqGammaUEtaSFinNew}
	| \Phi(u,\eta,s)| \leq L^{4(m-1)},
	\ee
	so our upper bound is larger by a factor of $\beta^{m - (u+1 )}$ relative to Inequality \eqref{EqGammaUEtaSFin}.
	
	Since we have improved by $\beta^{-m+1}$ and worsened by a smaller factor of $\beta^{m - (u+1)}$, the conclusion
	\be \label{IneqNaiveBunchOfCasesUltraConc}
	\sum_{\g\in \Phi(m,u,\eta) }  \left( \prod_{i=1}^{m-1} \flowtot{\init{\gamma_{i}}}(\gamma_{i}) \right)\leq  L^{4(m-1)} \beta^{-m + 1+ u}  \binom{m-1}{u}
	\ee
	holds in this case as well.
\end{enumerate}

Thus, Inequality \eqref{IneqNaiveBunchOfCasesUltraConc} holds in both cases. Substituting Inequality \eqref{IneqNaiveBunchOfCasesUltraConc} into \eqref{ineq:Sme} and continuing:
\be \label{IneqSInitMid}
S(m,e) &\lesssim  \sum_{u=0}^{m-1}  \binom{m-1}{u} \beta^{-(m-1-u)} L^{4(m-1)} \sum_{\eta \in S_k } \sum_{\gamma \in \Delta(\eta)} \, \flowtot{\eta}(\gamma) h(m) Q(\eta,e) e^{-2(m-1+u)\beta} \\
&=  \beta^{-(m-1)}  \sum_{\eta \in S_k} \sum_{\gamma \in \Delta(\eta)}  \flowtot{\eta}(\gamma) h(m) Q(\eta,e) \, \sum_{u=0}^{m-1} \binom{m-1}{u}  (\beta e^{-2\beta})^{u}  \\
&\lesssim   \beta^{-(m-1)}h(m)  \sum_{\eta \in S_k\cap G(e)} \sum_{\gamma \in \Delta(\eta)} \, \flowtot{\eta}(\gamma) h(m) Q(\eta,e) \,,
\ee
where on the second line we used $L^{4(m-1)} e^{-2(m-1)\beta} \lesssim 1$.

Next, we apply the congestion and energy bounds in Proposition \ref{prop:PTmain}.
We note there are effectively two cases: $\h \in \cup_{\theta \in \Theta} G^{\mathrm{(init)}}_\theta(e) \cup G^{\mathrm{(mid)}}_\theta(e)$ or $\h \in \cup_{\theta \in \Theta} G^{\mathrm{(fin)}}_\theta(e)$. 
In the first case, the congestion bound is slightly stronger while the energy bound is slightly weaker; in the second case, the congestion bound is slightly weaker while the energy bound is much stronger. Thus, in both cases, we continue from Inequality \eqref{IneqSInitMid} to find:

\be \label{IneqSMFinal}
S(m,e) \lesssim |\Theta| \, \beta^{4}  \, e^{2.5 \beta}  \beta^{-(m-1)} \frac{ h(m) }{|D(e_{-})|}.
\ee
Combining Inequalities \eqref{IneqSInitStart} and \eqref{IneqSMFinal}, and noting that $h(m) \leq \min(|D(e_-)|,\b L) + 4m$,

\be \label{IneqSInitFinal}
S_{\mathrm{init}} &\lesssim \sum_{m=1}^{M-1} S(m,e) \lesssim \beta^{6} \, e^{2.5 \beta} \min(1, \frac{\b L}{|D(e_{-})|}).
\ee


We now bound the term $S_{\mathrm{naive}}$ by a similar factorisation as we did for $S_{\textrm{init}}$. We replace the previous simplifying notation with the following,
\be \label{EqCanonicalPathSimplifyingNaive}
\G_{\textrm{init}}(\eta) &= \{ \gamma_{1} \conc \ldots \conc \gamma_{M-1} \in \Gamma \, : \, \init{\gamma_{1}}\in S_k, \ \fin{\gamma_{M-1}} = \eta\}, \\
\Omega_{\textrm{na}} &= \{ \eta \in S_k \, : \, \gamma_{\eta,\mathrm{na}} \ni e \}, \\
\Omega(v) &= \{ \eta \in S_k \, : \, |D(\eta)|  = |D(e_{-})| + 2v, \, \gamma_{\eta,\mathrm{na}} \ni e \}\,.
\ee
In the following we now think of $\eta$ as the configuration in $\gamma$ that appears at the beginning of the naive part of the path (which contains $e$), and $2v$ as the number of defects removed $\h$ to $e_-$, which could be negative.
\be \label{IneqSNaiveBoundBeforeCases}
S_{\mathrm{naive}} &=  \sum_{\substack{\gamma = \gamma_{1} \conc \ldots \conc \gamma_{M} \\ \init{\gamma} \in S_k\,, \gamma_{M} \ni e}} \flowtrunc{\init{\gamma}}{k}(\gamma) \, \min(|D(\init{\gamma})|, \b L) Q(\init{\gamma},e) \\
&= \sum_{\h\in\O_{\textrm{na}}}\sum_{\gamma \in \G_{\textrm{init}}(\h)}\left( \prod_{i=1}^{M-1} \flowtot{\init{\gamma_{i}}}(\gamma_{i}) \right)  \, \min(|D(\init{\gamma})|, \b L) Q(\init{\gamma},e)
\ee
Now, defining $\G_{\textrm{init}}(\eta,x )=  \{ \gamma_{1} \conc \ldots \conc \gamma_{m-1} \in \Gamma \, : \, \init{\gamma_{1}}\in S_k, \ \gamma_{m-1} \ni (\h^x,\eta)\}$, we observe that $\G_{\textrm{init}}(\eta) \subset \cup_{x\in\L}\G_{\textrm{init}}(\eta,x )$.
Furthermore, using reversibility and that the rates are bounded by $1$, we have  \[Q(\init{\gamma},e)/Q(\init{\gamma},(\h^x,\h)) \leq Q(\h,e)\,,\] it follows that 
\begin{align*}
  S_{\mathrm{naive}} &\leq \sum_{\h\in\O_{\textrm{na}}} Q(\h,e)\sum_{x\in\L}\underbrace{\sum_{\gamma \in \G_{\textrm{init}}(\h,x)} \left( \prod_{i=1}^{M-1} \flowtot{\init{\gamma_{i}}}(\gamma_{i}) \right) \, \min(|D(\init{\gamma})|, \b L) Q\big(\init{\gamma},(\h^x,\h)\big)}\,,
\end{align*}
where the under-braced term is exactly $S\big(M,(\h^x,\h)\big)$ defined in Eq. \eqref{eq:Sme}. Applying Ineq. \eqref{IneqSMFinal} and summing over the possible number of defects removed in going from $\h$ to $e_-$,
\begin{align*}
  S_{\mathrm{naive}} &\lesssim \b^6e^{8\b}\b^{-(M-1)}\sum_{v\geq -L-2}\sum_{\h\in\O(v)}e^{-2v\b}\,,
\end{align*}
where  we applied the bounds $h(m)/|D(e_-)| \leq \b L$ and $\cL(e_-,e_+) \geq e^{-4\b}$.
Finally, applying Lemma \ref{LemmaNumAncFullNaivePath}
\begin{align}
\label{IneqSNaiveFinal}
   S_{\mathrm{naive}} &\lesssim \b^6e^{8\b}\b^{-(M-1)}\sum_{2v \geq -2L}  2^{2L} \sum_{k=(2v-2L)^+ }^{2v+2L} \binom{L^{2}}{k} e^{-2v \beta} \nonumber\\
&\leq \b^6e^{8\b}\b^{-(M-1)}\sum_{2v \geq -2L}  2^{2L} (4L+1)\frac{L^{2(2v+2L)}}{(2v-2L)^+!}e^{-2v\b}\nonumber\\
&\lesssim \b^6e^{8.5\b}\b^{-(M-1)} 2^{2L} e^{2L\b}\sum_{n\geq 0} \frac{(L^{2}e^{-\b})^n}{(n-4L)^+!}\nonumber\\
&\stackrel{L^2e^{-\b} \approx 1}{\lesssim}  \b^6e^{9\b}\b^{-(M-1)} 2^{2L}e^{2L\b}\,.
\end{align}
Since $M =\Theta(\cnaive L) $ and $L =\Theta(e^{0.5 \beta})$, we observe that the term $L^{2}  S_{\mathrm{naive}}$ is negligible compared to $L^2 S_{\mathrm{init}}$.

Summarizing the calculations in this section, we have by Inequalities \eqref{IneqCanonicalPathBreakup}, \eqref{IneqSNaiveFinal} and \eqref{IneqSInitFinal}

\be \label{IneqMainCanonicalPathConclusion}
\quad & \sum_{\sigma \in S_k} \sum_{\gamma \ni e} \flowtrunc{\sigma}{k}(\gamma) | \gamma| \frac{\pi(\sigma)}{\pi(e_{-}) \cL(e_{-},e_{+})} \lesssim L^{2}(S_{\mathrm{init}} + S_{\mathrm{naive}}) \\
&\lesssim L^{2} \big( \beta^{6} \, e^{2.5 \beta} \min(1, \frac{\b L}{|D(e_{-})|}) +  L^{4M+1} e^{(9 -2M) \beta} 2^{2L} \beta^{-M} \big) \\
&\lesssim \beta^{6} \, e^{3.5 \beta} \min(1, \frac{\b L}{|D(e_{-})|})\,, 
\ee
for any edge $e$ with $e_- \in S_k$. 

\section{Analysis of All-Plus Boundary Condition: Proof of Theorem \ref{ThmMainResPlus}}  \label{SecAllPlusRes}

\subsection{Upper Bounds}

Applying Lemma \ref{LemmaMultiSpec} with $S = \{ \sigma \in \config \, : \, |D(\sigma)| = 0 \}$ and the usual variational characterization of the spectral gap given in Equation \eqref{eq:gap},
\be
 T_{\rm rel}^+( L_c ) \leq \mathcal{A},
\ee
\noindent
where $\mathcal{A}$ is defined as in Equation \eqref{IneqDefCanPathObject}. By Inequality \eqref{IneqMainCanonicalPathConclusion},
\be
\mathcal{A} \lesssim \beta^{6} \, e^{3.5 \beta}.
\ee
Combining these two bounds completes the proof of the upper bound in Inequality \eqref{IneqPlusRel}.

Next, we prove the upper bound on the mixing time. Define the function $k \, : \, [0,1] \mapsto \mathbb{N}$ as in Equality \eqref{DefHighDensityRound}, and for $r>0$ let $S_{k(r)}$ be as in Equation \eqref{DefHighDensitySet}. Applying Inequality \eqref{IneqSpecProfileSubsets}, we have by Lemma \ref{LemmaMultiSpec} and the bound on $\mathcal{A}$ given in Inequality \eqref{IneqMainCanonicalPathConclusion} that
\begin{align} \label{IneqMainSpecProfConclusion}
\lambda^{-1}(S_{k(r)}) &\lesssim \beta^{6} \, e^{3.5 \beta}, \qquad e^{- \beta L } < r\,,  \\
\lambda^{-1}(S_{k(r)}) &\lesssim \beta^{8} \, e^{4 \beta} \frac{-1}{\log(r)}, \qquad 0 < r < e^{-  \beta L },\nonumber
\end{align}
where we recall from Lemma \ref{lem:dom} that $\pi(+) \approx 1$, and $|D(e_{-})| {\geq} k(r)$ for any edge $e$ with $e_- \in S_{k(r)}$. Applying Inequality \eqref{IneqMainSpectralProfileBound} with the bound from Inequality \eqref{IneqMainSpecProfConclusion}, we conclude
\be
T_{\rm mix}^+( L_c ) &\stackrel{\text{Ineq } \eqref{IneqMainSpectralProfileBound}}{\leq} \int_{4 \, \min_{\sigma} \pi(\sigma)}^{16} \frac{2}{x \, \Lambda_{\cL_{\L}^{+}}(x)} dx \\
&\stackrel{\text{Ineqs } \eqref{IneqSetContainmentSpectralProfile}, \eqref{IneqMainSpecProfConclusion} }{\lesssim} 2 \beta^{8} \, e^{4 \beta} \int_{4  e^{-L^{2} \beta}}^{e^{-L \beta}} \frac{1}{x \log(1/x)} dx + 2 \beta^{6} e^{3.5 \beta} \, \int_{e^{-L \beta}}^{16} \frac{1}{x} dx \lesssim \beta^{9} \, e^{4 \beta},
\ee
completing the proof of the upper bound in Inequality \eqref{IneqPlusMix}.

\subsection{Lower Bound on Relaxation Time}
\label{sec:lwrbnd}

We give a lower bound on the relaxation time by choosing an explicit test function that relaxes to equilibrium quite slowly, and explicitly bounding the two terms in the ratio \eqref{eq:gap}. We choose a test function motivated by the heuristic that the relaxation time is dominated by the time it takes for $4$ initially well seperated defects, at the corners of a rectangle, to annihilate. As illustrated in Figure \ref{FigCornerWalk}, these opposite corners  perform nearly-independent symmetric random walk until they are in adjacent columns or rows, and so our test function will be quite similar to the optimal test function for symmetric random walk in $[0:L_{c}]^{2}$.

Throughout the proof, to reduce notation we use $\pi = \pi^+_\L$, $c = c^{+}_\L$, $p = p^+$, $L = L_c$ and $\L = [1:L]^{2}$.
Let $f : \O_\L^+ \to \bbR_+$ be given by
\begin{align}
  f(\s) = g\left( \frac{|\s|}{L^2}\right)\,,
\end{align}
where $g:[0,1] \to [0,1]$ is defined for $x \in [0,1/2]$ by
\begin{align}
  \label{eq:g}
 &\begin{cases}
    g(x) = 0 & \textrm{if } x \in [0,1/4]\,, \\
    g(x) = 12x - 3 & \textrm{if } x \in (1/4,1/3) \\
    g(x) = 1 & \textrm{if } x \in [1/3,1/2]\,,
  \end{cases}
\end{align}
and $g(x) = g(1-x)$ for $x \in [1/2, 1]$.

To estimate $\var_\L^+(f)$ from below, it will turn out to be enough to compare the $+$ configuration  to configurations consisting of
exactly $4$ defects at the corners of a rectangle (i.e. configurations with a single rectangular region of $-1$ spins, and all $+1$ spins outside).
Informally, there are $\Theta(L^4)=\Theta(e^{2\beta})$ rectangles of area $\Omega(L^{2})$,
and by Lemma \ref{lem:dom} and Equality \eqref{eq:simple} the probability of any particular such configuration is $\Theta(e^{-4\b})$.
It will follow from these estimates on the number and probability of the ``large-rectangle" configurations that the $ \var_\L^+(f) \gtrsim e^{-2 \beta}$, as shown below.

More formally, let $\mathbf{B}$ be the collection of rectangles $R \subset [1:L]^{2}$ with areas $\frac{L^{2}}{3} \leq |R| \leq \frac{2 L^{2}}{3}$, and define
\be
\cC = \{ \sigma \in \O^+_\L \,:\, \exists R \in \mathbf{B} \, \text{ s.t. } \sigma_{x} \equiv 1 - 2 \1_{x \in R} \}
\ee
to be the rectangles with spin $(-1)$ exactly on a rectangle in $\mathbf{B}$.

Observe that $| \cC | \geq c L^4$ for some $c>0$. Since $f(\s) \geq 1$ for all $\s \in \cC$ and $f(+) = 0$, our bound on $|\cC| \geq c L^{4}$ gives the following bound on $ \var_\L^+(f)$:
\begin{align}
  \var_\L^+(f) &= \frac{1}{2} \sum_{\s,\h\in \O_\L}\pi(\s)\pi(\h)\left(f(\s) - f(\h)\right)^2\nonumber\\
  &\geq \frac{1}{2}\pi(+)\sum_{\s \in \cC}  \pi(\s)f^2(\s) \geq \frac{1}{2}\pi(+)^2\sum_{\s \in \cC}  e^{-4\b}\nonumber \gtrsim e^{-2\b}.
\end{align}

To estimate the Dirichlet form from above, we consider  various cases defined by the number of defects present in the initial and final configuration. Let
\begin{align*}
  \cA_n &= \{\h\in\config\,:\, |D(\h)| = n \}\,,\\
  \cB_n &= \cA_n\cup \cA_{n-2}\cup \cA_{n-4}\,,
\end{align*}
where $\cA_n$ is non-empty only for $n=0$ or $n\geq 4$ and $n$ even.
By reversibility, we may consider transitions which only remove defects (reducing the number by $0$, $2$ or $4$), or leave the number of defects the same, so
\begin{align}
  \label{eq:dupbnd}
  \cD_\L^+(f) &\leq  \sum_{\s\in\O} \sum_{\substack{x \in \config \, :\\ |D(\s)| \geq |D(\s^x)|}}\pi(\s)c(x,\s)\left(\nabla_x f(\s)\right)^2\nonumber \\
&= \sum_{k=2}^{L^2/2}\sum_{\s\in\cA_{2k}} \sum_{\substack{x\ :\\ \s^x \in \cB_{2k}}}\pi(\s)c(x,\s)\left(\nabla_x f(\s)\right)^2\,.
\end{align}
It remains to bound this sum from above.  We divide the contributions to the sum in \eqref{eq:dupbnd} into $5$ cases. In the Cases $1,4$ and $5$ we will use the following general bound on $\nabla_xf(\s)$. For all $\s \in \config$, we have \footnote{Note that $g'(x)$ does not exist at some isolated points $x$. To ensure that the following calculation still holds, we define $g'(x) = \lim_{y \uparrow x} g'(y)$ at these points. Note that essentially any sensible choice for these isolated points will not change the final estimate.}
\begin{align*}
  \left(\nabla_x g(\s)\right)^2 &= \left(g\left(\frac{|\s|}{L^2}\right) - g\left(\frac{|\s|\pm 1}{L^2}\right)\right)^2  \\
&= g'\left(\frac{|\s|}{L^2}\right)^{2} \,\frac{1}{L^4} + o\left(\frac{1}{L^4}\right)\lesssim e^{-2\b}.  \\
\end{align*}

\noindent\textbf{Case 1 ($\s \in \union_{k=4}^{L^2/2}\cA_{2k}$):} We will show that the contribution to the sum in \eqref{eq:dupbnd} due to configurations $\s$ with $8$ or more defects are negligible.
For $\s\in\cA_{2k}$ we have $\pi(\s) = \pi(+)e^{-2k\b}$, also $c(x,\s) \leq 1$ and $|\{x \in \L \,:\,  \s^{x} \in \cB_{2k} \}| \leq 8k$. Combining these bounds with Lemma \ref{LemmaExtensionWoo}, it follows that
\be
  \sum_{k=4}^{\infty}\sum_{\s\in\cA_{2k}} \sum_{\substack{x\ :\\ \s^x \in \cB_{2k}}}\pi(\s)c(x,\s) &\left(\nabla_x f(\s)\right)^2 \lesssim \left( L^{8}e^{-8\beta} + L^{12}e^{-12\b} +\sum_{k=8}^{\infty}k L^{3k}e^{-2k\b}\right) \, e^{-2\b} \\
&\lesssim  \left(e^{-4\b} + e^{-4\b} \left(8-7e^{-\b/2}\right)\right) \, e^{-2\b}\lesssim  e^{-6\b} \,,
\ee
where we use the bound $| \cA_{2k}| \leq (ek)^{2k}L^{2k}$ from Lemma \ref{LemmaExtensionWoo} to bound the two terms corresponding to $k=4$ and $k=6$, and we use the bound $ | \cA_{2k} | \leq L^{3k}$ from the same lemma to bound the remaining terms with $k \geq 8$.

\noindent\textbf{Case 2 ($\s \in \cA_4$ and $\s^x \in \cA_0$):}
If a single spin flip at $x$ removes all defects, then $\s_x=-1$ and $\s_y=1$ for $y\in\L\setminus\{x\}$.
Therefore $\left(\nabla_xf(\s)\right)^2 \leq \left(g(1/L^2)\right)^2 = 0$ for $L>2$  by \eqref{eq:g}.
It follows that there is no contribution to the sum in \eqref{eq:dupbnd} from this case.\\
\noindent\textbf{Case 3 ($\s \in \cA_4$ and $\s^x \in \cA_4$):}
By a similar argument as in Case 2, these transitions do not contribute to the sum for $L>4$.  Observe that, if $\s$ and $\s^{x}$ have the same number of defects, this means that there must be exactly two defects in the plaquette associated with $x$ - that is, $|D(\sigma) \cap \{x,x-e_{1}, x-e_{2}, x-e_{1} - e_{2}\}| = 2.$
Since $\s \in \cA_{4}$, this fact (and the parity Lemma \ref{lem:parity}) implies that the collection $\{y \in [1:L]^{2} \,: \, \sigma_{y} = -1 \}$
of sites with spin $-1$ must be a rectangle of width or height 1. In particular, $|\sigma| \leq L$ and so $\left(\nabla_xf(\s)\right)^2 \leq 4\left(g(1/L)\right)^2 = 0$ for $L > 4$ by \eqref{eq:g}.

\noindent\textbf{Case 4 ($\s \in \cA_6$ and $\s^x \in \cA_4$):}
The only spin flips which remove $2$ defects are those in which the initial state has $3$ neighboring defects which are replaced by a single defect (see the first transition in Figure \ref{FigPairMove}, which is a generic transition in this sense). It follows that $|\{\s \in \cA_6\,:\, \exists\, x \in \L \textrm{ with } \s^x \in \cA_4 \}|\lesssim |\cA_4|$.
By Lemma \ref{LemmaExtensionWoo}, $|\cA_4|\lesssim L^4$.
Also for any $\s \in \cA_6$ the number of sites at which a spin flip will reduce the number of defects is trivially bounded from above by $24$ (the total number of sites that can contain a defect in any of their associated plaquettes).
It follows that
\begin{align}
  \sum_{\s\in\cA_{6}} \sum_{\substack{x\ :\\ \s^x \in \cA_4}}\pi(\s)c(x,\s)\left(\nabla_x f(\s)\right)^2 \lesssim   e^{2\b} e^{-8\b} \pi(+) \lesssim  e^{-6 \b}\,.
\end{align}

The following case is the only one which gives a non-negligible contribution to our upper bound on the Dirichlet form.

\noindent\textbf{Case 5 ($\s \in \cA_6$ and $\s^x \in \cA_6$):}
As in Case 3, $\sigma$ has at least two neighbouring defects.
By the same argument as in Lemma \ref{LemmaExtensionWoo}, we have $|\{\s \in \cA_6\,:\, \exists\, x \in \L \textrm{ with } \s^x \in \cA_6 \}|\lesssim  L^5$.
Furthermore, any spin flip at a site which is not a neighbour of a defect must change the number of defects, so $|\{x\in \L\,:\, |p(\s)| = |p(\s^x)|\}| \leq 24$.
Finally for $\s \in \cA_6$ and $\s^x \in \cA_6$ we have $\pi(\s) = \pi(+)e^{-6\b}$ and $c(x,\s) = 1$, it follows that
\begin{align}
  \sum_{\s\in\cA_{6}} \sum_{\substack{x\ :\\ \s^x \in \cA_6}}\pi(\s)c(x,\s)\left(\nabla_x f(\s)\right)^2 \lesssim  e^{-2 \b} e^{(5/2)\b} e^{-6\b} \pi(+) \approx e^{-2 \b}\, e^{-(7/2) \b}.
\end{align}

Combining the bounds from Cases $1$--$5$ above with Inequality \eqref{eq:dupbnd} we have,
\begin{align}
  T^+_{\rm rel}(L_c) \geq \frac{\var_{\p}^{+}(f)}{\cD_\L^{+}(f)} \gtrsim e^{(7/2)\b}\,.
\end{align}


\section{Analysis of Periodic Boundary Condition: Proof of Theorem \ref{ThmMainResPer}} \label{SecPerBoundRes}

Throughout this section, we often reserve subscripts for a time index and use the ``bracket" notation $x[i] \equiv x_{i}$ to indicate an element of a vector $x = (x_{1},\ldots,x_{k})$. In this section, we have many explicit probabilistic calculations related to Markov processes $\{X_{t}\}_{t \geq 0}$ running according to some generator and with starting points $X_{0} = x$. When we wish to emphasize the starting point of a process in such a calculation, we use subscripts, as in \textit{e.g.} $\P_{x}[X_{t} \in S]$, $\E_{x}[f(X_{t})]$.

\subsection{Lower Bound}
We prove the lower bound in Theorem \ref{ThmMainResPer} by constructing and analyzing a test function.

We begin by giving the heuristic that guides our proof of this inequality. Roughly speaking, we expect the relaxation time of the dynamics of the spin model $\{X_{t}\}_{t \geq 0}$ to be bounded below by the relaxation time of the dynamics of the trace $\{ \hat{X}_{t} \}_{t \geq 0}$ of $\{X_{t}\}_{t \geq 0}$ on the set $\cG = \{ X \in \config \, : \, H_{\L}^{\mathrm{per}}(X) = \min_{\h\in\config}H_{\L}^{\mathrm{per}}(\h)\}$ of ground states. Furthermore, this ``trace walk" $\{ \hat{X}_{t}\}_{t \geq 0}$ behaves very much like simple random walk on the hypercube $\{-1,+1\}^{2L-1}$. This can be made precise in the following way: there is a natural bijection $w \, : \, \cG \mapsto \{-1,+1\}^{2L-1}$  (see Definition \ref{DefBijectionPeriodic} below) under which the dynamics of the trace walk $\{w(\hat{X}_{t}) \}_{t \geq 0}$ are a very small perturbation of the usual simple random walk on the hypercube $\{-1,+1\}^{2L-1}$, with one added transition that flips all signs. Thus, we expect that a test function based on an eigenfunction for the kernel of simple random walk on the hypercube with largest eigenvalue, the usual ``Hamming weight" function, should be able to pick out the slowest mixing behaviour of $\{X_{t}\}_{t \geq 0}$.

We now make this heuristic precise. We begin by defining the bijection $w \, : \, \cG \mapsto \{-1,+1\}^{2L-1}$, used to identify the ground states, as well as a collection of functions that will allow us to extend a test function on $\cG$ to a test function on all of $\config$.
Note, that to avoid excessive subscripts we use the notation $\sigma[i,j]$ to denote the $(i,j)^{\textrm{th}}$ component of $\sigma$, previously denoted $\sigma_{(i,j)}$.
In words, $w$ assigns a $-1$ to all the rows and columns of a ground state that are flipped with respect to the $+$ configuration, and $+1$ to all the others, where the columns are labeled $1$ to $L$ (left to right), and the rows $L+1$ to $2L-1$ (bottom to top). The top row of spins are fully specified by all the other rows and columns.

We give the full notation required to define our test function in the following definition; the informal description of this notation in Remark \ref{RemRealRDef} is sufficient to follow the remainder of the argument:

\begin{defn} [Preliminary Notation for Test Function] \label{DefBijectionPeriodic}
 We define a bijection $w \, : \, \cG \mapsto \{-1,+1\}^{2L-1}$ between the collection of ground states and the $(2L-1)$-hypercube by the following formula for its inverse:
\be \label{EqCristinaBijection}
w^{-1}(v) [i,j] &= v[i] v[L+j], \qquad i \in [1:L], \, j \in [1:(L-1)] \\
w^{-1}(v)[i,L] &= v[i], \qquad \qquad \quad \,  i \in [1:L].
\ee

For $v \in \{ -1, 1\}^{2L-1}$, define $|v| = \sum_{i=1}^{2L-1} \1_{v[i] = 1}$ to be the Hamming weight of $v$. We also denote by $\prec$ some fixed total order on $\{-1,+1\}^{2L-1}$ that extends the usual Hamming partial order.

For  $m \in [1:L]$, $k \in [1:L]$, and $\sigma, \eta \in \cG$ with $w(\sigma)$, $w(\eta)$ differing at the single index $\ell \in [1\!:\!L]$ and with $w(\sigma) \prec w(\eta)$, define the configuration $R(\sigma,\eta,m,k)$ by
\be
R(\sigma, \eta,m,k)[\ell,j] &= \eta[\ell,j], \qquad j \in [m:(m+k-1)] \\
R(\sigma, \eta, m,k) [i,j] &= \sigma[i,j], \qquad \text{all other entries},
\ee
where in this case the ``interval" $[m\! :\! (m+k-1)]$ is defined modulo $L$ (so that, if $L = 9$, we have $[7\!:\!2] = \{7,8,9,1,2\}$).

Similarly, for $\sigma,\eta$ with $w(\sigma)$, $w(\eta)$ differing at the single index $\ell \in [(L+1):(2L-1)]$ and with $w(\sigma) \prec w(\eta)$, define the configuration $R(\sigma,\eta,k,m)$ by
\be
R(\sigma, \eta,m,k)[i,\ell] &= \eta[i,\ell-L], \qquad i \in [m:(m+k-1)] \\
R(\sigma, \eta, m,k) [i,j] &= \sigma[i,j], \qquad \text{all other entries.}
\ee
Finally, for $\sigma, \eta$ with $w(\sigma) = - w(\eta)$ and $w(\sigma) \prec w(\eta)$, define the configuration $R(\sigma, \eta, m,k)$ by
\be
R(\sigma, \eta,m,k)[i,L] &= \eta[i,L], \qquad i \in [m:(m+k-1)] \\
R(\sigma, \eta, m,k) [i,j] &= \sigma[i,j], \qquad \text{all other entries.}
\ee
For convenience,  when $w(\s) \prec w(\h)$ we let $R(\s,\h,m,0)=\s$, and when $w(\eta) \prec w(\sigma)$ we define $R(\sigma,\eta,m,k) \equiv \emptyset$ for each $k\in [0:L]$.

We note that, if you restrict the range of $k$ to the set $[1:L]$, then $R$ is an injective map\footnote{Note that obtaining this injectivity is the reason we insist that $w(\sigma) \prec w(\eta)$.} - that is, if $R(\sigma, \eta,m,k) = R(\sigma', \eta', m',k')$ for $k, k' \in [2:L-1]$, then $\sigma = \sigma'$, $\eta = \eta'$, $m=m'$ and $k=k'$. For $k \in [1:L]$, let
\be \label{DefMinimalPaths}
\mathcal{R}_{k} = \{ \zeta \in \Omega_{[1:L]^{2}} \, : \, \exists  \, \sigma, \, \eta, \, m \, \, \text{ s.t. } \, \, \zeta = R(\sigma,\eta,m,k) \},
\ee
so that $\mathcal{R} \equiv \cup_{k=1}^{L} \mathcal{R}_{k}$ is the collection of points  along these minimal-length and -energy paths between ground states. Note that, for $k, k' \in [2:(L-2)]$, $\mathcal{R}_{k} \cap \mathcal{R}_{k'} = \emptyset$. Thus, for $\zeta \in \cup_{k=2}^{L-1}\mathcal{R}_{k}$, let $\sigma(\zeta), \eta(\zeta), m(\zeta), k(\zeta)$ be the unique elements that satisfy
\be \label{EqDefZetaBackFunc}
\zeta = R(\sigma(\zeta),\eta(\zeta),m(\zeta),k(\zeta)).
\ee
When $\zeta \in \mathcal{R}_{1} \cup \mathcal{R}_{L-1}$, there may be several choices that satisfy Equation \eqref{EqDefZetaBackFunc}, since these configurations have a single spin flipped relative to an element of the ground state; this single spin might be the first or last spin to be flipped in a path that is flipping a row or column; in total there can be up to four choices.  Denote by $\phi(\zeta)$ the number of these choices, and denote the choices themselves by $(\sigma_{1}(\zeta), \ldots, k_{1}(\zeta)), \ldots , ( \sigma_{\phi(\zeta)}(\zeta),\ldots, k_{\phi(\zeta)}(\zeta))$, in any arbitrary but fixed order.

\end{defn}

\begin{remark} \label{RemRealRDef}
 For fixed $m \in L$, the sequence $\{ R(\sigma,\eta,m,k) \}_{k=0}^{L}$ gives one of many minimal-length paths from $\sigma$ to $\eta$: that is, $R(\sigma,\eta,m,0) = \sigma$, $R(\sigma,\eta,m,L) = \eta$, and $R(\sigma,\eta,m,k)$ is obtained by flipping $k$ adjacent spins of the column on which $\sigma, \eta$ disagree. Roughly speaking: $m$ identifies where in the column we start ``flipping" from $\sigma$ to $\eta$, and $k$ tells us ``how far" we are in the path from $\sigma$ to $\eta$. See Figure \ref{FigDominantTransitionsBetweenGroundStates} for a sample pair of ground states and an element along the path between them.

\begin{figure}[h]
\includegraphics[scale = 0.7]{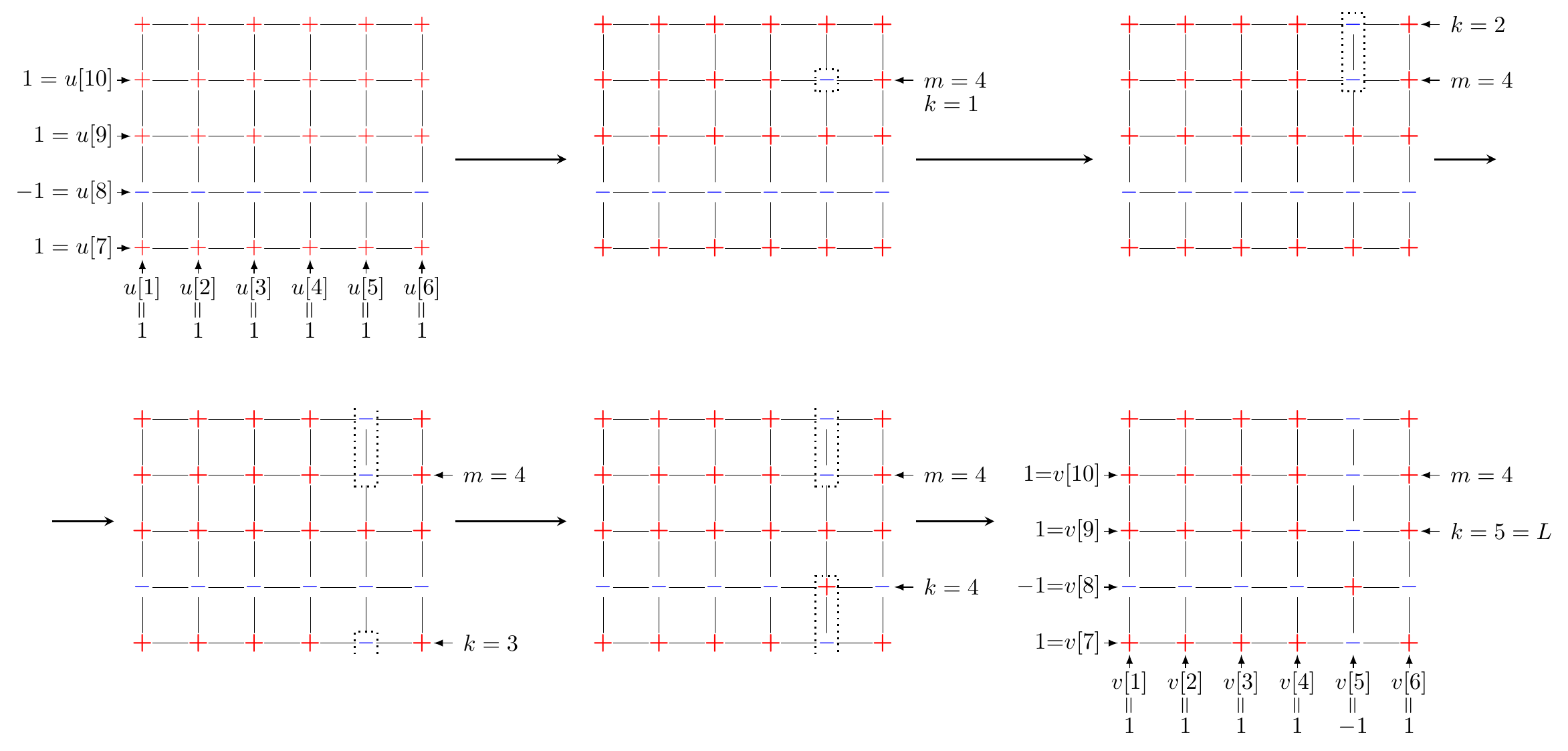}
\caption{\label{FigDominantTransitionsBetweenGroundStates} A minimal-length path $\{R(\s,\h,m,k)\}_{k=0}^{L}$ between two ground states $\s = w^{-1}(u)$ and $\h=w^{-1}(v)$, where $u$ and $v$ are given in the first and final frames.}
\end{figure}

\end{remark}

We will see that the paths of the form $( R(\sigma,\eta,m,k) )_{k=0}^{L}$ describe the only high-probability paths between pairs of ``neighboring" ground states (differ by one coordinate of $w$). This allows us to lift test functions from $\{-1,+1\}^{2L-1}$ to similar test functions on $\Omega_{[1:L]^{2}}$. We can now define a simple test function:

\begin{defn} [Test Function] \label{DefTestTrelLowerPeriodic}
Define $\hat{g} \, : \, \{-1,+1\}^{2L-1} \mapsto \mathbb{R}$ by
\be
\hat{g}(u) = \left| \, |u| - \frac{2L-1}{2} \, \right|.
\ee
Define $g \, : \, \config \mapsto \mathbb{R}$ by
\begin{align*}
  g(\zeta) = 
  \begin{cases}
    \hat{g}(w(\zeta)), & \zeta \in \mathcal{G},\\
    \frac{k(\zeta)}{L} \hat{g}(w(\eta(\zeta))) + (1 - \frac{k(\zeta)}{L})\hat{g}(w(\sigma(\zeta))), &\zeta \in \cup_{k=2}^{L-2}\mathcal{R}_{k}, \\
    \frac{k_{1}(\zeta)}{L} \hat{g}(w(\eta_{1}(\zeta))) + (1 - \frac{k_{1}(\zeta)}{L})\hat{g}(w(\sigma_{1}(\zeta))), & \zeta \in \mathcal{R}_{1} \cup \mathcal{R}_{L-1}, \\
    0,&\zeta \notin \mathcal{R}.
  \end{cases}
 \end{align*}
\end{defn}
This is well-defined for $L > 4$, since $\mathcal{R}_{1} \cup \mathcal{R}_{L-1}$ is disjoint from $ \cup_{k=2}^{L-2}\mathcal{R}_{k} $ (the former are obtained by flipping exactly 1 spin from a ground state; the latter all differ from any ground state by at least $2$ spins). Note that this test function essentially ignores the behaviour of $\{X_{t}\}_{t \geq 0}$ except on the collection of ground states and minimal-length paths between ground states. On the minimal-length paths between ground states, $g$ linearly interpolates the values of $\hat{g}$, with the possibility of an error of $\frac{2}{L}$ for configurations that are the first or last configuration along a minimal-length path.

We are primarily interested in the following property of $g$: if $\zeta_{1}, \zeta_{2} \in \mathcal{R}$ and $\cL^{\rm per}_{\L}(\zeta_{1},\zeta_{2})  \neq 0$, then
\be \label{IneqSimpleGDiff}
|g(\zeta_{1}) - g(\zeta_{2})| \leq \frac{3}{L};
\ee
note that we require a factor of 3 (rather than 1) in this inequality to deal with the possibility that $\zeta_{1}$ or $\zeta_{2}$ are in the set $\mathcal{R}_{1} \cup \mathcal{R}_{L-1}$.

We now compare the Dirichlet form and variance of the test function $g$ given in Definition \ref{DefTestTrelLowerPeriodic}. We begin with a lower bound on the variance. If $X \sim \pi^{\rm per}_{\L}(\cdot)$, and $Y \sim \pi^{\rm per}_{\L}( \cdot \, | \, Y \in \mathcal{G})$ is distributed uniformly on the ground states, then
\be \label{IneqLowerBoundTRelPeriodicVarianceFinal}
\var_{\L}^{\rm per}(g) &= \E[g(X)^{2}] - \E[g(X)]^{2} \\
&\geq \pi(\mathcal{G}) \times (\E[g(Y)^{2}] - \E[g(Y)]^{2}) \\
&\gtrsim (\E[g(Y)^{2}] - \E[g(Y)]^{2}) \gtrsim L, \\
\ee
where the second-last line uses the fact that $\pi(\mathcal{G}) \gtrsim 1$ from Lemma \ref{lem:dom} and the last line uses the fact that $g(Y) \sim \mathrm{Bin}(2L-1, 0.5)$.

We next calculate an upper bound on the Dirichlet form. We observe that if $\sigma, \eta$ are two ground states with $w(\sigma) = - w(\eta)$, then $g(\sigma) = g(\eta)$, and so these minimal-length transitions do not contribute to the following calculation. Using reversibility of $\cL^{\rm per}_{\L}$ and the fact that $\cL^{\rm per}_{\L}(\sigma, \eta) = 0$ for all pairs $\sigma \in \mathcal{G}$, $\eta \notin \mathcal{R}$, we have
\be \label{IneqLowerBoundTRelPeriodicSplitting}
2 \cD_{\L}^{\rm per}(g) &=  \sum_{\sigma, \eta \in \config} (g(\sigma) - g(\eta))^{2} \pi_{\L}^{\rm per}(\sigma) \cL(\sigma, \eta) \\
&\leq \sum_{\sigma \in \mathcal{G}} \sum_{\eta \in \mathcal{R}}  (g(\sigma) - g(\eta))^{2} \pi_{\L}^{\rm per}(\sigma) \cL^{\rm per}_{\L}(\sigma, \eta) + \sum_{k=1}^{L-1} \sum_{\sigma \in \mathcal{R}_{k}} \sum_{\eta \in \mathcal{R}_{k \pm 1}}  (g(\sigma) - g(\eta))^{2} \pi_{\L}^{\rm per}(\sigma) \cL(\sigma, \eta) \\
&+ \sum_{\sigma \in \mathcal{R}} \sum_{\eta \notin  \mathcal{R}}  (g(\sigma) - g(\eta))^{2} \pi_{\L}^{\rm per}(\sigma) \cL(\sigma, \eta)  \\
&\equiv S_{1} + S_{2} + S_{3},
\ee
where the inequality is there because the sets $\{\mathcal{R}_{k}\}_{k=1}^{L-1}$ are not quite disjoint. We calculate these three terms separately. Using the estimate $\pi_{\L}^{\rm per}(\mathcal{G}) = 1 - o(1)$ from Lemma \ref{lem:dom}, the bound \eqref{IneqSimpleGDiff}, and the fact that $\cL(\sigma,\eta) \in \{ e^{-4 \beta}, 0\}$ for $\sigma \in \mathcal{G}, \eta \notin \mathcal{G}$, we have

\be \label{IneqLowerBoundTRelPeriodicS1}
S_{1} &= \sum_{\sigma \in \mathcal{G}} \sum_{\eta \in \mathcal{R}}  (g(\sigma) - g(\eta))^{2} \pi_{\L}^{\rm per}(\sigma) \cL(\sigma, \eta)  \\
&\lesssim \sum_{\sigma \in \mathcal{G}} \, |\{\eta \in \mathcal{R} \, : \, \cL(\sigma,\eta) > 0\}| \, L^{-2} \, | \mathcal{G}|^{-1} \, e^{-4 \beta} \\
&\lesssim | \mathcal{G} | \times L^{2} \times L^{-2} \times |\mathcal{G}|^{-1} \times e^{- 4 \beta}  \\
&=  e^{-4 \beta}.
\ee

Next, note that an element $ \zeta \in \mathcal{R}_{k}$  is uniquely determined by a choice of $\sigma \in \mathcal{G}$, $\eta \in \mathcal{G}$ s.t. $|w(\sigma) - w(\eta)|=1$, and $m \in [1:L]$. Multiplying these three factors,
\be
|\mathcal{R}_{k}| \leq |\mathcal{G}| \times (2L) \times (L).
\ee
Using this bound, the trivial bound $\cL(\sigma,\eta) \leq 1$ for any configurations $\sigma,\eta$, and applying Lemma \ref{lem:dom} and Inequality \eqref{IneqSimpleGDiff} as above, we have:

\be \label{IneqLowerBoundTRelPeriodicS2}
S_{2} &= \sum_{k=1}^{L-1} \sum_{\sigma \in \mathcal{R}_{k}} \sum_{\eta \in \mathcal{R}_{k \pm 1} \, : \, \cL(\sigma,\eta) > 0}  (g(\sigma) - g(\eta))^{2} \pi_{\L}^{\rm per}(\sigma) \cL(\sigma, \eta) \\
&\lesssim L \times (|\mathcal{G}| \times L^{2}) \times (1) \times L^{-2} \times (e^{-4 \beta} \times \frac{1}{|\mathcal{G}|}) \times (1) \\
&= L \, e^{- 4 \beta} \lesssim e^{-\frac{7}{2} \beta}.
\ee

Finally, we note that for $\sigma \in \mathcal{R}$ and $\eta$ with $\cL(\sigma,\eta) > 0$, we must have $|D(\eta)| \leq |D(\sigma)|+4 \leq 8$. Furthermore, if $|D(\eta)| = 6$, $|D(\sigma)| = 4$, and $\eta = \sigma^{x}$ for some $x \in [1:L]^{2}$, then the single flipped spin $x$ must share a plaquette with one of the $O(1)$ defects in $D(\sigma)$. In particular, for each fixed $\sigma \in \mathcal{R}$,
\be
|\{\eta \notin \mathcal{R} \, : \, \cL(\sigma,\eta)> 0, \, |D(\eta)| = 6\}| = O(1).
\ee
 Finally, we use the trivial bound $g(\sigma) \leq L$. Then,  applying Lemma \ref{lem:dom} and Inequality \eqref{IneqSimpleGDiff} as above, we have:

\be \label{IneqLowerBoundTRelPeriodicS3}
S_{3} &= \sum_{\sigma \in \mathcal{R}} \sum_{\eta \notin \mathcal{R}}  (g(\sigma) - g(\eta))^{2} \pi_{\L}^{\rm per}(\sigma) \cL(\sigma, \eta) \\
&= \sum_{\sigma \in \mathcal{R}} \sum_{\eta \notin  \mathcal{R}, \, |D(\eta)| = 6}  (g(\sigma) - g(\eta))^{2} \pi_{\L}^{\rm per}(\sigma) \cL(\sigma, \eta) + \sum_{\sigma \in \mathcal{R}} \sum_{\eta \notin  \mathcal{R}, \, |D(\eta)| = 8}  (g(\sigma) - g(\eta))^{2} \pi_{\L}^{\rm per}(\sigma) \cL(\sigma, \eta)\\
&\lesssim (L^{2} \times | \mathcal{G} |) \times  (1) \times L^{2} \times (\frac{1}{|\mathcal{G}|} \, e^{-4 \beta}) \times e^{-2 \beta} + (L^{2} \times | \mathcal{G}|) \times L^{2} \times L^{2} \times (\frac{1}{|\mathcal{G}|} \, e^{-4 \beta}) \times e^{-4 \beta} \\
&= L^{4} e^{-6 \beta} + L^{6} e^{-8 \beta} \lesssim e^{-4 \beta}.
\ee
Combining Inequality \eqref{IneqLowerBoundTRelPeriodicSplitting} with Inequalities \eqref{IneqLowerBoundTRelPeriodicS1}, \eqref{IneqLowerBoundTRelPeriodicS2} and \eqref{IneqLowerBoundTRelPeriodicS3}, we have
\be \label{IneqLowerBoundTRelPeriodicDirichletFinal}
\cD_{\L}^{\rm per}(g) \lesssim e^{-\frac{7}{2} \beta}.
\ee
Combining Inequalities \eqref{IneqLowerBoundTRelPeriodicVarianceFinal} and \eqref{IneqLowerBoundTRelPeriodicDirichletFinal}, the variational characterization of the spectrum in \eqref{eq:gap} gives

\be
\trel(L_{c}, +) &\geq \sup_{\var_{\pi}(f) \neq 0} \frac{\var_{\L}^{\rm per}(f)}{\cD_{\L}^{\rm per}(f)} \\
&\geq \frac{\var_{\L}^{\rm per}(g)}{\cD_{\L}^{\rm per}(g)} \\
&\gtrsim  L \times e^{\frac{7}{2} \beta} = e^{4 \beta}.
\ee
This completes the proof of the lower bound of Theorem \ref{ThmMainResPer}.


\subsection{Upper Bound}

We now prove the upper bound in Theorem \ref{ThmMainResPer}. As in the lower bound, the main heuristics are that the trace $\{\hat{X}_{t}\}_{t \geq 0}$ of $\{X_{t}\}_{t \geq 0}$ on $\cG$ should determine the mixing time, and that this trace is very similar to simple random walk on the hypercube. Roughly speaking, our argument has the following steps:

\begin{enumerate}
\item We show that the dynamics of the trace process can be compared to simple random walk on the hypercube, and conclude that it has a similar mixing time of order $O(L \log(L)^{2})$ (Lemma \ref{LemmaHyperComp} and Corollary \ref{LemmaTraceMixing}).
\item Using the characterization of mixing times in terms of hitting times of large sets (see Theorem 1 of \cite{Peres2015}), we conclude from step 1 that the hitting time of any large set (for the trace process) is bounded by $O(L \log(L)^{2})$.
\item We show, for the original process, that the expected hitting time of the set $\mathcal{G}$ is  $ O \left( \beta^{9} e^{4 \beta} \right)$. (Lemma \ref{LemmaHittingGround}).
\item 
We show, for the original process, that the expected transition time between elements of $\mathcal{G}$ is $ O \left( \beta^{8} e^{3.5 \beta} \right)$ (Lemma \ref{LemmaExpectedExcursionPerGr}).
\item Putting these bounds together, and recalling that $\pi_{\L}^{\rm per}(\mathcal{G}) = 1 - o(1)$, we conclude that the hitting time of any large set (for the original process) cannot be much more than $O(L \times \beta^{8} e^{3.5 \beta} + \beta^{9} e^{4 \beta}) = O(\beta^{9} e^{4 \beta})$. We conclude the argument by again using the characterization of \cite{Peres2015}.
\end{enumerate}

The main ingredient in Step 3 is relating the defect dynamics in the case of periodic boundary conditions to the defect dynamics in the case of all-plus boundary conditions.
This relation is a consequence of the following three simple observations:

\begin{enumerate}
\item By Lemma \ref{lem:parity}, the sets $p^{+}(\Omega_{[1:L]^{2}})$ and $p^{\mathrm{per}}(\Omega_{[0:L]^{2}})$ are the same.
\item Let $\{X_{t}^{(+)}\}_{t \geq 0}$, $\{X_{t}^{(\mathrm{per})}\}_{t \geq 0}$ be Markov chains evolving according to the generators $\cL_{[1:L]^{2}}^{+}$ and $\cL_{[0:L]^{2}}^{(\mathrm{per})}$ respectively, as defined in Equation \eqref{eq:gen}. Then $\{ p^{+}(X_{t}^{+})\}_{t \geq 0}$ and $\{ p^{\mathrm{per}}(X_{t}^{(\mathrm{per})})\}_{t \geq 0}$ are both themselves Markov chains, on the common space $p^{+}(\Omega_{[1:L]^{2}}) = p^{\mathrm{per}}(\Omega_{[0:L]^{2}})$, and they have the same stationary distribution. We denote their generators by $\cQ_{[1:L]^{2}}^{+}$ and $\cQ_{[0:L]^{2}}^{(\mathrm{per})}$ respectively.
\item For  all $x,y \in p^{+}(\Omega_{[1:L]^{2}})$,
\be \label{LemmaCriBijection}
\cQ_{[0:L]^{2}}^{(\mathrm{per})}(x,y) \geq   \cQ_{[1:L]^{2}}^{+}(x,y).
\ee
Note that this inequality is really not an equality, as there are transitions in the periodic boundary condition case that are not possible in the all-plus boundary condition case.
\end{enumerate}

There is an important sequence of observations that makes most of the following calculation quite straightforward, even if some of the details are messy:
First, almost all ``excursions" between elements of the ground state $\mathcal{G}$ remain within the low-energy set $\mathcal{R}$. Furthermore, each excursion from $\mathcal{G}$ can be decomposed into its first jump into $\mathcal{R}_{1} \cup \mathcal{R}_{k-1}$, and the remaining excursions into sets of the form $\{R(\sigma,\eta,m,k)\}_{k=2}^{L-2}$ for some fixed $\sigma, \eta \in \mathcal{G}$ and starting point $m \in [1:L]$. As long as these longer excursions remain within a set $\{R(\sigma,\eta,m,k)\}_{k=2}^{L-2}$ , they follow \textit{exactly} the law of a simple random walk on $[2:(L-2)]$, with the obvious bijection between $\{R(\sigma,\eta,m,k)\}_{k=2}^{L-2}$ and $[2:(L-2)]$ (see Figure \ref{FigDominantTransitionsBetweenGroundStates}). In particular, essentially all of the following calculations will be written in terms of simple random walk with killing, combined with simple 1-step analyses of what occurs between excursions.

Continuing more formally, we introduce the following definition for the discrete-time trace of a continuous-time Markov chain:

\begin{defn} [Non-Lazy Trace Chain] \label{DefTrace}
Let $\{X_{t}\}_{t \geq 0}$ be a c\`adl\`ag Markov process on finite set $\Omega$, and let $S \subset \Omega$. Define two sequences of times $\{t_{i}\}_{i \in \mathbb{N}}$, $\{s_{i}\}_{i \in \mathbb{N}}$ by the following recursions:
\be
t_{1} &= \inf \{t > 0 \, : \, X_{t} \in S \} \\
s_{i} &= \inf \{t > t_{i} \, : \, X_{t} \notin S\} \\
t_{i+1} &= \inf \{t > s_{i} \, : \, X_{t} \in S, \, X_{t} \neq X_{t_{i}} \}.
\ee
Then define the trace process $\{\hat{X}_{k}\}_{k \in \mathbb{N}}$, on $S$,  by the equation
\be
\hat{X}_{k} \equiv X_{t_{k}}.
\ee
We denote by $Q_{\mathcal{G}}^{\rm per}$ the transition kernel associated with the trace of $\{X_{t}\}_{t \geq 0}$ on the set $\mathcal{G}$.
\end{defn}

Denote by $Q_{\rm hyp}$ the following discrete-time transition kernel on $\{-1,+1\}^{2L-1}$:
\begin{align*}
Q_{\rm hyp}(u,v) &= \frac{1}{2L-1}, \qquad \textrm{if }\ |u-v| = 1, \\
Q_{\rm hyp}(u,v) &= 0, \qquad \text{otherwise.} \\
\end{align*}
Recall $|\cdot|$ is the Hamming distance, i.e. $|u-v| = 1$ iff $u$ and $v$ differ at exactly one coordinate.

The following estimate allows us to make very strong conclusions about the mixing of $Q_{\mathcal{G}}^{\rm per}$ in terms of $Q_{\rm hyp}$:

\begin{lemma} \label{LemmaHyperComp}
For  $Q_{\mathcal{G}}^{\rm per}$, $Q_{\rm hyp}$ as above, and the function $w$ as in Definition \ref{DefBijectionPeriodic}, we have
\be \label{IneqLemmaHyperCompMainConc}
Q_{\mathcal{G}}^{\rm per} (w^{-1}(u),w^{-1}(v)) \gtrsim Q_{\rm hyp} (u,v)
\ee
uniformly in $u,v \in \{-1,+1\}^{2L-1}$. Furthermore, for $\{X_{t} \}_{t \geq 0}$ a copy of the plaquette dynamics started at $X_{0} = \sigma \in \mathcal{G}$, let
\be \label{EqDefZetaEndCleanup}
\tau_{\rm start} &= \inf \{t > 0 \, : \, X_{t} \neq \sigma\} \\
\tau_{\rm end} &= \inf \{ t > \tau_{\rm start} \, : \, X_{t} \in \cG \text{ or } |D(X_{t})| > 4\}
\ee
be the times at which the first excursion from $\cG$ starts and ends. Then
\be \label{IneqExcursionBadProbs}
\P_\s[|D(X_{\tau_{\rm end}} )| > 4] &\lesssim L^{2} e^{-2 \beta} \\
1 - \P_\s[X_{\tau_{\rm end}} = \sigma] &\approx  L^{-1}.
\ee
\end{lemma}

\begin{remark} [Coupling to Simple Random Walks] \label{RemSRWCouplingExplicit}

The proof of Lemma \ref{LemmaHyperComp} (and the rest of this section) relies heavily on the following simple observation: every configuration $\zeta \in  \mathcal{R} \backslash (\mathcal{G} \cup \mathcal{R}_{1} \cup \mathcal{R}_{L-1})$ consists of two pairs of adjacent defects, and these two pairs each  \textit{exactly} undergo independent random walk until the first time they are adjacent or new defects are added (again, see Figure \ref{FigDominantTransitionsBetweenGroundStates}).

More formally, fix $m \in [1:L]$ and $k \in [2:(L-2)]$, and let $\{L_{t}, U_{t}\}_{t \geq 0}$ be two independent simple random walks on $\mathbb{Z}$ with rate 1 and starting points $L_{0} = m$, $U_{0} = m + k$. Let $\{X_{t}\}_{t \geq 0}$ be the usual defects process, started at a point $X_{0} = \mathcal{R}(\sigma,\eta,m,k)$. Finally, let
\be
\tau_{1} &= \inf \{ s > 0 \, : \, X_{s} \notin \cup_{m \in [1:L]} \cup_{k \in [2:(L-1)]} \mathcal{R}(\sigma,\eta,m,k) \}\\
\tau_{2} &= \inf \{ s > 0 \, : \, U_{s} = L_{s} + 1 \text{ or } U_{s} = L_{s} + L -1\}.
\ee
Then, inspecting the generators of these processes, it is possible to couple $\{(X_{t},L_{t},U_{t})\}_{t \geq 0}$ so that
\be
X_{t} = \mathcal{R}(\sigma,\eta, L_{t}, U_{t} - L_{t})
\ee
for all $0 \leq t \leq \tau_{1} \leq \tau_{2}$.

This coupling gives a \textit{complete} description of the excursions of $\{X_{t}\}_{t \geq 0}$ in sets of the form $\{R(\sigma,\eta,m,k)\}_{k=2}^{L-2}$ for some fixed $\sigma, \eta \in \mathcal{G}$ and starting point $m \in [1:L]$. We use this coupling, and its immediate consequences, extensively in the remainder of this section to ``translate" well-known facts about simple random walk to our setting. Most of the proof simply consists of applying this coupling to excursions on the set $ \mathcal{R} \backslash (\mathcal{G} \cup \mathcal{R}_{1} \cup \mathcal{R}_{L-1})$.

The behaviour of the process outside of these excursions on $ \mathcal{R} \backslash (\mathcal{G} \cup \mathcal{R}_{1} \cup \mathcal{R}_{L-1})$ is fairly simple. Some excursions end by exiting $\mathcal{R}$, and we will see that this is quite rare. The remaining excursions will end at $\mathcal{R}_{1} \cup \mathcal{R}_{L-1} \cup \mathcal{G}$; since each element $\sigma \in \mathcal{R}_{1} \cup \mathcal{R}_{L-1}$ has at most $4$ neighbours $\eta$ with high transition probability $\cL(\sigma,\eta) = 1$ (again, see Figure \ref{FigDominantTransitionsBetweenGroundStates}), it is straightforward to analyze all transitions from within the set $\mathcal{R}_{1} \cup \mathcal{R}_{L-1} \cup \mathcal{G}$ that have non-negligible probability.

\end{remark}

\begin{proof}
Fix $u,v \in \{-1,+1\}^{2L-1}$ that differ at a single index $\ell$ and let $\sigma = w^{-1}(u)$, $\eta = w^{-1}(v)$. Assume for now that $1 \leq \ell \leq L$; the other cases will be essentially identical. Recalling notation from Definition \ref{DefBijectionPeriodic}, which will be used heavily throughout this proof, we assume without loss of generality that $u \prec v$.

We study $\{X_{t}\}_{t \geq 0}$ by charting its movement from a ground state, to an adjacent non-ground state with exactly 4 adjacent defects, to moving along minimal-energy paths between ground states; we stop keeping track of $\{X_{t}\}_{t \geq 0}$ as soon as it either returns to a ground state \textit{or} enters a state with 6 or more defects.


Let $\{X_{t}\}_{t \geq 0}$ be a copy of the plaquette dynamics started at $X_{0} = \sigma \in \cG$ and recall $\tau_{\rm start} = \inf \{t > 0 \, : \, X_{t} \neq \sigma\}$. By symmetry we have
\be \label{EqStartingExcursions}
\P_\s[X_{\tau_{\rm start}} \in \cup_{m=1}^{L} (R(\sigma,\xi,m,1) \cup R(\xi, \sigma, m,L-1))] = \frac{1}{L}
\ee
for all $\xi \in \mathcal{G}$ such that $|w(\sigma) - w(\xi)|= 1$ or $w(\s)=-w(\xi)$. Note that this means all other transition probabilities are 0.

Define the event $\mathcal{E}_{\rm start} = \{ X_{\tau_{\rm start}} \in \cup_{m=1}^{L} R(\sigma,\eta,m,1)\}$.  Next, let $\tau_{\rm exc}^{(1)} = \inf \{ t > \tau_{\rm start} \, : \, X_{t} \neq X_{\tau_{\rm start}}\}$. Then define a sequence of successive return and excursion times by the following formulae, with $i \in \mathbb{N}$:
\be
\tau_{\rm ret}^{(i)} &= \inf\{t > \tau_{\rm exc}^{(i)} \, : \, X_{t}  \notin \cup_{m = 1}^{L} \cup_{k=2}^{L-2} R(\sigma,\eta,m,k) \} \\
\tau_{\rm exc}^{(i+1)} &= \inf \{ t > \tau_{\rm ret}^{(i)} \, : \, X_{t} \neq X_{\tau_{\rm ret}^{(i)}}\}.
\ee
The reader may note that this formula looks slightly different from \textit{e.g.} \eqref{EqStartingExcursions} in that it contains only unions over $ R(\sigma,\eta,m,k)$, with none over $R(\eta,\sigma,m,k)$; this is because we have assumed that $\eta \in \Omega_{[1:L]^{2}}$ satisfies $w(\sigma) \prec w(\eta)$, so that \textit{e.g.} $R(\eta,\sigma,m,k) = \emptyset$ for all $m,k$.

We keep track of the behaviour of $\{X_{t}\}_{t \geq 0}$ in the excursions from $\cup_{m=1}^{L} (R(\sigma,\eta,m,1) \cup \mathcal{R}(\sigma,\eta,m,L-1))$ by using the following families of indicator functions:
\be
\phi^{(i)}_{\cG} = \1 \{X_{\tau_{\rm exc}^{(i)}} \in \cG \}, \,\, \phi^{(i)}_{\mathcal{R}} = \1 \{|D(X_{\tau_{\rm exc}^{(i)}})| = 4\}, \, \,\phi^{(i)}_{\Delta} = 1 - \phi^{(i)}_{\cG} - \phi^{(i)}_{\mathcal{R}}
\ee
and
\be
\psi^{(i)}_{1} = \1 \{X_{\tau_{\rm ret}^{(i)}} = X_{\tau_{\rm exc}^{(i)}} \}, \, \,
\psi^{(i)}_{\Delta} = \1 \{|D(X_{\tau_{\rm ret}^{(i)}})| > 4\}, \, \,
\psi^{(i)}_{L-1} = 1 - \psi^{(i)}_{0} - \psi^{(i)}_{\Delta}.
\ee
In the family associated with the letter $\phi$, the first indicates that the $i$'th excursion begins with $\{X_{t}\}_{t \geq 0}$ immediately returning to a ground state; the second indicates that the excursion begins by moving along a minimal-energy path between neighbouring ground states; the third indicates that something else has happened. In the family associated with the letter $\psi$, the first indicates that the $i$'th excursion ended where it started; the second indicates that the $i$'th excursion involves moving off of a minimal-energy path; the third indicates that the $i$'th excursion involves travelling along a minimal-energy path until a point that is adjacent to a different ground state.

Let $K_{\rm SRW}$ be the transition kernel on state space $[1:(L-1)] \cup \{\Delta\}$ with holding at $\{1,L-1,\Delta\}$ and other nonzero transitions rates:
\be
K_{\rm SRW}(i, i \pm 1) &= 2 \\
K_{\rm SRW}(i,\Delta) &= 8 e^{-2 \beta} + (L^{2}-12)e^{- 4 \beta}.
\ee
Let $\{Z_{t}\}_{t \geq 0}$ be a Markov process with this transition kernel started at 2, and let $\hat{\tau} = \inf\{t > 0 \, : \, Z_{t} \in \{1,L-1,\Delta\} \}$. By the discussion in Remark \ref{RemSRWCouplingExplicit} and simple counting of the states adjacent to $\cup_{k \in [2:(L-2)]} \cup_{m \in [1:L]} \mathcal{R}(\sigma,\eta,m,k)$ with 6 or more defects, on the event $\{\phi^{(i)}_{\mathcal{R}}  = 1\}$, we can couple $\{Z_{t}\}_{t \geq 0}$ to $\{X_{t}\}_{\tau_{\rm exc}^{(i)}}^{\tau_{\rm ret}^{(i)}}$ so that the following hold:
\be \label{SRW_Coupling_Properties}
\hat{\tau} &= \tau_{\rm ret}^{(i)} - \tau_{\rm exc}^{(i)} \\
\1 \{Z_{\hat{\tau}} = z \} &= \psi^{(i)}_{z}, \qquad \qquad z \in \{1,L-1,\Delta\}.
\ee

This correspondence allows us to analyze excursions using well-known properties of $K_{\rm SRW}$, which is simple random walk on the path with killing at constant rate. Recall the definition of $\tau_{\rm end}$ from Equality \eqref{EqDefZetaEndCleanup}, and let $I \in \mathbb{N}$ be the unique (random) integer satisfying
\be
\tau_{\rm end} \in \{\tau_{\rm exc}^{(I)}, \tau_{\rm ret}^{(I)} \};
\ee
that is, the first time that the walk either enters a ground state \textit{or} a state with 6 or more defects.

Combining Equalities \eqref{EqStartingExcursions} and \eqref{SRW_Coupling_Properties}, and using standard facts about hitting probabilities for simple random walk on the path, we have
\be
\P_{\sigma}[X_{\tau_{\rm end}} = \eta] &\gtrsim \P_{\sigma}[ \mathcal{E}_{\rm start}] \P_{\sigma}[X_{\tau_{\rm end}} = \eta | \mathcal{E}_{\rm start}] \\
&\gtrsim L^{-1} \sum_{i \in \mathbb{N}} \P_{\sigma}[X_{\tau_{\rm end}} = \eta | \mathcal{E}_{\rm start}, \, I = i] \P_{\sigma}[I=i | \mathcal{E}_{\rm start}] \\
&\gtrsim L^{-1} \sum_{i \in \mathbb{N}}  (L^{-1} - Le^{-2 \beta}) \P_{\sigma}[I=i | \mathcal{E}_{\rm start}] \approx L^{-2}.
\ee
By essentially the same calculation, we have
\be
\P_{\sigma}[|D(X_{\tau_{\rm end}} )| > 4] &\lesssim L^{2} e^{-2 \beta} \\
1- \P_{\sigma}[X_{\tau_{\rm end}} = \sigma] &\approx L^{-1}.
\ee
This completes the proof of \eqref{IneqExcursionBadProbs} in the case that $1 \leq \ell \leq L$ and $w(\sigma) \prec w(\eta)$. Combining these two inequalities, and using symmetry,
\be
\P_{\sigma}[X_{\tau_{\rm end}} = \eta | X_{\tau_{\rm end}} \neq \sigma] &\gtrsim L^{-1} \frac{\P_{\sigma}[X_{\tau_{\rm end}} = \eta]}{\P_{\sigma}[X_{\tau_{\rm end}} \in \cG \backslash \{\sigma\}] + \P[|D(X_{\tau_{\rm end}} )| > 4]} \\
&\gtrsim \frac{L^{-2}}{L^{-1} + L^{2} e^{-2 \beta}} \approx L^{-1}.
\ee
This completes the proof of Inequality \eqref{IneqLemmaHyperCompMainConc}, and thus the lemma, in the case that $1 \leq \ell \leq L$ and $w(\sigma) \prec w(\eta)$. The other cases are essentially identical by symmetry.

\end{proof}

This comparison implies the mixing bound:

\begin{cor} \label{LemmaTraceMixing}
The mixing time of the transition kernel $\frac{1}{2}(\mathrm{Id} + Q_{\mathcal{G}}^{\rm per})$ is $O(L \, \log(L)^{2})$.
\end{cor}
\begin{proof}
This follows immediately from the well-known fact that the mixing time of $\frac{1}{2}$-lazy simple random walk on $\{-1,+1\}^{2L-1}$ is $O( L \log(L))$, the comparison of $Q_{\mathcal{G}}^{\rm per}$ to the transition kernel for simple random walk on $\{-1,+1\}^{2L-1}$ given in Lemma \ref{LemmaHyperComp}, and Theorem 1 of \cite{kozma2007precision}.
\end{proof}

Fix $\sigma \in \Omega_{[1:L]^{2}}$ and let $\{X_{t}\}_{t \geq 0}$ be a copy of the original Markov process started at point $X_{0} = \sigma$. For $A \subset \Omega_{[1:L]^{2}}$, denote by
\be
\tau_{A} = \inf \{ t > 0 \, : \, X_{t} \in A \}
\ee
the hitting time of the set $A$.  Using the upper bound on the mixing time in Theorem \ref{ThmMainResPlus}, we have:

\begin{lemma} \label{LemmaHittingGround}
The hitting time $\tau_{\cG}$ satisfies
\be
\E_{\sigma}[\tau_{\cG}] \lesssim \beta^{9} e^{4 \beta},
\ee
uniformly in $\sigma \in \config$.
\end{lemma}
\begin{proof}

Denote by $\tau_{\mathrm{mix}, d}^{\mathrm{per}}$ and $\tau_{\mathrm{mix}, d}^{+}$ the mixing times of the defect dynamics given by the generators $\cQ_{[0:L]^{2}}^{(\mathrm{per})}$ and $\cQ_{[1:L]^{2}}^{+}$ respectively. Since the defect dynamics are a deterministic function of the spin dynamics, it is clear that
\be
\tau_{\mathrm{mix}, d}^{+} \leq \tau_{\mathrm{mix}}^{+}, \quad \tau_{\mathrm{mix}, d}^{\mathrm{per}} \leq \tau_{\mathrm{mix}}^{\mathrm{per}}.
\ee
By the tightness of the spectral profile (as stated in Theorem 1 of \cite{kozma2007precision}) and Inequality \eqref{LemmaCriBijection}, we also have
\be
\tau_{\mathrm{mix}, d}^{\mathrm{per}} \lesssim \beta \tau_{\mathrm{mix}, d}^{+}.
\ee
Combining these two bounds with Theorem \ref{ThmMainResPlus},
\be \label{IneqPerDefMixing}
\tau_{\mathrm{mix}, d}^{\mathrm{per}} \lesssim \beta \tau_{\mathrm{mix}, d}^{+} \lesssim \beta^{9} e^{4 \beta}.
\ee
Recalling that $\pi(\cG) = 1 - o(1)$ (by Lemma \ref{lem:dom}), Inequality \eqref{IneqPerDefMixing} and  Theorem 1 of \cite{Peres2015} immediately imply the result.

\end{proof}

For $\sigma \in \cG$, we have the improved bound on the expected hitting time:

\begin{lemma} \label{LemmaExpectedExcursionPerGr}
Let $\sigma \in \cG$. Then
\be
\E_{\sigma}[\tau_{\cG \backslash \{\sigma \}}] \lesssim \beta^{9} \,e^{3.5 \beta}.
\ee
\end{lemma}

\begin{proof}
Let $\{X_{t}\}_{t \geq 0}$ be a copy of the original Markov process with starting point $X_{0} = \sigma$. Define inductively the sequence of times
\be
t_{1} &= \inf \{t > 0 \, :\, X_{t} \neq \sigma\} \\
s_{i} &= \inf \{t > t_{i} \, : \, X_{t} \in \cG \} \\
t_{i+1} &= \inf \{t > s_{i} \, : \, X_{t} \neq X_{s_{i}} \}.
\ee
Denote by
\be
\mathcal{B}_{i} = \{ \max_{t_{i} \leq t \leq s_{i}} | D(X_{t})| > 4\}
\ee
the event that $\{X_{t}\}_{t \geq 0}$ escapes from the collection of minimal-energy paths during the $i$'th excursion. Define
\be
J = \min \{ j \in \mathbb{N} \, : \, X_{s_{j}} \neq \sigma\},
\ee
so that $\tau_{\cG \backslash \{\sigma\}} = s_{J}$. Thus, we have
\be \label{IneqExcursionMainAttempt}
\E_{\sigma}[\tau_{\cG \backslash \{\sigma\}}] &= \E_{\sigma}[s_{J}] \\
&=\E_{\sigma}[\sum_{i=1}^{J} (s_{i}-t_{i}) + \sum_{i=2}^{J} (t_{i} - s_{i-1}) + t_{1}] \\
&=\sum_{j = 1}^{\infty}(\sum_{i=1}^{j} \E_{\sigma}[ (s_{i}-t_{i}) | J=j] +  \sum_{i=2}^{j} \E_{\sigma}[(t_{i} - s_{i-1})| J=j]  +\E_{\sigma}[ t_{1}| J=j]) \P_{\sigma}[J = j].
\ee
We now bound these three types of terms, and the probability $\P_{\sigma}[J = j]$. By direct inspection of the transition probabilities, and the Markov property,
\be
\E_{\sigma}[t_{1} | J=j] \lesssim L^{-2} e^{4 \beta} \approx e^{3 \beta}.
\ee
By the same calculation,
\be \label{IneqInitialSquareIneq}
\E_{\sigma}[t_{i} - s_{i-1} | J=j] = \E_{\sigma}[t_{1} | J=j] \lesssim L^{-2} e^{4 \beta} \approx e^{3 \beta}
\ee
for $1 \leq i \leq j$. By Inequality \eqref{IneqExcursionBadProbs} ,
\be
\P_{\sigma}[J = 1] \gtrsim L^{-1},
\ee
and so by the Markov property
\be \label{IneqFinallyLeavingProb}
\P_{\sigma}[J >j] \leq (1 - C \, L^{-1})^{j}
\ee
for some fixed $0 < C < \infty$ and all $j \in \mathbb{N}$. Finally, we calculate the term $\E[s_{i}-t_{i} | J=j]$. By the coupling Equality \eqref{SRW_Coupling_Properties} from the proof of Lemma \ref{LemmaHyperComp} and the usual bound of $O(L^{2})$ on the first hitting time of $\{1,L-1\}$ for simple random walk on $[1:(L-1)]$,
\be \label{IneqSTdiff1}
\E_{\sigma}[(s_{i}-t_{i}) \1_{\mathcal{B}_{i}^{c}} | J=j] \lesssim L^{2}.
\ee
By Inequality \eqref{IneqExcursionBadProbs},
\be
\P_{\sigma}[\mathcal{B}_{i} | J=j] \lesssim L^{2} e^{-2 \beta} \approx e^{-\beta}.
\ee
Thus, also applying Lemma \ref{LemmaHittingGround} and then again Inequality \eqref{IneqExcursionBadProbs},
\be
\E_{\sigma}[(s_{i}-t_{i}) \1_{\mathcal{B}_{i}} | J=j] &= \E_{\sigma}[(s_{i}-t_{i})| J=j, \, \mathcal{B}_{i}] \P_{\sigma}[\mathcal{B}_{i} | J=j] \\
&\stackrel{\text{Lemma } \ref{LemmaHittingGround}}{\lesssim} \beta^{9} \, e^{4 \beta} \P_{\sigma}[\mathcal{B}_{i} | J=j] \\
&\stackrel{\text{Inequality } \eqref{IneqExcursionBadProbs}}{\lesssim} \beta^{9} \, e^{4 \beta} \, L^{2} \, e^{-2 \beta} \\
&\approx \beta^{9} e^{3 \beta}.
\ee
Combining this with Inequality \eqref{IneqSTdiff1}, we conclude
\be \label{IneqSTdiff2}
\E_{\sigma}[(s_{i}-t_{i}) | J=j] \lesssim \beta^{9} e^{3 \beta}.
\ee

Applying Inequalities \eqref{IneqInitialSquareIneq}, \eqref{IneqFinallyLeavingProb}, and \eqref{IneqSTdiff2} to Inequality \eqref{IneqExcursionMainAttempt}, we have:
\be
\E_{\sigma}[\tau_{\cG \backslash \{ \sigma \}}] &=\sum_{j = 1}^{\infty}(\sum_{i=1}^{j} \E_{\sigma}[ (s_{i}-t_{i}) | J=j] +  \sum_{i=2}^{j} \E_{\sigma}[(t_{i} - s_{i-1})| J=j]  +\E_{\sigma}[ t_{1}| J=j]) \P[J = j] \\
&\lesssim \sum_{j=1}^{\infty} j \, \beta^{9} \, e^{3 \beta} \P_{\sigma}[J=j] \\
&\lesssim \beta^{9} \, e^{3.5 \beta}.
\ee
This completes the proof.
\end{proof}

We now put these results together:

\begin{proof} [Proof of Upper Bound of Theorem \ref{ThmMainResPer}]

Fix a point $\sigma \in \Omega_{[1:L]^{2}}$ and a subset $A \subset \Omega_{[1:L]^{2}}$ with $\pi_{\L}^{\rm per}(A) \geq \frac{3}{4}$. We wish to bound the expected hitting time $\E_{\sigma}[\tau_{A}]$.

Assume for now that $\sigma \in \cG$. Let $\{\hat{X}_{i}\}_{i \in \mathbb{N}}$ be the trace of $\{X_{t}\}_{t \geq 0}$ on $\cG$, as in Definition \ref{DefTrace}.  Let
\be
\hat{\tau}_{A} = \min \{i \, : \, \hat{X}_{i} \in A\}.
\ee
By Lemma \ref{lem:dom}, we know that $\pi_{\L}^{\rm per}(A \cap \cG) \geq \frac{2}{3}$ for all $\beta > \beta_{0}$ sufficiently large. Thus, by Theorem 1 of \cite{Peres2015} and Corollary \ref{LemmaTraceMixing},
\be
\E_{\sigma}[\hat{\tau}_{A}] \lesssim L \, \log(L)^{2},
\ee
uniformly in the particular choice of $\sigma \in \cG$ and $A \subset \config$ with large measure. In particular, there exists a constant $0 < C_{1} < \infty$ that satisfies
\be \label{IneqHittingNearConcProj}
\P_{\sigma}[\hat{\tau}_{A} > C_{1} \, L \, \log(L)^{2}] \leq 0.01
\ee
uniformly in these choices.

Next, let $\{t_{i}\}_{i \in \mathbb{N}}$ be the random sequence of times given in Definition \ref{DefTrace} when constructing the trace $\{\hat{X}_{i}\}_{i \in \mathbb{N}}$. Define also the timescale $I = \lceil C_{1} \, L \, \log(L)^{2} + 1 \rceil$. By Lemma \ref{LemmaExpectedExcursionPerGr},
\be \label{IneqHittingNearConcExc}
\E[t_{i+1} - t_{i} | X_{t_{i}} = \eta] \lesssim \beta^{9} e^{3.5 \beta},
\ee
uniformly in the choice of $\eta \in \cG$. In particular, there exists a constant $0 < C_{2} < \infty$ that satisfies
\be
\P_{\sigma}[t_{I} > C_{2} \, \beta^{9} \, e^{3.5 \beta} \, L \, \log(L)^{2}] \leq 0.01
\ee
uniformly in the initial point $\sigma \in \cG$. Combining Inequalities \eqref{IneqHittingNearConcProj} and \eqref{IneqHittingNearConcExc}, we have uniformly in $\sigma \in \cG$ and $A \subset \Omega_{[1:L]^{2}}$ with sufficiently large measure
\be \label{IneqHittingFromGround}
\P_{\sigma}[\tau_{A} > C_{2} \,\beta^{9} \, e^{3.5 \beta} \, L \, \log(L)^{2}] &\leq \P_{\sigma}[\hat{\tau}_{A} > I] + \P_{\sigma}[t_{I} > C_{2} \, \beta^{9} \, e^{3.5 \beta} \, L \, \log(L)^{2}] \\
&\leq 0.02.
\ee

Next we extend this to general starting positions. By Lemma \ref{LemmaHittingGround}, there exists $0 < C_{3} < \infty$ so that, for all $\sigma \in \Omega_{[1:L]^{2}}$,
\be
\P_{\sigma}[\tau_{\cG} > C_{3} \beta^{9} e^{4 \beta}] \leq 0.01.
\ee
Combining this with Inequality \eqref{IneqHittingFromGround}, we conclude that for all $\sigma \in \Omega_{[1:L]^{2}}$,

\be
\P_{\sigma}[\tau_{A} &> C_{3} \beta^{9} e^{4 \beta} + C_{2} \, \beta^{9} e^{3.5 \beta} \, L \, \log(L)^{2}]  \leq \P_{\sigma}[\tau_{\cG} > C_{3} \beta^{9} e^{4 \beta}] + \sup_{\eta \in \cG} \P_{\eta}[\tau_{A} > C_{2}  \, \beta^{9} \, e^{3.5 \beta} \, L \, \log(L)^{2}]  \\
&\leq 0.01 + 0.02 = 0.03.
\ee
This immediately implies
\be
\max_{\sigma \in \Omega_{[1:L]^{2}}} \E_{\sigma}[\tau_{A}] \lesssim \beta^{9} e^{4 \beta}.
\ee
By Theorem 1 of \cite{Peres2015}, this implies the desired bound on the mixing time.

\end{proof}

\section{Acknowledgments}

We thank Alessandra Faggionato, Fabio Martinelli and Cristina Toninelli for their help, hospitality and encouragement. Several of the basic observations (e.g. Equation \eqref{EqCristinaBijection}) are theirs, and much of the initial work on this project, including the lower bound in Section \ref{sec:lwrbnd}, was completed during visits hosted by Cristina Toninelli and Fabio Martinelli. Of course, all errors introduced in the implementation are our own.

\bibliographystyle{abbrv}
\bibliography{PlaqBib}

\begin{thebibliography}{10}

\bibitem{Berthier2011a}
L.~Berthier and G.~Biroli.
\newblock {Theoretical perspective on the glass transition and amorphous
  materials}.
\newblock {\em Reviews of Modern Physics}, 83(2):587--645, 2011.

\bibitem{Blondel2012a}
O.~Blondel, N.~Cancrini, F.~Martinelli, C.~Roberto, and C.~Toninelli.
\newblock {Fredrickson-Andersen one spin facilitated model out of equilibrium}.
\newblock {\em Markov Process. Relat. Fields}, 19:383--406, 2013.

\bibitem{Chamon2005}
C.~Chamon.
\newblock {Quantum glassiness in strongly correlated clean systems: An example
  of topological overprotection}.
\newblock {\em Phys. Rev. Lett.}, 94(4):1--4, 2005.

\bibitem{Chleboun2012}
P.~Chleboun, A.~Faggionato, and F.~Martinelli.
\newblock {Time Scale Separation and Dynamic Heterogeneity in the Low
  Temperature East Model}.
\newblock {\em Communications in Mathematical Physics}, 328(3):955--993, 2014.

\bibitem{Chleboun2015}
P.~Chleboun, A.~Faggionato, and F.~Martinelli.
\newblock {Mixing time and local exponential ergodicity of the {E}ast-like
  process in $\mathbb{Z}^{d}$}.
\newblock {\em Annales de la facult{\'{e}} des sciences de Toulouse
  Math{\'{e}}matiques}, XXIV(4):717--743, 2015.

\bibitem{Chleboun2014}
P.~Chleboun, A.~Faggionato, and F.~Martinelli.
\newblock {Relaxation to equilibrium of generalized {E}ast processes on
  $\mathbb{Z}^{d}$: Renormalization group analysis and energy-entropy
  competition}.
\newblock {\em Ann. Probab.}, 44(3):1817--1863, 2016.

\bibitem{PlaquetteBigSquareLattice}
P.~Chleboun, A.~Faggionato, F.~Martinelli, A.~Smith, and C.~Toninelli.
\newblock Relaxation time of plaquette spin models.
\newblock {\em Still being written}, 2019.

\bibitem{Chleboun2017}
P.~Chleboun, A.~Faggionato, F.~Martinelli, and C.~Toninelli.
\newblock {Mixing Length Scales of Low Temperature Spin Plaquettes Models}.
\newblock {\em Journal of Statistical Physics}, 169(3):441--471, 2017.

\bibitem{SPMcutoff}
P.~Chleboun and A.~Smith.
\newblock Cutoff for the square plaquette model on a critical length scale.
\newblock {\em In preparation}, 2019.

\bibitem{Faggionato2011}
A.~Faggionato, F.~Martinelli, C.~Roberto, and C.~Toninelli.
\newblock {Aging Through Hierarchical Coalescence in the East Model}.
\newblock {\em Communications in Mathematical Physics}, 309(2):459--495, 2011.

\bibitem{Faggionato2012}
A.~Faggionato, F.~Martinelli, C.~Roberto, and C.~Toninelli.
\newblock {The East model: recent results and new progresses}.
\newblock {\em Markov Process. Related Fields}, 3(19):407--452, 2013.

\bibitem{Ganguly2015}
S.~Ganguly, E.~Lubetzky, and F.~Martinelli.
\newblock {Cutoff for the East Process}.
\newblock {\em Commun. Math. Phys.}, 335(3):1287--1322, 2015.

\bibitem{Garrahan2002b}
J.~P. Garrahan.
\newblock {Glassiness through the emergence of effective dynamical constraints
  in interacting systems}.
\newblock {\em Journal of Physics: Condensed Matter}, 14(7):1571--1579, 2002.

\bibitem{Garrahan2010}
J.~P. Garrahan, P.~Sollich, and C.~Toninelli.
\newblock {Kinetically Constrained Models}.
\newblock In {\em Dynamical heterogeneities in glasses, colloids, and granular
  media}, chapter Kineticall. OUP, 2011.

\bibitem{Goel2006}
S.~Goel, R.~Montenegro, and P.~Tetali.
\newblock {Mixing time bounds via the spectral profile}.
\newblock {\em Electronic Journal of Probability}, 11(2000):1--26, 2006.

\bibitem{Jack2005a}
R.~L. Jack, L.~Berthier, and J.~P. Garrahan.
\newblock {Static and dynamic length scales in a simple glassy plaquette
  model}.
\newblock {\em Physical Review E}, 72(1):016103, 2005.

\bibitem{Jack2016}
R.~L. Jack and J.~P. Garrahan.
\newblock Phase transition for quenched coupled replicas in a plaquette spin
  model of glasses.
\newblock {\em Phys. Rev. Lett.}, 116:055702, 2016.

\bibitem{kozma2007precision}
G.~Kozma.
\newblock On the precision of the spectral profile.
\newblock {\em Latin American Journal of Probability and Mathematical
  Statistics}, 3:321--329, 2007.

\bibitem{Nandkishore2018}
R.~M. Nandkishore and M.~Hermele.
\newblock {Fractons}.
\newblock {\em arXiv preprint arXiv:1803.11196}, pages 1--30, 2018.

\bibitem{Newman1999}
M.~E. Newman and C.~Moore.
\newblock {Glassy dynamics and aging in an exactly solvable spin model.}
\newblock {\em Phys. Rev. E}, 60(5):5068--5072, 1999.

\bibitem{Oliveira2012}
R.~I. Oliveira.
\newblock {Mixing and hitting times for finite Markov chains}.
\newblock {\em Electronic Journal of Probability}, 17(May 2010):1--12, 2012.

\bibitem{Peres2015}
Y.~Peres and P.~Sousi.
\newblock Mixing times are hitting times of large sets.
\newblock {\em Journal of Theoretical Probability}, 28(2):488--519, 2015.

\bibitem{Pillai2017}
N.~S. Pillai and A.~Smith.
\newblock {Mixing times for a constrained Ising process on the torus at low
  density}.
\newblock {\em Annals of Probability}, 45(2):1003--1070, 2017.

\bibitem{Pillai2017a}
N.~S. Pillai and A.~Smith.
\newblock Mixing times for a constrained ising process on the two-dimensional
  torus at low density.
\newblock {\em Annales de l'Institut Henri Poincar{\'e}, Probabilit{\'e}s et
  Statistiques}, 55(3):1649--1678, 2019.

\bibitem{Sinclair1992}
A.~Sinclair.
\newblock {Improved Bounds For Mixing Rates Of Markov Chains And Multicommodity
  Flow}.
\newblock {\em Combinatorics, Probability {\&} Computing}, 1(April
  1992):351--370, 1992.

\end{thebibliography}

\end{document}